\documentclass[11pt]{amsart}

\usepackage{amsmath}
\usepackage{amssymb}
\usepackage{amscd}
\usepackage[table]{xcolor}
\usepackage{hyperref}
\usepackage{mathrsfs}
\usepackage{eucal}
\usepackage{upgreek}
\usepackage[makeroom]{cancel}
\usepackage[normalem]{ulem}
\usepackage{array}
\usepackage{verbatim}
\usepackage{mathtools}

\usepackage{enumitem}

\usepackage{xy}
\xyoption{all}
\usepackage{tikz-cd}

\newcommand{\nc}{\newcommand}

\nc{\xrightiso}[1]{ \xrightarrow[{\ \raisebox{0.5ex}[0ex][0ex]{$\sim$}\ }]{#1} }

\renewcommand{\AA}{{\mathbb{A}}}
\nc{\CC}{{\mathbb{C}}}
\nc{\DD}{{\mathbb{D}}}
\nc{\LL}{{\mathbb{L}}}
\nc{\RR}{{\mathbb{R}}}
\renewcommand{\P}{{\mathbb{P}}}
\nc{\OO}{{\mathbb{O}}}

\nc{\QQ}{{\mathbb{Q}}}
\nc{\ZZ}{{\mathbb{Z}}}
\nc{\Z}{{\mathbb{Z}}}

\nc{\cA}{{\mathcal{A}}}
\nc{\cB}{{\mathcal{B}}}
\nc{\cC}{{\mathcal{C}}}
\nc{\cD}{{\mathcal{D}}}
\nc{\cE}{{\mathcal{E}}}
\nc{\cF}{{\mathcal{F}}}
\nc{\cG}{{\mathcal{G}}}
\nc{\cH}{{\mathcal{H}}}
\nc{\cI}{{\mathcal{I}}}
\nc{\cJ}{{\mathcal{J}}}
\nc{\cK}{{\mathcal{K}}}
\nc{\cL}{{\mathcal{L}}}
\nc{\cM}{{\mathcal{M}}}
\nc{\cN}{{\mathcal{N}}}
\nc{\cO}{{\mathcal{O}}}
\nc{\cP}{{\mathcal{P}}}
\nc{\cQ}{{\mathcal{Q}}}
\nc{\cR}{{\mathcal{R}}}
\nc{\cS}{{\mathcal{S}}}
\nc{\cT}{{\mathcal{T}}}
\nc{\cU}{{\mathcal{U}}}
\nc{\cV}{{\mathcal{V}}}
\nc{\cW}{{\mathcal{W}}}
\nc{\cX}{{\mathcal{X}}}
\nc{\cY}{{\mathcal{Y}}}
\nc{\cZ}{{\mathcal{Z}}}

\nc{\rb}{{\mathbf{b}}}
\nc{\rc}{{\mathrm{c}}}
\nc{\rd}{{\mathrm{d}}}
\nc{\rr}{{\mathrm{r}}}
\nc{\rf}{F}
\nc{\rh}{{\mathrm{h}}}
\nc{\rch}{{\mathrm{ch}}}
\nc{\rtd}{{\mathrm{td}}}

\nc{\rA}{{\mathrm{A}}}
\nc{\rB}{{\mathrm{B}}}
\nc{\rC}{{\mathrm{C}}}
\nc{\rD}{{\mathrm{D}}}
\nc{\rE}{{\mathrm{E}}}
\nc{\rF}{{\mathrm{F}}}
\nc{\rG}{{\mathrm{G}}}
\nc{\rH}{{\mathrm{H}}}
\nc{\rI}{{\mathrm{I}}}
\nc{\rJ}{{\mathrm{J}}}
\nc{\rK}{{\mathrm{K}}}
\nc{\rL}{{\mathrm{L}}}
\nc{\rM}{{\mathrm{M}}}
\nc{\rN}{{\mathrm{N}}}
\nc{\rO}{{\mathrm{O}}}
\nc{\rP}{{\mathrm{P}}}
\nc{\rQ}{{\mathrm{Q}}}
\nc{\rR}{{\mathrm{R}}}
\nc{\rS}{{\mathrm{S}}}
\nc{\rT}{{\mathrm{T}}}
\nc{\rU}{{\mathrm{U}}}
\nc{\rV}{{\mathrm{V}}}
\nc{\rW}{{\mathrm{W}}}
\nc{\rX}{{\mathrm{X}}}
\nc{\rY}{{\mathrm{Y}}}
\nc{\rZ}{{\mathrm{Z}}}

\nc{\bA}{{\mathbf{A}}}
\nc{\bB}{{\mathbf{B}}}
\nc{\bC}{{\mathbf{C}}}
\nc{\bD}{{\mathbf{D}}}
\nc{\bE}{{\mathbf{E}}}
\nc{\bF}{{\mathbf{F}}}
\nc{\bG}{{\mathbf{G}}}
\nc{\bH}{{\mathbf{H}}}
\nc{\bI}{{\mathbf{I}}}
\nc{\bJ}{{\mathbf{J}}}
\nc{\bK}{{\mathbf{K}}}
\nc{\bL}{{\mathbf{L}}}
\nc{\bM}{{\mathbf{M}}}
\nc{\bN}{{\mathbf{N}}}
\nc{\bO}{{\mathbf{O}}}
\nc{\bP}{{\mathbf{P}}}
\nc{\bQ}{{\mathbf{Q}}}
\nc{\bR}{{\mathbf{R}}}
\nc{\bS}{{\mathbf{S}}}
\nc{\bT}{{\mathbf{T}}}
\nc{\bU}{{\mathbf{U}}}
\nc{\bV}{{\mathbf{V}}}
\nc{\bW}{{\mathbf{W}}}
\nc{\bX}{{\mathbf{X}}}
\nc{\bY}{{\mathbf{Y}}}
\nc{\bZ}{{\mathbf{Z}}}

\nc{\ba}{{\mathbf{a}}}
\nc{\bb}{{\mathbf{b}}}
\nc{\bc}{{\mathbf{c}}}
\nc{\bd}{{\mathbf{d}}}
\nc{\be}{{\mathbf{e}}}
\nc{\bg}{{\mathbf{g}}}
\nc{\bh}{{\mathbf{h}}}
\nc{\bi}{{\mathbf{i}}}
\nc{\bj}{{\mathbf{j}}}
\nc{\bk}{{\mathbf{k}}}
\nc{\bl}{{\mathbf{l}}}
\nc{\bm}{{\mathbf{m}}}
\nc{\bn}{{\mathbf{n}}}
\nc{\bo}{{\mathbf{o}}}
\nc{\bp}{{\mathbf{p}}}
\nc{\bq}{{\mathbf{q}}}
\nc{\br}{{\mathbf{r}}}
\nc{\bs}{{\mathbf{s}}}
\nc{\bt}{{\mathbf{t}}}
\nc{\bu}{{\mathbf{u}}}
\nc{\bv}{{\mathbf{v}}}
\nc{\bw}{{\mathbf{w}}}
\nc{\bx}{{\mathbf{x}}}
\nc{\by}{{\mathbf{y}}}
\nc{\bz}{{\mathbf{z}}}

\nc{\fA}{{\mathfrak{A}}}
\nc{\fB}{{\mathfrak{B}}}
\nc{\fC}{{\mathfrak{C}}}
\nc{\fD}{{\mathfrak{D}}}
\nc{\fE}{{\mathfrak{E}}}
\nc{\fF}{{\mathfrak{F}}}
\nc{\fG}{{\mathfrak{G}}}
\nc{\fH}{{\mathfrak{H}}}
\nc{\fI}{{\mathfrak{I}}}
\nc{\fJ}{{\mathfrak{J}}}
\nc{\fK}{{\mathfrak{K}}}
\nc{\fL}{{\mathfrak{L}}}
\nc{\fM}{{\mathfrak{M}}}
\nc{\fN}{{\mathfrak{N}}}
\nc{\fO}{{\mathfrak{O}}}
\nc{\fP}{{\mathfrak{P}}}
\nc{\fQ}{{\mathfrak{Q}}}
\nc{\fR}{{\mathfrak{R}}}
\nc{\fS}{{\mathfrak{S}}}
\nc{\fT}{{\mathfrak{T}}}
\nc{\fU}{{\mathfrak{U}}}
\nc{\fV}{{\mathfrak{V}}}
\nc{\fW}{{\mathfrak{W}}}
\nc{\fX}{{\mathfrak{X}}}
\nc{\fY}{{\mathfrak{Y}}}
\nc{\fZ}{{\mathfrak{Z}}}

\nc{\fa}{{\mathfrak{a}}}
\nc{\fb}{{\mathfrak{b}}}
\nc{\fc}{{\mathfrak{c}}}
\nc{\fd}{{\mathfrak{d}}}
\nc{\fe}{{\mathfrak{e}}}
\nc{\ff}{{\mathfrak{f}}}
\nc{\fg}{{\mathfrak{g}}}
\nc{\fh}{{\mathfrak{h}}}
\nc{\fj}{{\mathfrak{j}}}
\nc{\fk}{{\mathfrak{k}}}
\nc{\fl}{{\mathfrak{l}}}
\nc{\fm}{{\mathfrak{m}}}
\nc{\fn}{{\mathfrak{n}}}
\nc{\fo}{{\mathfrak{o}}}
\nc{\fp}{{\mathfrak{p}}}
\nc{\fq}{{\mathfrak{q}}}
\nc{\fr}{{\mathfrak{r}}}
\nc{\fs}{{\mathfrak{s}}}
\nc{\ft}{{\mathfrak{t}}}
\nc{\fu}{{\mathfrak{u}}}
\nc{\fv}{{\mathfrak{v}}}
\nc{\fw}{{\mathfrak{w}}}
\nc{\fx}{{\mathfrak{x}}}
\nc{\fy}{{\mathfrak{y}}}
\nc{\fz}{{\mathfrak{z}}}

\nc{\bfa}{{\bar{\mathfrak{a}}}}
\nc{\tfb}{{\tilde{\mathfrak{b}}}}
\nc{\bfc}{{\bar{\mathfrak{c}}}}
\nc{\pfb}{{\partial{\mathfrak{b}}}}

\nc{\sA}{{\mathsf{A}}}
\nc{\sB}{{\mathsf{B}}}
\nc{\sC}{{\mathsf{C}}}
\nc{\sD}{{\mathsf{D}}}
\nc{\sE}{{\mathsf{E}}}
\nc{\sF}{{\mathsf{F}}}
\nc{\sG}{{\mathsf{G}}}
\nc{\sH}{{\mathsf{H}}}
\nc{\sI}{{\mathsf{I}}}
\nc{\sJ}{{\mathsf{J}}}
\nc{\sK}{{\mathsf{K}}}
\nc{\sL}{{\mathsf{L}}}
\nc{\sM}{{\mathsf{M}}}
\nc{\sN}{{\mathsf{N}}}
\nc{\sO}{{\mathsf{O}}}
\nc{\sP}{{\mathsf{P}}}
\nc{\sQ}{{\mathsf{Q}}}
\nc{\sR}{{\mathsf{R}}}
\nc{\sS}{{\mathsf{S}}}
\nc{\sT}{{\mathsf{T}}}
\nc{\sU}{{\mathsf{U}}}
\nc{\sV}{{\mathsf{V}}}
\nc{\sW}{{\mathsf{W}}}
\nc{\sX}{{\mathsf{X}}}
\nc{\sY}{{\mathsf{Y}}}
\nc{\sZ}{{\mathsf{Z}}}

\nc{\sa}{{\mathsf{a}}}
\nc{\sd}{{\mathsf{d}}}
\nc{\se}{{\mathsf{e}}}
\nc{\sg}{{\mathsf{g}}}
\nc{\sh}{{\mathsf{h}}}
\nc{\si}{{\mathsf{i}}}
\nc{\sj}{{\mathsf{j}}}
\nc{\sk}{{\mathsf{k}}}
\nc{\sm}{{\mathsf{m}}}
\nc{\sn}{{\mathsf{n}}}
\nc{\so}{{\mathsf{o}}}
\nc{\sq}{{\mathsf{q}}}
\nc{\sr}{{\mathsf{r}}}
\nc{\st}{{\mathsf{t}}}
\nc{\su}{{\mathsf{u}}}
\nc{\sv}{{\mathsf{v}}}
\nc{\sw}{{\mathsf{w}}}
\nc{\sx}{{\mathsf{x}}}
\nc{\sy}{{\mathsf{y}}}
\nc{\sz}{{\mathsf{z}}}

\nc{\oA}{{\overline{A}}}
\nc{\oB}{{\overline{B}}}
\nc{\oC}{{\overline{C}}}
\nc{\oD}{{\overline{D}}}
\nc{\oE}{{\overline{E}}}
\nc{\oF}{{\overline{F}}}
\nc{\oG}{{\overline{G}}}
\nc{\oH}{{\overline{H}}}
\nc{\oI}{{\overline{I}}}
\nc{\oJ}{{\overline{J}}}
\nc{\oK}{{\overline{K}}}
\nc{\oL}{{\overline{L}}}
\nc{\oM}{{\overline{M}}}
\nc{\oN}{{\overline{N}}}
\nc{\oO}{{\overline{O}}}
\nc{\oP}{{\overline{P}}}
\nc{\oQ}{{\overline{Q}}}
\nc{\oR}{{\overline{R}}}
\nc{\oS}{{\overline{S}}}
\nc{\oT}{{\overline{T}}}
\nc{\oU}{{\overline{U}}}
\nc{\oV}{{\overline{V}}}
\nc{\oW}{{\overline{W}}}
\nc{\oX}{{\overline{X}}}
\nc{\oY}{{\overline{Y}}}
\nc{\oZ}{{\overline{Z}}}

\nc{\oa}{{\overline{a}}}
\nc{\ob}{{\overline{b}}}
\nc{\oc}{{\overline{c}}}
\nc{\od}{{\overline{d}}}
\nc{\of}{{\overline{f}}}
\nc{\og}{{\overline{g}}}
\nc{\oh}{{\overline{h}}}
\nc{\oi}{{\overline{i}}}
\nc{\oj}{{\overline{j}}}
\nc{\ok}{{\overline{k}}}
\nc{\ol}{{\overline{l}}}
\nc{\om}{{\overline{m}}}
\nc{\on}{{\overline{n}}}
\nc{\oo}{{\overline{o}}}
\nc{\op}{{\overline{p}}}
\nc{\oq}{{\overline{q}}}
\nc{\os}{{\overline{s}}}
\nc{\ot}{{\overline{t}}}
\nc{\ou}{{\overline{u}}}
\nc{\ov}{{\overline{v}}}
\nc{\ow}{{\overline{w}}}
\nc{\ox}{{\overline{x}}}
\nc{\oy}{{\overline{y}}}
\nc{\oz}{{\overline{z}}}

\nc{\tA}{{\tilde{A}}}
\nc{\tB}{{\tilde{B}}}
\nc{\tC}{{\tilde{C}}}
\nc{\tD}{{\tilde{D}}}
\nc{\tE}{{\tilde{E}}}
\nc{\tF}{{\tilde{F}}}
\nc{\tG}{{\tilde{G}}}
\nc{\tH}{{\tilde{H}}}
\nc{\tI}{{\tilde{I}}}
\nc{\tJ}{{\tilde{J}}}
\nc{\tK}{{\tilde{K}}}
\nc{\tL}{{\tilde{L}}}
\nc{\tM}{{\tilde{M}}}
\nc{\tN}{{\tilde{N}}}
\nc{\tO}{{\tilde{O}}}
\nc{\tP}{{\tilde{P}}}
\nc{\tQ}{{\tilde{Q}}}
\nc{\tR}{{\tilde{R}}}
\nc{\tS}{{\tilde{S}}}
\nc{\tT}{{\tilde{T}}}
\nc{\tU}{{\tilde{U}}}
\nc{\tV}{{\tilde{V}}}
\nc{\tW}{{\tilde{W}}}
\nc{\tX}{{\tilde{X}}}
\nc{\tY}{{\tilde{Y}}}
\nc{\tZ}{{\tilde{Z}}}

\nc{\tfD}{{\tilde{\fD}}}
\nc{\tcA}{{\tilde{\cA}}}
\nc{\tcB}{{\tilde{\cB}}}
\nc{\tcC}{{\tilde{\cC}}}
\nc{\tcD}{{\tilde{\cD}}}
\nc{\tcF}{{\tilde{\cF}}}
\nc{\tcM}{{\tilde{\cM}}}
\nc{\tcT}{{\tilde{\cT}}}

\nc{\tcL}{{\tilde{\cL}}}
\nc{\tcX}{{\tilde{\cX}}}
\nc{\hcX}{{\hat{\cX}}}
\nc{\tcY}{{\tilde{\cY}}}

\nc{\hrU}{{\widehat{\rU}}}
\nc{\hrS}{{\widehat{\rS}}}

\nc{\bcA}{{\bar{\cA}}}

\nc{\ta}{{\tilde{a}}}
\nc{\tb}{{\tilde{b}}}
\nc{\tc}{{\tilde{c}}}
\nc{\td}{{\tilde{d}}}
\nc{\te}{{\tilde{e}}}
\nc{\tf}{{\tilde{f}}}
\nc{\tg}{{\tilde{g}}}
\nc{\ti}{{\tilde{\imath}}}
\nc{\tj}{{\tilde{j}}}
\nc{\tk}{{\tilde{k}}}
\nc{\tl}{{\tilde{l}}}
\nc{\tm}{{\tilde{m}}}
\nc{\tn}{{\tilde{n}}}
\nc{\tp}{{\tilde{p}}}
\nc{\tq}{{\tilde{q}}}
\nc{\tr}{{\tilde{r}}}
\nc{\ts}{{\tilde{s}}}
\nc{\tu}{{\tilde{u}}}
\nc{\tv}{{\tilde{v}}}
\nc{\tw}{{\tilde{w}}}
\nc{\tx}{{\tilde{x}}}
\nc{\ty}{{\tilde{y}}}
\nc{\tz}{{\tilde{z}}}

\nc{\tell}{{\tilde{\ell}}}
\nc{\hell}{{\hat{\ell}}}

\nc{\hA}{{\hat{A}}}
\nc{\hB}{{\hat{B}}}
\nc{\hC}{{\hat{C}}}
\nc{\hD}{{\hat{D}}}
\nc{\hE}{{\hat{E}}}
\nc{\hF}{{\hat{F}}}
\nc{\hG}{{\hat{G}}}
\nc{\hH}{{\hat{H}}}
\nc{\hI}{{\hat{I}}}
\nc{\hJ}{{\hat{J}}}
\nc{\hK}{{\hat{K}}}
\nc{\hL}{{\hat{L}}}
\nc{\hM}{{\hat{M}}}
\nc{\hN}{{\hat{N}}}
\nc{\hO}{{\hat{O}}}
\nc{\hP}{{\hat{P}}}
\nc{\hQ}{{\hat{Q}}}
\nc{\hR}{{\hat{R}}}
\nc{\hS}{{\widehat{S}}}
\nc{\hT}{{\hat{T}}}
\nc{\hU}{{\hat{U}}}
\nc{\hV}{{\hat{V}}}
\nc{\hW}{{\hat{W}}}
\nc{\hX}{{\widehat{X}}}
\nc{\hY}{{\hat{Y}}}
\nc{\hZ}{{\widehat{Z}}}

\nc{\ha}{{\hat{a}}}
\nc{\hb}{{\hat{b}}}
\nc{\hc}{{\hat{c}}}
\nc{\hd}{{\hat{d}}}
\nc{\he}{{\hat{e}}}
\nc{\hg}{{\hat{g}}}
\nc{\hh}{{\hat{h}}}
\nc{\hi}{{\hat{i}}}
\nc{\hj}{{\hat{j}}}
\nc{\hk}{{\hat{k}}}
\nc{\hl}{{\hat{l}}}
\nc{\hm}{{\hat{m}}}
\nc{\hn}{{\hat{n}}}
\nc{\ho}{{\hat{o}}}
\nc{\hp}{{\hat{p}}}
\nc{\hq}{{\hat{q}}}
\nc{\hr}{{\hat{r}}}
\nc{\hs}{{\hat{s}}}
\nc{\hu}{{\hat{u}}}
\nc{\hv}{{\hat{v}}}
\nc{\hw}{{\hat{w}}}
\nc{\hx}{{\hat{x}}}
\nc{\hy}{{\hat{y}}}
\nc{\hz}{{\hat{z}}}

\nc{\hcC}{{\widehat{\cC}}}
\nc{\hcT}{{\widehat{\cT}}}

\nc{\eps}{\varepsilon}
\nc{\lan}{\big\langle}
\nc{\ran}{\big\rangle}
\nc{\kk}{{\Bbbk}}
\nc{\io}{\upiota}
\nc{\Kr}{\mathsf{Kr}}
\nc{\cKr}{\mathcal{K}\!\mathit{r}}

\nc{\Dm}{\bD^{-}}
\nc{\Db}{\bD^{\mathrm{b}}}
\nc{\Dp}{\bD^{\mathrm{perf}}}
\nc{\Dperf}{\bD^{\mathrm{perf}}}
\nc{\Dqc}{\bD_{\mathrm{qc}}}
\nc{\Du}{\bD}
\nc{\Dsing}{\bD^{\mathrm{sg}}}
\nc{\Dg}{\bD^{\mathrm{sg}}}

\newcommand{\can}{\mathrm{can}}

\nc{\Rn}{\rR_{\mathrm{node}}}
\nc{\Cn}{\cC_{\mathrm{node}}}

\nc{\prU}{\partial\rU}

\def\bw#1#2{\textstyle{\bigwedge\hskip-0.9mm^{#1}}\hskip0.2mm{#2}}

\DeclareMathOperator{\Hom}{\mathrm{Hom}}

\DeclareMathOperator{\Ext}{\mathrm{Ext}}

\DeclareMathOperator{\Spec}{\mathrm{Spec}}

\DeclareMathOperator{\Coh}{\mathrm{Coh}}

\DeclareMathOperator{\Bl}{\mathrm{Bl}}
\DeclareMathOperator{\Sing}{\mathrm{Sing}}

\DeclareMathOperator{\VB}{\mathrm{VB}}
\DeclareMathOperator{\Pic}{\mathrm{Pic}}
\DeclareMathOperator{\rKt}{\mathrm{K}_1^{\mathrm{top}}}
\DeclareMathOperator{\rKn}{\mathrm{K}_0^{\mathrm{num}}}
\DeclareMathOperator{\bPic}{\mathbf{Pic}}
\DeclareMathOperator{\Alb}{\mathrm{Alb}}
\DeclareMathOperator{\Jac}{\mathrm{Jac}}

\DeclareMathOperator{\Cl}{\mathrm{Cl}}

\DeclareMathOperator{\CH}{\mathrm{CH}}
\DeclareMathOperator{\Bs}{\mathrm{Bs}}

\nc{\Dfd}[1]{\bD_{\mathrm{fd}}(#1)}

\DeclareMathOperator{\Sym}{\mathrm{Sym}}

\DeclareMathOperator{\Cone}{\mathrm{Cone}}

\DeclareMathOperator{\Gr}{\mathrm{Gr}}
\DeclareMathOperator{\CGr}{\mathrm{CGr}}

\DeclareMathOperator{\PGL}{\mathrm{PGL}}

\DeclareMathOperator{\spe}{\mathbf{sp}}

\DeclareMathOperator{\rank}{\mathrm{rk}}

\DeclareMathOperator{\Br}{\mathrm{Br}}

\DeclareMathOperator{\Qu}{\mathsf{Qu}}
\DeclareMathOperator{\Sch}{\mathrm{Sch}}

\DeclareMathOperator{\cMuk}{\mathcal{M}\!\mathit{u}}

\DeclareMathOperator{\MF}{\mathrm{MF}}

\DeclareMathOperator{\tMF}{\mathrm{MFC}}
\DeclareMathOperator{\MT}{\mathrm{MTrCat}}
\DeclareMathOperator{\brM}{\overline{\mathrm{M}}}

\DeclareMathOperator{\bMF}{\overline{\mathrm{MF}}}
\DeclareMathOperator{\bMFM}{\overline{\mathrm{MFM}}}
\DeclareMathOperator{\MFM}{\mathrm{MFM}}

\DeclareMathOperator{\MQ}{\mathrm{MQ}}
\DeclareMathOperator{\bMQ}{\overline{\mathrm{MQ}}}
\DeclareMathOperator{\MQS}{\mathrm{MQS}}

\newcommand{\pmu}{\mathrm{p}_{\mathrm{Mu}}}

\nc{\tor}{{\mathrm{tor}}}

\def\Pinfty#1{\P^{\infty,{#1}}}

\def\d{\mathrm{d}}
\def\g{\mathrm{g}}

\theoremstyle{plain}

\newtheorem{theorem}{Theorem}[section]
\newtheorem{conjecture}[theorem]{Conjecture}

\newtheorem{problem}[theorem]{Problem}
\newtheorem{lemma}[theorem]{Lemma}
\newtheorem{proposition}[theorem]{Proposition}
\newtheorem{corollary}[theorem]{Corollary}

\theoremstyle{definition}

\newtheorem{definition}[theorem]{Definition}

\theoremstyle{remark}

\newtheorem{remark}[theorem]{Remark}
\newtheorem{setup}[theorem]{Setup}

\newenvironment{renumerate}{\begin{enumerate}[label={\textup{(\roman*)}}]}{\end{enumerate}}

\newenvironment{alenumerate}{\begin{enumerate}[label={\textup{(\alph*)}}]}{\end{enumerate}}

\title{Derived categories of Fano threefolds and degenerations}

\author{Alexander Kuznetsov}
\address{{\sloppy
\parbox{0.9\textwidth}{
Algebraic Geometry Section, Steklov Mathematical Institute of Russian Academy of Sciences,\\
8 Gubkin str., Moscow 119991 Russia
\hfill\\[5pt]
Laboratory of Algebraic Geometry, National Research University Higher School of Economics, Russian Federation
}\bigskip}}
\email{akuznet@mi-ras.ru}
\date{}

\author{Evgeny Shinder}
\address{{\sloppy
\parbox{0.9\textwidth}{
School of Mathematics and Statistics, University of Sheffield,
Hounsfield Road, S3 7RH, UK, and
Hausdorff Center for Mathematics
at the University of Bonn, Endenicher Allee 60, 53115.
}\bigskip}}
\email{eugene.shinder@gmail.com}

\thanks{A.K. was partially supported by the HSE University Basic Research Program.
E.S. was partially supported by the EPSRC grant
EP/T019379/1 ``Derived categories and algebraic K-theory of singularities'', and by the
ERC Synergy grant ``Modern Aspects of Geometry: Categories, Cycles and Cohomology of Hyperk\"ahler Varieties".}

\begin{document}

\begin{abstract}
Using the technique of categorical absorption of singularities 
we prove that the nontrivial components of the derived categories of del Pezzo threefolds of degree~$d \in \{2,3,4,5\}$
and crepant categorical resolutions of the nontrivial components 
of the derived categories of nodal del Pezzo threefolds of degree~\mbox{$d = 1$}
can be smoothly deformed to the nontrivial components 
of the derived categories of prime Fano threefolds of genus~\mbox{$g = 2d + 2 \in \{4,6,8,10,12\}$}.
This corrects and proves the Fano threefolds conjecture of the first author from~\cite{K09},
and opens a way to interesting geometric applications, 
including a relation between the intermediate Jacobians and Hilbert schemes of curves of the above threefolds.
We also describe a compactification of the moduli stack of prime Fano threefolds endowed with an appropriate exceptional bundle
and its boundary component that corresponds to degenerations associated with del Pezzo threefolds.
\end{abstract}

\maketitle

\setcounter{tocdepth}{1}

\tableofcontents


\section{Introduction}

\subsection{Fano threefolds and their derived categories}

The main characters of this paper are smooth or nodal Fano threefolds with Picard number~$1$ and index~$1$ or~$2$, i.e.,
\begin{itemize}
\item 
$5$ families of \emph{del Pezzo} threefolds,
i.e., threefolds~$Y$ with the Picard group~$\Pic(Y)$ generated 
by the half of the anticanonical class~$H \coloneqq -\frac12K_Y$, 
classified by the degree
\begin{equation*}
\d(Y) \coloneqq H^3 = \tfrac18(-K_Y)^3 \in \{1,2,3,4,5\}
\end{equation*}
(see~\S\ref{sec:bridge} for a more detailed description), 
and
\item 
$10$ families of \emph{prime} Fano threefolds,
i.e., threefolds~$X$ with the Picard group~$\Pic(X)$ generated by the anticanonical class~$-K_X$,
classified by the genus
\begin{equation*}
\g(X) \coloneqq \tfrac12(-K_X)^3 + 1 \in \{2,3,4,5,6,7,8,9,10,12\}
\end{equation*}
(see~\S\ref{sec:stacks} for a more detailed description for~$\g(X) \in \{4,6,8,10,12\}$), 
so that~$(-K_X)^3 = 2\g(X) - 2$.
\end{itemize}

Since any terminal Gorenstein Fano threefold is \emph{smoothable} by~\cite{Na97}, 
we can consider singular varieties of this type as degenerations of smooth varieties.
The main advance of this paper is the discovery of a relation between derived categories 
associated with some of these degenerations.

A systematic study of derived categories of smooth Fano threefolds was initiated in~\cite{K09}, 
where it was shown that
\begin{itemize}
\item 
If~$X_{g}$ is a prime Fano threefold with~$\g(X_g) = g \in \{4,6,7,8,9,10,12\}$
(for~$g = 4$ the corresponding threefold should be general, see Proposition~\ref{prop:mukai-bundles}), 
there is a semiorthogonal decomposition
\begin{equation}
\label{eq:dbx-i1}
\Db(X_g) = \langle \cA_{X_{g}}, \cO_{X_{g}}, \cU^\vee_{X_{g}} \rangle,
\end{equation}
where~$\cU_{X_{g}}$ is the \emph{Mukai bundle}, see Definition~\ref{def:mukai} below (if~$g = 4$ and~$X_4$ is general 
there are two different Mukai bundles, hence two different decompositions of the form~\eqref{eq:dbx-i1}).
\item 
If~$X_g$ is a prime Fano threefold with~$g \in \{2,3,5\}$ (or a special threefold with~$g = 4$)
there is only a more coarse semiorthogonal decomposition
\begin{equation*}
\Db(X_g) = \langle \cA_{X_{g}}, \cO_{X_{g}} \rangle.
\end{equation*}
\item 
If~$Y_d$ is a del Pezzo threefold with~$\d(Y_d) = d$, $1 \le d \le 5$, there is a semiorthogonal decomposition
\begin{equation}
\label{eq:dby-i2}
\Db(Y_d) = \langle \cB_{Y_d}, \cO_{Y_d}, \cO_{Y_d}(H) \rangle.
\end{equation}
\end{itemize}

The components~$\cA_{X_g}$ and~$\cB_{Y_d}$ of these decompositions 
encode the most important geometric properties of the corresponding varieties~$X_g$ and~$Y_d$.
For instance, one can detect rationality of the varieties~$X_g$ or~$Y_d$ 
from the properties of their components~$\cA_{X_g}$ and~$\cB_{Y_d}$, see~\cite{K16}.

In some cases, the components~$\cA_{X_g}$ have an explicit description: 
\begin{equation*}
\cA_{X_7} \simeq \Db(C_7(X)),
\qquad 
\cA_{X_9} \simeq \Db(C_3(X)),
\qquad 
\cA_{X_{10}} \simeq \Db(C_2(X)),
\qquad 
\cA_{X_{12}} \simeq \Db(\Qu_3),
\end{equation*}
where~$C_g(X)$ is a (smooth proper) curve of genus~$g$ depending on~$X$
and~$\Qu_3$ is the quiver with~$2$ vertices and~$3$ arrows (which can be thought of as a \emph{noncommutative curve}),
see~\cite[\S\S6.2--6.4]{K06} for the first three equivalences 
and~\cite[Theorem~4.1]{K09} (summarizing the results from~\cite[Theorem~3]{K96} and~\cite[Theorem~2]{K97}) for the fourth.
Similarly, for the categories~$\cB_{Y_d}$ we have
\begin{equation*}
\cB_{Y_4} \simeq \Db(C_2(Y)),
\qquad 
\cB_{Y_5} \simeq \Db(\Qu_3),
\end{equation*}
see~\cite[Theorem~2.9]{BO95} or~\cite[\S6.5]{K06}, or~\cite[Corollary~5.7]{K08} for the first,
and~\cite{O91} or~\cite[\S6.1]{K06} for the second.
In particular, we have equivalences
\begin{equation}
\label{eq:cax1012-cby45}
\cA_{X_{10}} \simeq \cB_{Y_4}
\qquad\text{and}\qquad
\cA_{X_{12}} \simeq \cB_{Y_5}
\end{equation}
for pairs~$(X_{10},Y_4)$ such that~$C_2(X) \cong C_2(Y)$ and all pairs~$(X_{12},Y_5)$.
In the other cases no explicit description of the categories~$\cA_{X_g}$ or~$\cB_{Y_d}$ is available, 
but still one can prove yet another equivalence
\begin{equation}
\label{eq:cax8-cby3}
\cA_{X_8} \simeq \cB_{Y_3}
\end{equation}
for many pairs~$(X_8,Y_3)$, see~\cite[Theorem~3.17]{K04} or~\cite[Theorem~4.7]{K09}.

\subsection{Fano threefolds conjecture}

Motivated by the equivalences~\eqref{eq:cax1012-cby45} and~\eqref{eq:cax8-cby3}
and by some numerical coincidences, 
the first author suggested in~\cite[Conjecture~3.7]{K09}
so-called ``Fano threefolds conjecture'', saying
that for any~$1 \le d \le 5$ there is an equivalence
\begin{equation*}
\cA_{X_{2d+2}} \simeq \cB_{Y_d}
\end{equation*}
for ``many'' pairs~$(X_{2d+2},Y_d)$.
To explain what ``many'' means here, we need to define appropriate moduli spaces.
We introduce the necessary definitions in a slightly more general form, 
allowing for singularities and dealing with the stack structure of the moduli spaces;
this will be useful later.

\begin{definition}
\label{def:mf}
The moduli stacks~$\bMF_{\sX_g}$ and~$\bMF_{\sY_d}$ of {\sf nodal prime Fano threefolds of genus~$g$} or 
{\sf nodal del Pezzo threefolds of degree~$d$} are the fibered categories over~$(\Sch/\kk)$ 
with fiber over a scheme~$S$ the groupoid of flat projective morphisms of schemes~$f \colon \cX \to S$
such that for every geometric point~$s \in S$ the scheme~$\cX_s$ 
is a prime Fano threefold of genus~$g$ or a del Pezzo threefold of degree~$d$, respectively, 
with only (isolated) ordinary double points as singularities.
A morphism from~$f \colon \cX \to S$ to~$f' \colon \cX' \to S'$ is a fiber product diagram
\begin{equation*}
\xymatrix{
\cX \ar[r]^{\phi_\cX} \ar[d]_{f} &
\cX' \ar[d]^{f'} 
\\
S \ar[r]^{\phi} &
S'.
}
\end{equation*}
The moduli stacks~$\MF_{\sX_g} \subset \bMF_{\sX_g}$ and~$\MF_{\sY_d} \subset \bMF_{\sY_d}$ 
of {\sf smooth prime Fano threefolds} or {\sf smooth del Pezzo threefolds}
are defined as the substacks of morphisms~$f \colon \cX \to S$ such that~$f$ is smooth.
\end{definition}

In the case of prime Fano threefolds of even genus we will need an upgrade of the stack~$\bMF_{\sX_g}$;
to define it we need the following notion:

\begin{definition}
\label{def:mukai}
Let~$X$ be a nodal Fano threefold of even genus~$g$ over a field~$\kk$.
A {\sf Mukai bundle} on~$X$ is a $(-K_X)$-stable vector bundle~$\cU$
such that
\begin{equation}
\label{eq:mukai}
\rank(\cU) = 2,
\qquad
\rc_1(\cU) = K_{X},
\qquad
\rH^\bullet(X,\cU) = 0,
\qquad\text{and}\qquad
\Ext^\bullet(\cU,\cU) = \kk.
\end{equation}
Note that if~$\cU$ is a Mukai bundle on~$X$ then the pair~$(\cO_X,\cU^\vee)$ is exceptional.
\end{definition}

Given a morphism~$\cX \to S$ we define the {\sf \'etale sheaf~$\VB_{\cX/S}$ of vector bundles}
as the \'etale sheafification of the presheaf~on $S$ that takes an \'etale morphism~$S' \to S$ 
to the set of isomorphism classes of vector bundles on~$\cX \times_S S'$;
this is analogous to the widely used \'etale sheaf~$\Pic_{\cX/S}$ of line bundles (\cite[\S9.2]{Kleiman}).
A global section~$\cE$ of~$\VB_{\cX/S}$ over~$S$ is, by definition, the data of an \'etale covering~$\{S_i\} \to S$
and a collection of vector bundles~$\cE_i$ on~$\cX \times_S S_i$ whose pullbacks to~$\cX \times_S (S_i \times_S S_j)$ are isomorphic.
For a global section~\mbox{$\cE = (S_i,\cE_i) \in \VB_{\cX/S}(S)$} and a geometric point~$s \in S$
we denote by~$\cE_{\cX_s}$ the vector bundle on~$\cX_s$ defined as the restriction 
of the local section~$\cE_i$ of~$\cE$ from any open~$\cX \times_S S_i$ over~$s$.

\begin{definition}
\label{def:mfm}
Let~$g \in \{4,6,8,10,12\}$.
The {\sf moduli stack~$\bMFM_{\sX_g}$ of Fano--Mukai pairs} 
is the fibered category over~$(\Sch/\kk)$ with fiber over a scheme~$S$ 
the groupoid of pairs~$(f \colon \cX \to S, \cU)$,
where
\begin{itemize}
\item 
$f \colon \cX \to S$ is a flat projective
morphism of schemes and
\item 
$\cU \in \VB_{\cX/S}(S)$ is a global section of the \'etale sheaf of vector bundles,
\end{itemize}
such that~$f \colon \cX \to S$ is an $S$-point of~$\bMF_{\sX_g}$
and~$\cU_{\cX_s}$ is a Mukai bundle on~$\cX_s$ for every geometric point~$s \in S$.
Morphisms are defined as fiber product diagrams as in Definition~\ref{def:mf}
such that~$\phi_\cX^*(\cU') = \cU$ as sections of~$\VB_{\cX/S}(S)$.
The stack~$\MFM_{\sX_g} \subset \bMFM_{\sX_g}$ of {\sf smooth Fano--Mukai pairs}
is defined as the substack of pairs~$(f \colon \cX \to S, \cU)$ such that~$f$ is smooth.
\end{definition}

The Fano threefolds conjecture~\cite[Conjecture~3.7]{K09} claimed that for~$1 \le d \le 5$ there is a substack
\begin{equation}
\label{eq:def-zd}
\rZ_d \subset \MFM_{\sX_{2d+2}} \times \MF_{\sY_d}
\end{equation}
such that for each geometric point~$(X_{2d+2},\cU_X,Y_d)$ of~$\rZ_d$ 
there is an equivalence of categories~$\cA_{X_{2d+2}} \simeq \cB_{Y_d}$ 
defined in~\eqref{eq:dbx-i1} and~\eqref{eq:dby-i2}, respectively,
and that~$\rZ_d$ is \emph{dominant} over both factors.

The equivalences~\eqref{eq:cax1012-cby45} and~\eqref{eq:cax8-cby3} proved the cases~$3 \le d \le 5$ of the conjecture.
So only the cases~$d = 2$ and~$d = 1$ were left open.
However, for~$d = 1$ the dimensions of the intermediate Jacobians of~$X_4$ and~$Y_1$ differ by~$1$,
so the components~$\cA_{X_4}$ and~$\cB_{Y_1}$ have no chance to be equivalent on the nose.
Furthermore, for~$d = 2$ the conjecture also turned out to be wrong:
it was recently disproved in~\cite{Zhang,BP22};
in fact, it was shown that the categories~$\cA_{X_6}$ and~$\cB_{Y_2}$ are never equivalent.

\subsection{Modified Fano threefolds conjecture}
\label{ss:intro-modified-conjecture}

The goal of this paper is to explain a modification of the Fano threefolds conjecture and prove it.
The key idea is to include into consideration singular Fano threefolds
(this is why we extended Definitions~\ref{def:mf} and~\ref{def:mfm} to include singular threefolds),
and to rely on the technique of categorical absorption of singularities developed
in~\cite{KS22:abs}.

Our main theorem, stated below, uses the notion of base change for $B$-linear semiorthogonal decompositions for which we refer to~\cite{K11}. 
It also uses a natural crepant categorical resolution~$\tcB_Y$
(see Lemma~\ref{lem:tcb-y} for its construction)
of the category~$\cB_Y$ for a $1$-nodal del Pezzo threefold of degree~$1$ 
(that is defined by~\eqref{eq:dby-i2}, like in the smooth case).

\begin{theorem}[{= Theorem~\textup{\ref{thm:cax-family}}}]
\label{thm:intro-simple}
Let~$Y$ be a smooth del Pezzo threefold of degree~$2 \le d \le 5$ or a $1$-nodal del Pezzo threefold of degree~$d = 1$.
There is a $B$-point~$(f \colon \cX \to B, \cU_\cX)$ of~$\bMFM_{\sX_{2d + 2}}$ for a smooth pointed curve~$(B,o)$ 
and a $B$-linear semiorthogonal decomposition
\begin{equation*}
\Db(\cX) = \langle \Db(o), \bcA_\cX, f^*\Db(B), f^*\Db(B) \otimes \cU_\cX^\vee \rangle
\end{equation*}
such that 
\begin{alenumerate}
\item 
\label{it:thm-simple-open}
for each point~$b \ne o$ the fiber~$\cX_b$ is a smooth prime Fano threefold of genus~$2d + 2$
and~$(\bcA_\cX)_b \simeq \cA_{\cX_b}$\textup;
\item 
\label{it:thm-simple-closed}
the fiber~$\cX_o$ is $1$-nodal and~$(\bcA_\cX)_o \simeq \cB_Y$ if~$2 \le d \le 5$ 
or~$(\bcA_\cX)_o \simeq \tcB_Y$ if~$d = 1$.
\end{alenumerate}
In particular, the category~$\bcA_\cX$ is smooth and proper over~$B$.
\end{theorem}

In other words, the category~$\bcA_\cX$ provides a (smooth and proper) interpolation 
between the components~$\cA_{\cX_b}$ of smooth prime Fano threefolds of genus~$g = 2d + 2$ defined by~\eqref{eq:dbx-i1}
and the components~$\cB_Y$ (or their crepant categorical resolutions~$\tcB_Y$) 
of del Pezzo threefolds of degree~$d$ defined by~\eqref{eq:dby-i2}.

\begin{remark}
\label{rem:isotrivial}
We do not prove this here, but in the case~$d \ge 3$ 
one can choose the $B$-point~\mbox{$(f \colon \cX \to B, \cU_\cX)$} of the stack~$\bMFM_{\sX_{2d+2}}$
in such a way that the family of categories~$\bcA_\cX$ is \emph{isotrivial},
i.e., \mbox{$(\bcA_\cX)_b \simeq (\bcA_\cX)_o$} for all~$b$;
then~$\cA_{\cX_b} \simeq \cB_Y$, which implies that 
the correspondence~$\rZ_d \subset \MFM_{\sX_{2d+2}} \times \MF_{\sY_d}$ defined in~\eqref{eq:def-zd}
is dominant over~$\MF_{\sY_d}$.
\end{remark}

In fact, Theorem~\ref{thm:intro-simple} can be deduced by base change 
from a more general result describing the structure of the derived category
of more general families of Fano threefolds, see~\eqref{eq:dbcx-bca} and the preceding discussion.
Our techniques also allow us to prove similar extension results for the moduli stacks of curves
and moduli stacks of prime Fano threefolds of genus~$9$, $7$, and~$5$;
these results will be presented elsewhere.

A similar result in the case~$d = 2$ was proved by a completely different technique in~\cite[Theorem~1.6]{BP22}.
However, the approach of~\cite{BP22} only works for \emph{special}~$Y$ (for double solids whose ramification divisor contains a line)
and provides families~$\cX/B$ of \emph{special} threefolds of genus~$6$ (double covers of a quintic del Pezzo threefold),
but the category constructed in~\cite{BP22}  in this case  is equivalent to our category~$\bcA_\cX$.

\subsection{Geometric applications}

A smooth and proper family of complex varieties gives rise to a variation of pure Hodge structures
and, under appropriate assumptions, to a smooth and proper family of principally polarized abelian varieties.
In particular, the family of intermediate Jacobians 
of a smooth and proper family of Fano threefolds is also smooth and proper.
We expect the same to be true for the smooth and proper family of triangulated categories~$\bcA_\cX$
constructed in Theorem~\ref{thm:intro-simple}.

Recall that the intermediate Jacobian of an admissible triangulated subcategory~$\cC \subset \Db(Z)$
in the derived category of a smooth and proper complex variety~$Z$
was defined in~\cite[Definition~5.24]{Perry} as
\begin{equation*}
\cJ(\cC) = \cJ(\rKt(\cC)),
\end{equation*}
where~$\rKt(\cC)$ is the degree~$1$ component of the topological $\rK$-theory of categories, 
endowed with appropriately defined pure Hodge structure of weight~$-1$ (see~\cite[Proposition~5.4]{Perry}).
On the other hand, a relative version of topological $\rK$-theory of categories
was constructed by Moulinos in~\cite[\S7.2]{M19}.

\begin{conjecture}[{cf.~\cite[Remark~5.8]{Perry}}]
\label{conj:intro-relative-jacobian}
Let~$S$ be a scheme over~$\CC$ and let~$\cC/S$ be a smooth and proper $S$-linear triangulated category.
If for each point~$s \in S$ the category~$\cC_s$ is equivalent 
to an admissible subcategory in the derived category of a smooth and proper variety,
the relative topological $\rK$-theory~$\rKt(\cC/S)$ admits a variation of pure Hodge structures of weight~$-1$
which for each point~$s \in S$ agrees with the Hodge structure of~$\rKt(\cC_s)$.
In particular,
\begin{equation*}
\cJ(\cC/S) \coloneqq \cJ(\rKt(\cC/S))
\end{equation*}
is a smooth and proper family of complex tori, such that for each point~$s \in S$ we have~$\cJ(\cC/S)_s \cong \cJ(\cC_s)$.
\end{conjecture}

Applying this to the smooth and proper family of categories~$\bcA_\cX$
and using~\cite[Theorem~1.6]{Perry} to identify the fibers, we would obtain the following

\begin{corollary}
\label{cor:intro-jacobians}
For~$Y$ and~$\cX$ constructed in Theorem~\textup{\ref{thm:intro-simple}}
there is a family~$\cJ \to B$ of principally polarized abelian varieties such that
\begin{equation*}
\cJ_b \cong 
\begin{cases}
\Jac(\cX_b), & \text{if~$b \ne o$},  
\\
\Jac(Y), & \text{if~$b = o$},
\end{cases}
\end{equation*}
where in the case~$d = 1$ the right-hand side is understood 
as~$\Jac(\Bl_{y_0}(Y))$, where~$y_0 \in Y$ is the node.
\end{corollary}

Although we would like to see Corollary~\ref{cor:intro-jacobians} 
as a consequence of Theorem~\ref{thm:intro-simple} and Conjecture~\ref{conj:intro-relative-jacobian},
there is a direct Hodge-theoretic proof, using the geometric construction of~$\cX$ from~$Y$;
we deduce it in the Appendix from a slightly more general Proposition~\ref{prop:jacobian-extension}.

Corollary~\ref{cor:intro-jacobians} provides a conceptual explanation for
the coincidences between the dimensions of the intermediate Jacobians 
that we list in the following table (Mukai calls it ``the periodic table of Fano threefolds'', see~\cite[Table~1.4]{Mukai02}).
\begin{equation*}
\begin{array}{|c|c|c|c|c|c|c|c|c|c|c|c|c|}
    \hline
    \text{Prime Fano threefolds} & \sX_{12} & \sX_{10} & \sX_9 & \sX_8 & \sX_7 & \sX_6 & \sX_5 & \sX_4 & - & \sX_3 & \sX_2 \\
    \hline
    \text{Del Pezzo threefolds} & \sY_5 & \sY_4 & - & \sY_3 & - & \sY_2 & - & - & \sY_1 & - & - \\
    \hline
    \dim(\Jac(-)) & 0 & 2 & 3 & 5 & 7 & 10 & 14 & 20 & 21 & 30 & 52\\
    \hline
\end{array}
\end{equation*}
Note how the discrepancy between~$\dim(\Jac(\sX_4))$ and~$\dim(\Jac(\sY_1))$
matches the fact that in Theorem~\ref{thm:intro-simple} and Corollary~\ref{cor:intro-jacobians}
we consider 1-nodal threefolds of type~$\sY_1$:
if~$Y$ is a 1-nodal and~$Y'$ is a smooth del Pezzo threefold of degree~$d = 1$, 
we have~$\dim(\Jac(Y)) = \dim(\Jac(Y')) - 1 = 20$.

\medskip 

Similarly, a proper family of varieties gives rise to various relative moduli spaces of stable coherent sheaves;
the same is true for families of categories if they are endowed with appropriate stability conditions,
see~\cite[Definition~1.1 and Theorem~21.24]{BLMNPS}.
We expect that stability conditions for the families of categories~$\bcA_\cX$ exist and give rise to interesting relative moduli spaces.

\begin{conjecture}
\label{conj:intro-moduli}
For a del Pezzo threefold~$Y$, a family of prime Fano threefolds~$\cX/B$, and the $B$-linear category~$\bcA_\cX/B$
constructed in Theorem~\textup{\ref{thm:intro-simple}},
the \'etale sheafification~$\rKn(\bcA_\cX / B)$ of the relative numerical Grothendieck group is locally constant,
and there is a numerical stability condition~$\underline{\upsigma}$
on~$\bcA_\cX$ over~$B$ such that the corresponding stability condition~$\underline{\upsigma}_b$ 
on the fiber~$(\bcA_\cX)_b$ of~$\bcA_\cX$ is
\begin{itemize}
\item 
a 
numerical stability condition on~$\cA_{\cX_b}$, if~$b \ne o$, and
\item 
a numerical stability condition on~$\cB_Y$ \textup(for~$2 \le d \le 5$\textup) or~$\tcB_Y$ \textup(for~$d = 1$\textup),
if~$b = o$.
\end{itemize}
Moreover, if~$\bv \in \rKn(\bcA_\cX / B)(B)$ is a section,
there is a moduli space~$\rM_{\underline{\upsigma}}(\bcA_\cX,\bv)$ over~$B$ such that
\begin{equation*}
\rM_{\underline{\upsigma}}(\bcA_\cX,\bv)_b \cong
\begin{cases}
\rM_{\underline{\upsigma}_b}(\cA_{\cX_b},\bv), & \text{if~$b \ne o$},
\\
\rM_{\underline{\upsigma}_o}(\cB_Y,\bv)\ \text{or}\ 
\rM_{\underline{\upsigma}_o}(\tcB_Y,\bv), & \text{if~$b = o$ and~$2 \le d \le 5$ or~$d = 1$}.
\end{cases}
\end{equation*}
In other words, the fibers of~$\rM_{\underline{\upsigma}}(\bcA_\cX,\bv)$ 
are appropriate moduli spaces of stable objects in~$\cA_{\cX_b}$, $\cB_Y$, or~$\tcB_Y$.
\end{conjecture}

If~$d = 2$, an interesting example of a moduli space associated with a 
stability condition
in the category~$\cA_X \subset \Db(X)$ for a general smooth Fano threefold~$X$ of genus~$2d + 2 = 6$ 
is the minimal model~$\rF^{\mathrm{min}}_2(X)$ of the Hilbert scheme of conics~$\rF_2(X)$ on~$X$, see~\cite[Theorem~7.12]{JLLZ}.
It is natural to expect that the analogous moduli space associated with the category~$\cB_Y \subset \Db(Y)$
is the Hilbert scheme of lines~$\rF_1(Y)$.
In this example we expect the following to be true.

\begin{conjecture}
\label{conj:intro-hilbert}
For a quartic double solid~$Y \to \P^3$ that contains no lines in the ramification divisor 
and a family of prime Fano threefolds~$\cX/B$ of genus~$g = 6$
constructed in Theorem~\textup{\ref{thm:intro-simple}} 
there is a stability condition~$\underline{\upsigma}$ on~$\bcA_\cX$ and a section~$\bv \in \rKn(\bcA_\cX / B)(B)$
such that the corresponding moduli space~$\rM_{\underline{\upsigma}}(\bcA_\cX,\bv)$ is
a smooth and proper family of surfaces~$\cF(\cX/B)$ such that
\begin{equation*}
\cF(\cX/B)_b \cong 
\begin{cases}
\rF^{\mathrm{min}}_2(\cX_b), & \text{if~$b \ne o$},\\
\rF_1(Y), & \text{if~$b = o$}.
\end{cases}
\end{equation*}
\end{conjecture}

By~\cite[Theorem~1.1]{DK20jac} and~\cite[Theorem~4.1]{Welters} 
the relative Albanese variety~$\Alb(\cF(\cX/B)/B)$ of the above family of surfaces should be isomorphic
to the relative family of intermediate Jacobians, so Conjecture~\ref{conj:intro-hilbert} should give
yet another proof of Corollary~\ref{cor:intro-jacobians} for~$d = 2$.

Of course, we expect a similar result to be true for~$d \ge 3$, but it is less interesting,
because in this case~$\rF_1(Y_d) \cong \rF_2(X_{2d+2})$ for any~$Y_d$ and appropriate~$X_{2d+2}$,
see~\cite[Propositions~B.4.1, B.5.1, and~B.6.1]{KPS18}, and, for instance, 
if we consider a family~$\cX/B$ giving rise to an isotrivial family of categories~$\bcA_\cX/B$ 
(see Remark~\ref{rem:isotrivial}), the corresponding family of surfaces~$\cF(\cX/B)$ will also be isotrivial.

On the other hand, the case~$d = 1$ may be very interesting, and may provide a useful insight
into the geometry of the Hilbert scheme of conics on~$X_4$ and the Hilbert scheme of lines on~$Y_1$.

\begin{remark}
We were informed by the authors that Conjectures~\ref{conj:intro-moduli} and~\ref{conj:intro-hilbert} 
are proved when~$B$ is the spectrum of a complete DVR (and~$d \ge 2$ for the first conjecture)
in the forthcoming paper~\cite{LMPSZ}.
\end{remark}

\subsection{A sketch of the proof}

Our proof of Theorem~\ref{thm:intro-simple}
and its generalization~\eqref{eq:dbcx-bca}
is based on a geometric construction, which we call {\sf a bridge}.
This construction connects the realms of del Pezzo threefolds and prime Fano threefolds of even genus.

If~$d \ge 2$ to construct a bridge we consider
a smooth del Pezzo threefold~$Y$ of degree~$d$, a smooth rational curve~$C \subset Y$ of degree~\mbox{$d - 1$},
and the blowup~$\Bl_C(Y)$.
In Proposition~\ref{prop:x-y} we show that the anticanonical class of~$\Bl_C(Y)$ is nef and big 
and defines a small birational contraction
\begin{equation*}
\pi \colon \Bl_C(Y) \to \Bl_C(Y)_\can \eqqcolon X
\end{equation*}
to a 1-nodal nonfactorial prime Fano threefold~$X$ of genus~$2d + 2$,
the {\sf anticanonical model} of~$\Bl_C(Y)$.
The morphism~$\pi$ contracts a single smooth rational curve 
(the strict transform of the unique bisecant line for~$C$ in~$Y$) in~$\Bl_C(Y)$ to the node~\mbox{$x_0 \in X$}.
Furthermore, in Lemma~\ref{lemma:mutation-cu} and Proposition~\ref{prop:mutations-1}
we construct a Mukai bundle~$\cU_{X}$ on~$X$; it is worth pointing here that~$\cU_X$ depends on both~$Y$ and~$C$.

If~$d = 1$ the construction is similar.
In this case we consider a $1$-nodal del Pezzo threefold~$Y$ of degree~$d = 1$, 
the blowup~$\Bl_{y_0}(Y)$ of~$Y$ at the node~$y_0 \in Y$, 
its anticanonical model~$X \coloneqq \Bl_{y_0}(Y)_\can$, and construct a Mukai bundle~$\cU_X$ on~$X$
that depends on a choice of ruling~$F$ of the exceptional divisor of~$\Bl_{y_0}(Y) \to Y$
(the ruling~$F$ plays here the role similar to that of a curve~$C$ when~$d \ge 2$).

From now on we concentrate on the simpler case where~$d \ge 2$.
Let~$(Y,C)$ be a smooth del Pezzo threefold of degree~$d \ge 2$ with a smooth rational curve of degree~\mbox{$d - 1$},
let~$X = \Bl_C(Y)_\can$ be the anticanonical model of the blowup~$\Bl_C(Y)$,
and let~$\cU_X$ be the corresponding Mukai bundle, obtained by the bridge construction.
Let, furthermore, $f \colon \cX \to B$ be a smoothing of~$X$ (it exists by~\cite[Theorem~11]{Na97}) 
over a smooth pointed curve~$(B,o)$.
Using exceptionality of the bundle~$\cU_X$ we check that (possibly after base change to an \'etale neighborhood of~$o \in B$) 
it extends to a global section~\mbox{$\cU_\cX \in \VB_{\cX/B}(B)$} 
such that~$(f \colon \cX \to B, \cU_\cX)$ is a $B$-point of the stack~$\bMFM_{\sX_{2d+2}}$, 
the central fiber~$\cX_o$ is isomorphic to~$X$, 
and the morphism~$f$ is smooth over the punctured curve~$B \setminus \{o\}$.

To complete the proof of Theorem~\ref{thm:intro-simple} it remains to explain 
how the subcategory~$\bcA_\cX \subset \Db(\cX)$ is constructed and how its properties are verified.
We do this using the technique of \emph{categorical absorption of singularities} developed in~\cite{KS22:abs}.
More precisely, applying~\cite[Theorem~1.5 and Theorem~6.1]{KS22:abs} we obtain semiorthogonal decompositions
\begin{equation*}
\Db(\cX) = \langle \io_*\rP_X, \cD \rangle,
\qquad 
\Db(X) = \langle \rP_X, \cD_o \rangle,
\qquad\text{and}\qquad 
\Db(\cX_b) = \cD_b,
\end{equation*}
where~$\io \colon X \to \cX$ is the embedding of the central fiber,
$\rP_X \in \Db(X)$ is a so-called $\Pinfty{2}$-object,
so that~\mbox{$\io_*\rP_X \in \Db(\cX)$} is an exceptional object,
$\cD$ is a $B$-linear admissible subcategory in~$\Db(\cX)$
which is smooth and proper over~$B$,
while~$\cD_o$ and~$\cD_b$ are its base changes along the embeddings~$\{o\} \hookrightarrow B$ and~$\{b\} \hookrightarrow B$.
The category~$\cD$ can be thought of 
as a family of smooth and proper triangulated categories parameterized by the curve~$B$.
For more details about this construction see Theorem~\ref{thm:cax-family} and its proof.

Finally, we refine the above decomposition slightly.
We observe that the structure sheaf~$\cO_\cX$ and the dual Mukai bundle~$\cU^\vee_\cX$ on~$\cX$ 
are contained in the subcategory~$\cD \subset \Db(\cX)$ and form a relative over~$B$ exceptional pair,
hence they induce a $B$-linear semiorthogonal decomposition
\begin{equation*}
\cD = \langle \bcA_\cX, f^*(\Db(B)) \otimes \cO_\cX, f^*(\Db(B)) \otimes \cU_\cX^\vee \rangle,
\end{equation*}
thus defining a $B$-linear triangulated subcategory~$\bcA_\cX \subset \cD \subset \Db(\cX)$ 
which is smooth and proper over~$B$.
It follows from~\cite[Theorem~5.6]{K11}
that the fibers~$(\bcA_\cX)_b$ of this category corresponding to points~$b \ne o$
are equivalent to the subcategories~\mbox{$\cA_{\cX_b} \subset \Db(\cX_b)$} from~\eqref{eq:dbx-i1}.
On the other hand, the category~$(\bcA_\cX)_o$ corresponding to the origin~\mbox{$o \in B$} can be identified
(by an appropriate sequence of mutations, explained in Proposition~\ref{prop:mutations-1}) 
with the subcategory~\mbox{$\cB_Y \subset \Db(Y)$} defined by~\eqref{eq:dby-i2}.

In the case~$d = 1$ the argument proving Theorem~\ref{thm:intro-simple} is essentially the same. 

\subsection{Boundary components of the compactified moduli stack}

Theorem~\ref{thm:intro-simple} shows that the family of categories~$\cA_X$ 
defined by~\eqref{eq:dbx-i1} for all smooth prime Fano threefolds 
(or, more precisely, for all smooth Fano--Mukai pairs) of genus~\mbox{$g \in \{4,6,8,10,12\}$}
extends naturally to some degenerations of these varieties.
In~\S\ref{sec:stacks-proofs} we describe precisely the locus in the stack~$\bMFM_{\sX_g}$ corresponding to such degenerations.

For this we consider the stack~$\tMF_{\sY_d}$ parameterizing pairs~$(Y,C)$ or~$(Y,F)$ 
used in the bridge construction as described in~\S\ref{ss:intro-modified-conjecture} 
(see Definition~\ref{def:tmf} for the actual definition of this stack) 
and applying a relative version of the bridge construction we define in Lemma~\ref{lem:mu} a morphism of stacks
\begin{equation*}
\upmu \colon \tMF_{\sY_d} \to \bMFM_{\sX_{2d+2}}^{(1)} \subset \bMFM_{\sX_{2d+2}},
\end{equation*}
where~$\bMFM_{\sX_{2d+2}}^{(1)} \subset \bMFM_{\sX_{2d+2}}$ is the $1$-nodal locus.
In fact, as we check in~Theorem~\ref{thm:mf-x-y} 
the substack~$\bMFM_{\sX_{2d+2}}^{(1)}$ 
is a Cartier divisor in the open substack~$\bMFM_{\sX_{2d+2}}^{\le 1} \subset \bMFM_{\sX_{2d+2}}$ of at most $1$-nodal Fano--Mukai pairs
and we prove in Theorem~\ref{thm:tmfy-bmfx-d2} that~$\upmu$ is an isomorphism of~$\tMF_{\sY_d}$
onto a connected component of~$\bMFM_{\sX_{2d+2}}^{(1)}$, which we denote 
\begin{equation*}
\bMFM_{\sX_{2d+2},\sY_d}^{(1)} \subset \bMFM_{\sX_{2d+2}}^{(1)} \subset \bMFM_{\sX_{2d+2}}
\end{equation*}
and call {\sf the del Pezzo component} of~$\bMFM_{\sX_{2d+2}}^{(1)}$.
Thus, Theorem~\ref{thm:intro-simple} (or rather its generalization~\eqref{eq:dbcx-bca})
can be interpreted as a construction of a smooth and proper extension of the family of categories~$\cA_X$ 
across the del Pezzo component of the boundary of~$\bMFM_{\sX_{2d+2}}$.
We discuss a convenient way to think about such an extension in~\S\ref{ss:cpm} below.

The above 
observation motivates the following problem, which is also interesting by itself.

\begin{problem}
\label{ques}
Classify all connected components of the boundary divisor~$\bMFM_{\sX_{2d+2}}^{(1)} \subset \bMFM_{\sX_{2d+2}}$
and study possible extensions of the family of categories~$\cA_X$ across these components.
\end{problem}

The categorical absorption point of view suggests that a nice extension is possible for those boundary components
that correspond to \emph{nonfactorial} 1-nodal degenerations of~$X$;
such degenerations are classified in~\cite{KP23} and independently in~\cite{CKMS}, thus answering the first half of this question.

\subsection{The categorical period map}
\label{ss:cpm}

We finish the Introduction with a speculative section, 
discussing a possible reformulation of our results in terms of a \emph{categorical period map}.

To define the categorical period map we need to introduce a stack~$\MT$ of triangulated categories.
One option is to define it as the \emph{\'etale sheafification} 
of the fibered category over~$(\Sch/\kk)$ whose fiber over a scheme~$S$ 
is the groupoid of $S$-linear enhanced triangulated categories and their $S$-linear enhanced equivalences.
To make this into a real definition, we need, however, to treat~$\MT$ as \emph{higher stack},
which goes very far out of the scope of this paper.
Another option is to use the approach developed in~\cite{AT}.
Anyway, assuming that~$\MT$ is defined appropriately, given any flat projective morphism~$f \colon \cX \to S$
and an admissible $S$-linear subcategory~$\cD \subset \Db(\cX)$ one should be able to produce an $S$-point of~$\MT$, 
i.e., a morphism
\begin{equation*}
\wp_\cD \colon S \to \MT
\end{equation*}
which (by analogy with Hodge theory) we call the {\sf categorical period map}.

Now we explain how our results would be interpreted in the (hypothetical) terms 
of the moduli stack~$\MT$ and the categorical period map~$\wp$.
Since the components~$\cA_\cX \subset \Db(\cX)$ and~$\cB_\cY \subset \Db(\cY)$ 
of the family versions of decompositions~\eqref{eq:dbx-i1} and~\eqref{eq:dby-i2}
associated to families~$(f \colon \cX \to S, \cU_\cX)$ and~$g \colon \cY \to S$
are admissible $S$-linear subcategories, they should define the categorical period maps of stacks
\begin{equation*}
\wp_{\cA} \colon \MFM_{\sX_g} \to \MT,
\qquad\text{and}\qquad 
\wp_{\cB} \colon \MF_{\sY_d} \to \MT.
\end{equation*} 
Then the original Fano threefolds conjecture (which holds for~$d \ge 3$ but fails for~$d \le 2$)
can be rephrased as the dominance of the fiber product 
\begin{equation*}
\rZ_d = 
\MFM_{\sX_{2d+2}} \bigtimes\limits_{\MT} \MF_{\sY_d},
\end{equation*}
(with respect to the categorical period maps~$\wp_\cA$ and~$\wp_\cB$) over both factors.

On the other hand, the semiorthogonal decomposition of Theorem~\ref{thm:intro-simple} 
(or rather its generalization~\eqref{eq:dbcx-bca}) can be interpreted as a commutative diagram
\begin{equation*}
\vcenter{\xymatrix@R=4.5ex{
\tMF_{\sY_d} \ar[r]^-\upmu_-\sim \ar[d] &
\bMFM_{\sX_{2d+2},\sY_d}^{(1)} \ar[d]^{\wp_\bcA}
\\
\MF_{\sY_d} \ar[r]^-{\wp_\cB} &
\MT
}}
\quad\text{for~$d \ge 2$,}
\qquad\text{or}\qquad
\vcenter{\xymatrix@R=4.5ex{
\tMF_{\sY_1} \ar[r]^-\upmu_-\sim \ar[d] &
\bMFM_{\sX_{4},\sY_1}^{(1)} \ar[d]^{\wp_\bcA}
\\
\bMF_{\sY_1}^{(1)} \ar[r]^-{\wp_\tcB} &
\MT
}}
\quad\text{for~$d = 1$.}
\end{equation*}
In both diagrams the left vertical arrow is the forgetful map,
$\wp_{\bcA}$ is the categorical period map associated with the component~$\bcA_\cX \subset \Db(\cX)$ 
of the semiorthogonal decomposition from~\eqref{eq:dbcx-bca} for a family~\mbox{$(f \colon \cX \to S, \cU_\cX)$} of Fano--Mukai pairs,
while~$\wp_\cB$ or~$\wp_\tcB$ is the categorical period map 
associated with (the categorical resolution~$\tcB_\cY$ of) the component~$\cB_\cY \subset \Db(\cY)$.
In particular, we see that
\begin{equation*}
\tMF_{\sY_d} \subset
\bar\rZ_d^{(1)} \coloneqq
\bMFM_{\sX_{2d+2},\sY_d}^{(1)} \bigtimes\limits_{\MT} \MF_{\sY_d},
\qquad\text{or}\qquad 
\tMF_{\sY_1} \subset
\bar\rZ_1^{(1)} \coloneqq
\bMFM_{\sX_{4},\sY_1}^{(1)} \bigtimes\limits_{\MT} \bMF^{(1)}_{\sY_1},
\end{equation*}
and therefore~$\bar\rZ_d$ is dominant over~$\MF_{\sY_d}$ for~$d \ge 2$, 
and over the $1$-nodal locus~$\bMF_{\sY_1}^{(1)} \subset \bMF_{\sY_1}$ for~\mbox{$d = 1$}.

Note that~$\wp_\bcA$ differs on the boundary of~$\bMFM_{\sX_{2d+2}}$ from the naive extension of~$\wp_\cA$;
in fact, $\wp_\bcA$ takes values in \emph{smooth and proper} triangulated categories,
while the naive extension does not.

Thus, our construction provides a partial extension~$\wp_\bcA$ of the categorical period map~$\wp_\cA$
across the del Pezzo component~$\bMFM_{\sX_{2d+2},\sY_d}^{(1)} \subset \bMFM_{\sX_{2d+2}}$,
but it has different flavour in the cases~$d \ge 3$ and~$d \le 2$.

When~$d \ge 3$ the images of~$\wp_\cA$ and~$\wp_\bcA$ are the same (and equal to the image of~$\wp_\cB$), 
but the fibers of~$\wp_\bcA$ are partial compactifications of the fibers of~$\wp_\cA$.
For instance, in the case~$d = 3$ the fiber of~$\wp_\cA$ over the point~$\wp_\cB([Y]) = [\cB_Y]$ for a smooth cubic threefold~$Y$
is isomorphic to the ``locally free'' locus in a certain moduli space of stable coherent sheaves of rank~$2$ on~$Y$, see~\cite{K04}.
The complement of this locus is the union of two divisors, see~\cite{Druel}, 
and the corresponding fiber of~$\wp_\bcA$ also includes a dense open subset of one of these.
We expect that a dense open subset of the second divisor corresponds to another boundary component of~$\bMFM_{\sX_8}$;
we plan to discuss this elsewhere.

On the other hand, when~$d \le 2$ the fibers of~$\wp_\cA$ and~$\wp_\bcA$ are the same,
but the image of~$\wp_\bcA$ is larger than the image of~$\wp_\cA$.
Moreover, the fibers of~$\wp_\bcA$ over the image of~$\wp_\cB$ or~$\wp_\tcB$, respectively, are deformations of the fibers of~$\wp_\cA$.
For instance, in the case~$d = 2$ the fiber of~$\wp_\bcA$ over the point~$\wp_\cB([Y])$ for a quartic double solid~$Y$
is isomorphic to the Hilbert scheme of lines on~$Y$,
and this is a deformation of double Eisenbud--Popescu--Walter surfaces 
that are (conjecturally) isomorphic to the fibers of~$\wp_\cA$.

Thus, the map~$\wp_\bcA$ compactifies~$\wp_\cA$ ``vertically'' for~$d \ge 3$, 
and ``horizontally'' for~$d \le 2$.

\subsubsection*{Structure of the paper}

In~\S\ref{sec:bridge} we explain the bridge construction of a 1-nodal prime Fano threefold~$X$ of genus~$2d + 2$
from a del Pezzo threefold~$Y$ of degree~$d$ endowed with a curve~$C$ or a ruling~$F$.
In~\S\ref{sec:component} we construct the~$\Pinfty2$-object~$\rP_X$ and the Mukai bundle~$\cU_X$ on~$X$
and prove Theorem~\ref{thm:intro-simple}.
In~\S\ref{sec:stacks} we discuss general properties of the moduli spaces of Mukai bundles 
and moduli stacks of Fano threefolds and Fano--Mukai pairs.
Finally, in~\S\ref{sec:stacks-proofs} we identify the del Pezzo component of the boundary divisor 
of the moduli stack~$\bMFM_{\sX_{2d+2}}$.

In Appendix~\ref{sec:nodal} we discuss some material about nodal varieties that is used in the body of the paper
and give a Hodge-theoretic proof of Corollary~\ref{cor:intro-jacobians}.

\subsubsection*{Conventions}

We work over an algebraically closed field~$\kk$ of characteristic~$0$.
When we say that a variety~$X$ is {\sf nodal} (resp.\ {\sf $k$-nodal}), 
we mean that~$\Sing(X)$ is finite or empty (resp.\ has length~$k$)
and every point in~$\Sing(X)$ is an ordinary double point on~$X$,
see Definition~\ref{def:nodal}.

\subsubsection*{Acknowledgements}

We would like to thank Olivier Debarre, Sergey Gorchinskiy, Alex Perry, Yuri Prok\-ho\-rov, Andrey Soldatenkov, and Claire Voisin
for useful discussions and the referee for their comments.


\section{The bridge}
\label{sec:bridge}

In this section we present the main geometric construction of the paper --- 
the bridge, linking del Pezzo threefolds~$Y$ to prime Fano threefolds~$X$ of even genus.
We upgrade this construction to an isomorphism of moduli stacks in Section~\ref{sec:stacks-proofs}.

Recall that for any del Pezzo threefold~$Y$ with~$\Pic(Y) = \ZZ \cdot H$, where~$H \coloneqq - \tfrac12 K_Y$,
we denote by~$\d(Y) \coloneqq H^3$ the degree of~$Y$.
Then~$d = \d(Y) \in \{1,2,3,4,5\}$, and~$Y$ can be described as follows (see~\cite[Theorem~1.2]{KP22}):
\begin{itemize}
\item 
if~$d = 5$ then~$Y = \Gr(2,5) \cap \P^6 \subset \P^9$;
\item 
if~$d = 4$ then~$Y$ is a complete intersection of two quadrics in~$\P^5$;
\item 
if~$d = 3$ then~$Y$ is a cubic hypersurface in~$\P^4$;
\item 
if~$d = 2$ then~$Y$ is a {\sf quartic double solid}, i.e., a double covering of~$\P^3$ branched at a quartic;
\item 
if~$d = 1$ then~$Y$ is a sextic hypersurface in the weighted projective space~$\P(1,1,1,2,3)$.
\end{itemize}
By the Riemann--Roch Theorem, we have~$\dim |H| = d + 1$,
and the linear system~$|H|$ defines a map
\begin{equation}\label{eq:phi-d}
\varphi  = \varphi_d \colon Y \dashrightarrow \P^{d+1}.
\end{equation}
If~$d \ge 3$ it is a closed embedding, 
if~$d = 2$ it is a regular double covering, 
and if~$d = 1$ it is a rational elliptic fibration with a single indeterminacy point, called the {\sf base point} of~$Y$ 
that coincides with the intersection of~$Y$ with the weighted projective line~$\P(2,3) \subset \P(1,1,1,2,3)$.

\begin{remark}
\label{rem:2h}
If~$d = 1$ the linear system~$|2H|$ is base point free and defines a regular double covering
\begin{equation*}
\widehat\varphi_1 \colon Y \to \P(1,1,1,2).
\end{equation*}
This follows from the analogous property of del Pezzo surfaces of degree~$1$ 
by the argument of~\cite[Proposition~2.2(ii)]{KP22}. 
Because of this property, $Y$ is called a {\sf double Veronese cone}.
\end{remark}

We will say that a singular point~$y_0 \in Y$ is a  {\sf cusp} (also called {\sf generalized cusp} in~\cite{KP23})
if it is a hypersurface singularity such that~$\Bl_{y_0}(Y)$ is smooth along the exceptional divisor~$E \subset \Bl_{y_0}(Y)$
which is an irreducible singular quadric surface and~$\cO_E(-E)$ is the hyperplane class of~$E$.

\begin{setup}
\label{setup}
In the rest of the paper we work in one of the following two situations:
\begin{alenumerate}
\item 
\label{it:curve-d-big}
$(Y,C)$ is a pair, where
$Y$ is a smooth del Pezzo threefold of degree~$d \in \{2,3,4,5\}$, 
$C \subset Y$ is a smooth rational curve on~$Y$ of degree~$d - 1$, and 
\begin{equation*}
\tY \coloneqq \Bl_C(Y) \xrightarrow{\ \sigma\ } Y
\end{equation*}
is the blowup with exceptional divisor~$E \subset \tY$, or
\item 
\label{it:curve-d=1}
$(Y,F)$ is a pair, where~$Y$ is a del Pezzo threefold of degree~$d = 1$ with a single node or cusp~$y_0 \in Y$,
\begin{equation*}
\tY \coloneqq \Bl_{y_0}(Y) \xrightarrow{\ \sigma\ } Y
\end{equation*}
is the blowup with exceptional divisor~$E \subset \tY$, 
so that~$E$ is a quadric and~$F \in \Cl(E)$ is a ruling.
\end{alenumerate}
Note that in any case~$\tY$ is a smooth projective threefold, $\sigma_*\cO_{\tY} \cong \sigma_*\cO_{\tY}(E) \cong \cO_Y$, 
and the derived pushforward of~$\cO_{\tY}(-E)$ is the ideal sheaf of~$C$ or~$y_0$.
\end{setup}

\begin{remark}
\label{rem:branching}
In the case~$d = 2$ the curve~$C \subset Y$ is a line;
if it is contained in the ramification divisor of~$\varphi_2 \colon Y \to \P^3$, 
we say it is a {\sf ramification line}.
In the case~$d = 1$ the point~$y_0 \in Y$ is distinct from the base point of~$Y$ (\cite[Proposition~2.2(ii)]{KP22}) 
and contained in the ramification divisor of~$\widehat\varphi_1 \colon Y \to \P(1,1,1,2)$.
\end{remark}

\begin{remark}
\label{rem:factorial}
Any del Pezzo variety~$Y$ as above is locally factorial:
for~$2 \le d \le 5$ this follows from smoothness of~$Y$, 
and for~$d = 1$ this is proved in~\cite[Corollary~2.5]{Pavic-Shinder-delPezzo} or~\cite[Corollary~B.4]{KP23}
(in the case of a cusp, see~\cite[(2.8) and Proposition 3.6]{Kalck-Pavic-Shinder} or~\cite[Remark~2.2]{KP23}).
Thus, the class group of Weil divisors~$\Cl(Y)$ is generated by~$H$
and the degree~$D\cdot H^2$ of any surface in~$D \subset Y$ is divisible by~$d$.
\end{remark}

\begin{lemma}
\label{lemma:curve-c-l}
Assume Setup~\textup{\ref{setup}}.
The linear system~$|H-E|$ on~$\tY$ is a pencil, its base locus 
\begin{equation}
\label{eq:def-tl}
\tL \coloneqq \Bs(|H - E|) \subset \tY
\end{equation} 
is a smooth rational curve such that
\begin{equation}
\label{eq:intersection-ty}
H \cdot \tL = 1,
\qquad 
E \cdot \tL = 2,
\end{equation} 
and
\begin{equation}
\label{eq:cn-tl-ty}
\cN_{\tL/\tY} \cong \cO_\tL(-1)^{\oplus 2},
\end{equation}
Moreover, 
\begin{itemize}
\item 
if~$d \ge 3$ or~$d = 2$ and~$C$ is not a ramification line
the map~$\sigma \colon \tL \to \sigma(\tL)$ is an isomorphism onto a line~$\sigma(\tL) \subset Y$ distinct from~$C$,
the scheme~$E \cap \tL$ has length~$2$, it is not contained in a fiber of~$E \to C$, 
and the scheme~$C \cap \sigma(\tL)$ has length~$2$ as well;
\item 
if~$d = 2$ and~$C$ is a ramification line
then~$\tL$ is the exceptional section of the $\P^1$-bundle~$E \to C$ 
and the map~$\sigma \colon \tL \to \sigma(\tL)$ is an isomorphism onto the line~$\sigma(\tL) = C$;

\item 
if~$d = 1$ the scheme~$E \cap \tL$ has length~$2$, it is not contained in a ruling of~$E$,
the map~$\sigma \colon \tL \to \sigma(\tL)$ is the normalization morphism,
and~$\sigma(\tL) \subset Y$ is the fiber of the elliptic fibration~$\varphi_1 \colon Y \dashrightarrow \P^2$
with singularity at the singular point~$y_0$ of~$Y$.
\end{itemize}
\end{lemma}

\begin{proof}
In case~\ref{setup}\ref{it:curve-d-big} we consider the linear subsystem~$|H - C| \subset |H|$ and prove that
\begin{equation*}
\dim|H - C| = 1.
\end{equation*} 
Indeed, if~$2 \le d \le 4$ the image~$\varphi_d(C) \subset \P^{d+1}$  of~$C$ under the morphism~\eqref{eq:phi-d} 
is a smooth rational curve of degree~$1 \le d - 1 \le 3$, 
it spans a~$\P^{d - 1} \subset \P^{d + 1}$, hence~$\dim|H - C| = 1$.
On the other hand,
if~$d = 5$ the image~$\varphi_5(C) \subset \P^6$ is a smooth rational quartic curve.
If it only spans a~$\P^3$, then there is a unique quadric surface in this~$\P^3$ containing~$\varphi_5(C)$
(otherwise, $C$ would be contained in a curve of degree~$4$ and arithmetic genus~$1$, which is absurd),
and since~$Y$ is an intersection of quadrics, this quadric surface must be contained in~$Y$
which is impossible by Remark~\ref{rem:factorial}.
Therefore, $\varphi_5(C)$ spans a~$\P^4$ and~$\dim|H - C| = 1$.

Moreover, in all these cases the base locus~$\Bs(|H - C|) = Y \cap \langle C \rangle$ contains no divisorial components
(again by Remark~\ref{rem:factorial}), hence it is a local complete intersection curve.
Since on the other hand, its degree equals~$d$ and it contains the curve~$C$ of degree~$d - 1$,
it is equal to the union of~$C$ and an extra line~$L$, or the curve~$C$ with multiplicity~$2$ 
(this is only possible if~$d = 2$ and~$C \subset Y$ is a ramification line).
It also follows that the dimension of the tangent space to~$\Bs(|H - C|)$ at any point does not exceed~2;
indeed, if~$d = 2$ then~$\Bs(|H - C|)$ is the preimage of a line in~$\P^3$ under the double covering~$\varphi_2 \colon Y \to \P^3$,
which factors as the composition~$Y \hookrightarrow \P(1,1,1,1,2) \dashrightarrow \P^3$,
hence~$\Bs(|H - C|)$ is contained in the smooth locus of the surface~$\P(1,1,2) \subset \P(1,1,1,1,2)$,
and if~$d > 2$ this follows from the fact that~$\Bs(|H - C|)$ 
is a local complete intersection curve with two smooth components.

Now consider the linear system~$|H - E|$ on the blowup~$\tY$;
the morphism~$\sigma$ induces its isomorphism onto~$|H - C|$, so it is a pencil.
Note that~$\Bs(|H - E|)$ does not contain any fiber of~$E \to C$, 
because the dimension of the tangent space to~$\Bs(|H - C|)$ at any point of~$C$ does not exceed~2;
in particular~\mbox{$E \not\subset \Bs(|H - E|)$}, and therefore
the scheme~$\tL \coloneqq \Bs(|H - E|)$ has no divisorial components.
Moreover, it follows that~$\tL$ is a local complete intersection curve, 
the restriction of~$\sigma$ to~$\tL$ is finite, and
\begin{equation*}
[\tL] = (H - E)^2
\end{equation*}
in the Chow group~$\CH^2(\tY)$.
Standard intersection theory (see, e.g., \cite[Lemma~4.1.2]{IP}) gives
\begin{equation}
\label{eq:he-intersections}
H^3 = d,
\qquad 
H^2\cdot E = 0,
\qquad 
H \cdot E^2 = 1 - d,
\qquad 
E^3 = 4 - 2d,
\end{equation} 
and~\eqref{eq:intersection-ty} follows.
In particular, since~$H \cdot \tL = 1$ and~$|H|$ is base point free, 
it follows that~$\tL$  is irreducible and generically reduced, 
and since it is a local complete intersection, it is everywhere reduced.
Moreover, the image~$\varphi_d(\sigma(\tL)) \subset \P^{d+1}$ must be a line, 
and therefore~$\sigma(\tL) \subset Y$ is also a line,
and the map~$\sigma \colon \tL \to \sigma(\tL)$ is finite of degree~$1$, hence it must be an isomorphism;
in particular, $\tL$ is a smooth rational curve.

Next, we consider the Koszul complex
\begin{equation}
\label{eq:tl-koszul-untwisted}
0 \to \cO_{\tY}(2E-2H) \to \cO_{\tY}(E-H)^{\oplus 2} \to \cO_{\tY} \to \cO_\tL \to 0.
\end{equation}
Restricting it to~$\tL$ and using~\eqref{eq:intersection-ty}, we deduce~\eqref{eq:cn-tl-ty}.

Finally, if~$\tL$ is not contained in~$E$ it follows from~\eqref{eq:intersection-ty}
that the intersection~$E \cap \tL$ is a scheme of length~$2$.
This scheme is not contained in a fiber of~$E \to C$ 
because the projection~$\tL \to \sigma(\tL)$ is an isomorphism.
Therefore, $C \cap \sigma(\tL)$ is a scheme of length~$2$ as well.
On the other hand, if~$\tL \subset E$ then~$\sigma(\tL) = C$, 
hence~$d - 1 = H \cdot C = H \cdot \tL = 1$ by definition of~$C$ and~\eqref{eq:intersection-ty}, respectively, hence~$d = 2$. 
Moreover, since~$Y$ in this case is a quartic double solid and~$C$ is a line, the base locus of~$|H - C|$ 
is the union of~$C$ and its image under the involution of the double covering~$\varphi_2 \colon Y \to \P^3$,
therefore the equality~$\sigma(\tL) = C$ means that~$C$ is fixed by the involution, hence it is a ramification line.
Finally, it follows from~\eqref{eq:intersection-ty} in this case
that~$\tL$ is the exceptional section of the $\P^1$-bundle $E \to C$.

Now consider case~\ref{setup}\ref{it:curve-d=1}.
Since the singular point~$y_0 \in Y$ is distinct 
from the base point of the two-dimensional linear system~$|H|$ (see Remark~\ref{rem:branching}),
it follows that
\begin{equation*}
\dim |H - y_0| = 1
\qquad\text{and}\qquad 
|H - 2y_0| = \varnothing.
\end{equation*}
Now consider the linear system~$|H - E|$ on the blowup~$\tY$;
the morphism~$\sigma$ induces its isomorphism onto~$|H - y_0|$, so it is a pencil,
and its base locus~$\tL$ does not contain~$E$.
As before, $\tL$ is a local complete intersection curve, hence~$\cO_\tL$ has the Koszul resolution~\eqref{eq:tl-koszul-untwisted}.
Moreover, the equalities~\eqref{eq:he-intersections} still hold true (see, e.g., \cite[Lemma~4.1.6]{IP}), 
hence~\eqref{eq:intersection-ty} follows.
In particular, $H \cdot \tL = 1$,
and therefore~$\tL$ has a unique horizontal component~$\tL_0$, 
the map~\mbox{$\tL_0 \to \sigma(\tL_0)$} is finite of degree~$1$,
and the image~$\sigma(\tL) = \sigma(\tL_0)$ is a fiber 
of the elliptic fibration~$\varphi_1 \colon Y \dashrightarrow \P^2$.
Thus, $\sigma(\tL_0)$ is a curve of arithmetic genus~$1$.

On the other hand, 
$E$ is an irreducible quadric surface and~$\cO_E(-E)$ is its hyperplane class,
hence the divisor~$E$ has negative intersection with any curve in~$E$.
Therefore, $E \cdot \tL_0 \ge E \cdot \tL = 2$, hence the fiber of~$\tL_0 \to \sigma(\tL_0)$
over the point~$y_0 \in \sigma(\tL_0)$ has length at least~$2$.
This implies that~$y_0$ is a singular point of~$\sigma(\tL_0)$, $\tL_0$ is a smooth rational curve, 
the map~$\tL_0 \to \sigma(\tL_0)$ is the normalization morphism,
the length of the fiber over~$y_0$ is~$2$, 
and~$E \cdot \tL_0 = E \cdot \tL$, hence~$\tL$ has no vertical components.
Moreover, the length~$2$ scheme~$E \cap \tL$ is not contained in a ruling of~$E$
because it is an intersection in~$E$ of two divisors equivalent to~$(H - E)\vert_E$
and~$\cO_\tY(H-E)$ restricts to~$E$ as the hyperplane class.

It only remains to prove~\eqref{eq:cn-tl-ty}.
For this we again restrict the Koszul complex~\eqref{eq:tl-koszul-untwisted} to~$\tL$ and use~\eqref{eq:intersection-ty}.
\end{proof}

In what follows we will often use the notation introduced in Lemma~\ref{lemma:curve-c-l}.
Furthermore, we consider the following two classes in the group of $1$-cycles on~$\tY$
modulo numerical equivalence:
\begin{itemize}
\item 
$\tell$ --- the class of the curve~$\tL$ defined by~\eqref{eq:def-tl};
\item 
$\ell_E$ --- the class of a fiber of~$E \to C$ if~$d \ge 2$,
or the class of a ruling of~$E$ if~$d = 1$.
\end{itemize}
Note that in the latter case, if~$y_0 \in Y$ is a node, so that~$E \cong \P^1 \times \P^1$,
the classes of the two rulings are numerically equivalent
because~$Y$ is locally factorial, see Remark~\ref{rem:factorial}.

\begin{proposition}
\label{prop:x-y}
Assume Setup~\textup{\ref{setup}}.
\begin{renumerate}
\item 
\label{item:cones-ty} 
The nef cone of~$\tY$ is generated by~$H$ and~$2H-E$
and the Mori cone is generated by~$\tell$ and~$\ell_E$.

\item 
\label{item:eff-cone-ty} 
The effective cone of~$\tY$ is generated by~$E$ and~$H-E$.
\item
\label{item:map-to-X} 
The linear system~$|2H - E|$ on~$\tY$ is base point free and defines a small birational contraction 
\begin{equation*}
\pi \colon \tY \to X
\end{equation*}
onto a $1$-nodal prime Fano threefold~$X$ of genus~$g = 2d + 2$.
The exceptional locus of~$\pi$ is the smooth rational curve~$\tL \subset \tY$ defined in~\eqref{eq:def-tl},
and~$x_0 \coloneqq \pi(\tL)$ is the node of~$X$.
\end{renumerate}
\end{proposition}

\begin{proof}
First, we show that the sheaf~$\cO_{\tY}(2H-E)$ on~$\tY$ is globally generated.
For this we twist the Koszul complex~\eqref{eq:tl-koszul-untwisted} by~$\cO_\tY(2H-E)$, 
and using~\eqref{eq:intersection-ty} we obtain the following exact sequence:
\begin{equation}
\label{eq:tl-koszul}
0 \to \cO_{\tY}(E) \to \cO_{\tY}(H)^{\oplus 2} \to \cO_{\tY}(2H-E) \to \cO_\tL \to 0.
\end{equation}
The sheaves~$\cO_{\tY}(E)$, $\cO_{\tY}(H)$, and~$\cO_\tL$ have no higher cohomology, 
hence the same is true for~$\cO_{\tY}(2H-E)$.
It also follows that
\begin{equation}
\label{eq:dim-h0-jc2h}
\dim \rH^0(\tY,\cO_\tY(2H-E)) = 2(d+2) - 1 + 1 = 2d + 4,
\end{equation}
and~\eqref{eq:tl-koszul} induces a long exact sequence of global sections.
Therefore, there is a commutative diagram
\begin{equation*}
\xymatrix@C=1.4em{
0 \ar[r] & 
\rH^0(\cO_\tY(E)) \otimes \cO_\tY \ar[r] \ar[d] & 
\rH^0(\cO_\tY(H))^{\oplus 2} \otimes \cO_\tY \ar[r] \ar[d] & 
\rH^0(\cO_\tY(2H-E)) \otimes \cO_\tY \ar[r] \ar[d] & 
\rH^0(\cO_\tL) \otimes \cO_\tY \ar[r] \ar[d] & 
0
\\
0 \ar[r] & 
\cO_\tY(E) \ar[r] & 
\cO_\tY(H)^{\oplus 2} \ar[r] & 
\cO_\tY(2H-E) \ar[r] & 
\cO_\tL \ar[r] & 
0,
}
\end{equation*}
where the vertical arrows are given by evaluation and the rows are exact.

If~$d \ge 2$ the sheaves~$\cO_\tY(H)$ and~$\cO_\tL$ are globally generated,
hence the second and fourth vertical arrows are surjective, 
hence so is the third arrow, i.e., $\cO_\tY(2H - E)$ is globally generated.

If~$d = 1$ the above argument works as well over the complement of the base point of~$|H|$, therefore
\begin{equation*}
\Bs(|2H - E|) \subset \Bs(|H|).
\end{equation*}
On the other hand, the linear system~$|2H|$ 
defines the regular double covering~$\widehat{\varphi}_1 \colon Y \to \P(1,1,1,2)$ 
(see Remark~\ref{rem:2h})
and the point~$y_0$ lies on its ramification divisor (see Remark~\ref{rem:branching}), hence we have
\begin{equation*}
\Bs(|2H - E|) \subset E.
\end{equation*}
The right sides of the inclusions are disjoint (by Remark~\ref{rem:branching}), hence~$\cO_\tY(2H-E)$ is globally generated.

\ref{item:cones-ty}
Now we describe the nef cone and the Mori cone of~$\tY$.
Since $\Cl(Y) = \ZZ \cdot H$ (see Remark~\ref{rem:factorial})
the group~$\Pic(\tY) = \Cl(\tY)$ is generated by~$H$ and~$E$.
Since~$\cO_\tY(2H - E)$ is globally generated, it is nef.
Moreover, the line bundle~$\cO_\tY(H)$ is the pullback of an ample line bundle from~$Y$, hence it is also nef.
On the other hand, 
using~\eqref{eq:intersection-ty}, we compute
\begin{equation}
\label{eq:int-ty}
\begin{aligned}
H \cdot \ell_E &= 0,
\qquad\qquad&
H \cdot \tell &= 1,
\\
(2H - E) \cdot \ell_E &= 1,
&
(2H - E) \cdot \tell &= 0.
\end{aligned}
\end{equation} 
Since both~$\ell_E$ and~$\tell$ are effective curve classes, we conclude that~$H$ and~$2H - E$ 
generate the nef cone of~$\tY$, while~$\ell_E$ and~$\tell$ generate the Mori cone.

\ref{item:eff-cone-ty}
Assume~$D = aH + bE$ is an effective divisor.
It is enough to show that~$a \ge 0$ and~$a + b \ge 0$.
The first follows immediately because~$\sigma(D) \sim aH$ is also effective.
For the second, using~\eqref{eq:he-intersections} we compute
\begin{equation*}
D \cdot (2H - E) \cdot (H - E) = (a + b)(d + 1),
\end{equation*}
and since~$H - E$ has no fixed components and~$2H - E$ is nef, this must be nonnegative, hence~$a + b \ge 0$.

\ref{item:map-to-X}
We already checked that the linear system~$|2H - E|$ is base point free and has dimension~$2d + 3$.
Now consider the morphism~$\tY \to \P^{2d+3}$ induced by this linear system, and its Stein factorization
\begin{equation}
\label{eq:stein}
\tY \xrightarrow{\ \pi\ } X \xrightarrow\quad \P^{2d+3},
\end{equation}
where the morphism~$\pi$ has connected fibers, $X$ is normal, and the morphism~$X \to \P^{2d+3}$ is finite.

Note that any curve in a fiber of~$\pi$ has zero intersection with~$2H - E$;
therefore the description of the Mori cone in~\ref{item:cones-ty} implies that 
its class is a positive multiple of the class of~$\tL$.
Since~$(H - E) \cdot \tL = -1$ by~\eqref{eq:intersection-ty}, 
it follows that any such curve has negative intersection with~$H - E$,
hence it is contained in the base locus of the linear system $|H - E|$ 
which, as we showed in Lemma~\ref{lemma:curve-c-l}, equals~$\tL$.
Thus, $\tL$ is the only curve contracted by~$\pi$; in particular~$\pi$ is birational and small.

Since~$\tY$ is smooth, $X$ is smooth away from the point~$x_0 \coloneqq \pi(\tL)$.
On the other hand, \eqref{eq:cn-tl-ty} implies that the point~$x_0 \in X$ is an ordinary double point, so~$X$ is 1-nodal.

Furthermore, by~\eqref{eq:intersection-ty} the subgroup in~$\Pic(\tY)$ of classes restricting trivially to~$\tL$ 
is generated by the anticanonical class~$2H - E$ of~$\tY$,
and since~$\pi$ is small, it is the pullback of the anticanonical class of~$X$, 
which therefore generates~$\Pic(X)$.
By definition of~$\pi$, this class is the pullback of the hyperplane class of~$\P^{2d+3}$
under the finite morphism in~\eqref{eq:stein},
hence it is ample, hence~$X$ is a prime Fano threefold.
Finally,
\begin{equation*}
\rH^0(X, \cO_X(-K_X)) =
\rH^0(\tY, \cO_{\tY}(2H - E)) = 
2d + 4,
\end{equation*}
where we used~\eqref{eq:dim-h0-jc2h} in the second equality,
so it follows that the genus of~$X$ is~$2d + 2$.
\end{proof}

In what follows we refer to the construction of Proposition~\ref{prop:x-y} as {\sf the bridge construction},
and to the variety~$X$ constructed in Proposition~\ref{prop:x-y}
as the {\sf anticanonical model} of~$\Bl_C(Y)$, and use for it the notation
\begin{equation}
\label{eq:def-x}
X \coloneqq 
\begin{cases}
\Bl_C(Y)_\can, & \text{if~$d \ge 2$,}\\
\Bl_{y_0}(Y)_\can, & \text{if~$d = 1$.}
\end{cases}
\end{equation}

\begin{remark}\label{rem:nonfact}
By construction, $X$ is not factorial;
indeed, $\Pic(X) = \Z \cdot K_X$ while~$\Cl(X) = \Cl(\tY) \cong \Z^2$.
Furthermore, since~$H \cdot \tL = 1$, it follows that~$X$ is \emph{maximally nonfactorial}, 
see~\S\ref{sec:mnf} or~\cite[Lemma~6.14]{KS22:abs}, 
that is the map from $\Cl(X)$ to the direct sum of local class groups of the singular points is surjective.
This is not a coincidence: in Proposition~\ref{prop:mnf} we show that every nodal Fano threefold such that 
\begin{equation*}
\rank(\Cl(X)) = \rank(\Pic(X)) + |\Sing(X)|
\end{equation*}
is maximally nonfactorial.
This is important because maximal nonfactoriality is a necessary condition for
the categorical absorption of singularities~\cite[Proposition~6.12]{KS22:abs},
which is crucial for our applications.
\end{remark}

\begin{remark}
\label{rem:hyperelliptic-trigonal}
For~$d \ge 3$ one can check that the anticanonical class of~$X$ is very ample,
so that the second arrow in~\eqref{eq:stein} is a closed embedding, 
and its image is an intersection of quadrics, see, e.g., \cite[Theorem~4.5]{Prokhorov-GorensteinRational} or~\cite[Corollary~4.11]{KP23}.
On the other hand, one can check that if~$d = 1$ then~$X$ is a hyperelliptic threefold of type~$H_5$ from~\cite{PCS}
(see~\cite[Example~4.3]{Prokhorov-GorensteinRational} or~\cite[Proposition~4.4(iv)]{KP23}),
and if~$d = 2$ then~$X$ is a trigonal threefold of type~$T_7$ from~\cite{PCS}
(see~\cite[Example~4.7]{Prokhorov-GorensteinRational} or~\cite[Proposition~4.7(iv)]{KP23}).
\end{remark}

\begin{remark}
\label{rem:sarkisov}
One can extend the blowup~$\sigma \colon \tY \to Y$ and the small contraction~$\pi \colon \tY \to X$ 
to a Sarkisov link by flopping the curve~$\tL$.
By~\cite{Tak09} (or~\cite[Table~2]{KP23}) the other extremal contraction 
is a del Pezzo fibration~$\tY^+ \to \P^1$ of degree~$d + 1$, 
so that we have the following diagram
\begin{equation}
\label{eq:sl}
\vcenter{\xymatrix{
&& \tX \ar[dl] \ar[dr]
\\
& \tY \ar[dl]_\sigma \ar[dr]^\pi \ar@{-->}[rr]^\psi &&
\tY^+ \ar[dl]_{\pi^+} \ar[dr]^{\sigma^+}
\\
Y &&
X &&
\P^1,
}}
\end{equation}
where the map~$\tX \to \tY$ is the blowup of~$\tL$, the map~$\tX \to X$ is the blowup of the node~$x_0 \in X$, 
and~$\psi$ is the flop of~$\tL$.
Moreover, if~$D$ is the exceptional divisor of~$\tX$ over~$X$, 
it is not hard to check that the nef cone of~$\tX$ 
is generated by the pullbacks~$H$, $2H-E$, and~$H-E-D$ of the ample generators of~$\Pic(Y)$, $\Pic(X)$, and~$\Pic(\P^1)$, respectively.
In particular, the three rays of the nef cone give contractions of different type, 
which allows one to reconstruct the entire diagram~\eqref{eq:sl} from~$X$ by blowing up the node and running the minimal model program.
We will use this idea in the proof of Theorem~\ref{thm:tmfy-bmfx-d2}.
\end{remark}

We will also need the following observation regarding the Hilbert scheme~$\rF_{d-1}(Y)$
of curves with Hilbert polynomial~$p(t) = (d-1)t + 1$ on a del Pezzo threefold~$Y$.

\begin{lemma}
\label{lem:hilbert-smooth}
Let~$Y$ be a smooth del Pezzo threefold of degree~$2 \le d \le 5$.
\begin{renumerate}
\item\label{it:Hilb-nonenmpty} 
The open subset~$\rF_{d-1}^\circ(Y) \subset \rF_{d-1}(Y)$ parameterizing smooth rational curves of degree~$d - 1$ 
is nonempty and connected.

\item\label{it:Hilb-smooth}
Assume~$C \subset Y$ is a smooth rational curve of degree~$d - 1$ 
which is not a ramification line when~$d = 2$.
Then
\begin{equation}
\label{eq:h-cn-c-y}
\rH^1(C,\cN_{C/Y}) = 0
\qquad\text{and}\qquad
\dim \rH^0(C,\cN_{C/Y}) = 2d - 2.
\end{equation} 
In particular, the Hilbert scheme~$\rF_{d-1}(Y)$ is smooth of dimension~$2d - 2$ at~$[C]$.
\end{renumerate}
\end{lemma}

\begin{proof}
\ref{it:Hilb-nonenmpty}
If~$d = 2$ then~$\rF_{d-1}^\circ(Y) = \rF_{d-1}(Y)$ is the Hilbert scheme of lines;
connectedness (and nonemptiness) of this Hilbert scheme is proved in~\cite[Theorem~3.57 and Remark~3.58]{Welters}.

Assume~$d \ge 3$.
In this case we apply a construction converse to that of Lemma~\ref{lemma:curve-c-l}.
More precisely, we consider a general hyperplane section~$S \subset Y$ 
(this is a smooth del Pezzo surface of degree~$d$),
choose any line~$L \subset S$, and a general curve in the linear system~$|H - L|$ on~$S$;
it is easy to see that this is a smooth rational curve of degree~$d - 1$ (see, e.g., \cite[Lemma~4.6]{LSZ}).
This proves that~$\rF_{d-1}^\circ(Y)$
has a structure of a fibration over the Hilbert scheme of lines on~$Y$ 
(which is connected, see~\cite[Proposition~2.2.10]{KPS18}) with fiber over a line~$L$ 
an open dense subset of the Grassmannian~$\Gr(2,d)$ that parameterizes pencils of hyperplanes through~$L$,
hence~$\rF_{d-1}^\circ(Y)$ is nonempty and connected.

\ref{it:Hilb-smooth} We use notation of Lemma~\ref{lemma:curve-c-l}.
Twisting the Koszul complex~\eqref{eq:tl-koszul-untwisted} by~$\cO_\tY(-E)$ and pushing it forward along~$\sigma$, we obtain an exact sequence
\begin{equation}
\label{eq:jc}
0 \to \cO_{Y}(-2H) \to \cO_{Y}(-H)^{\oplus 2} \to \cJ_C \to \sigma_*(\cO_\tL(-E)) \to 0.
\end{equation}
Combining it with~\eqref{eq:intersection-ty} we obtain a right exact sequence
\begin{equation*}
\cO_Y(-H)^{\oplus 2} \to \cJ_C \to \cO_L(-2) \to 0,
\end{equation*}
where we set~$L \coloneqq \sigma(\tL)$.
Restricting it to~$C$ we obtain a right exact sequence
\begin{equation*}
\cO_C(1-d)^{\oplus 2} \to \cN_{C/Y}^\vee \to \cO_{C \cap L} \to 0.
\end{equation*}
where~$C \cap L$ is a scheme of length~$2$ by Lemma~\ref{lemma:curve-c-l}.
Note that~$\cO_C(1-d)^{\oplus 2}$ and~$\cN_{C/Y}^\vee$ are locally free sheaves of rank~$2$,
and the cokernel of the first arrow is a torsion sheaf, 
hence the kernel is a torsion subsheaf of a locally free sheaf, hence it vanishes, and this morphism is injective.
Thus, the above sequence is also left exact.
Twisting it by~$\det(\cN_{C/Y}) \cong \cO_C(2d - 4)$ we obtain
\begin{equation*}
0 \to \cO_C(d - 3)^{\oplus 2} \to \cN_{C/Y} \to \cO_{C \cap L} \to 0.
\end{equation*}
Since~$d \ge 2$, the equalities~\eqref{eq:h-cn-c-y} 
as well as the dimension and smoothness of the Hilbert scheme follow.
\end{proof}

\begin{remark}
\label{rem:f1-sing}
The argument does not work when~$d = 2$ and~$C$ is a ramification line, 
because in this case~$L = C$, so restricting~\eqref{eq:jc} to~$C$ we obtain a right exact sequence
\begin{equation*}
\cO_C(-1)^{\oplus 2} \to \cN_{C/Y}^\vee \to \cO_C(-2) \to 0
\end{equation*}
which implies~$\cN_{C/Y} \cong \cO_C(2) \oplus \cO_C(-2)$.
In particular, the cohomology groups of~$\cN_{C/Y}$ jump, 
hence the point~$[C]$ on the Hilbert scheme is singular.
\end{remark}

In the next lemma we construct an important vector bundle~$\cU_\tY$ of rank~$2$ on~$\tY$.
As we will prove later (see Proposition~\ref{prop:mutations-1}),
the sheaf~$\pi_*(\cU_\tY^\vee)$ is locally free and dual to a Mukai bundle on~$X$.

If~$d = 1$ the construction works both in the nodal and cuspidal cases,
but the nodal case is simpler and as we do not need the cuspidal case for applications, we omit it.

\begin{lemma}
\label{lemma:mutation-cu}
Assume Setup~\textup{\ref{setup}} and if~$d = 1$ assume that~$Y$ is $1$-nodal.
Then there is an exact sequence
\begin{equation}
\label{eq:def-cu}
0 \to \cU_\tY^\vee \to \cO_\tY(H) \oplus \cO_\tY(H) \to \cO_E(d\rf) \to 0,
\end{equation}
where~$F$ is the class of a fiber of~$E \to C$ if~$d \ge 2$, or a ruling on~$E \cong \P^1 \times \P^1$ if~$d = 1$.
It defines a vector bundle~$\cU_\tY$ on~$\tY$ of rank~$2$ with~$\rc_1(\cU_\tY) = E - 2H = K_\tY$.
Moreover, the sheaf~$\pi_*(\cU_\tY^\vee)$ is $(-K_X)$-stable.
\end{lemma}

\begin{proof}
If~$d \ge 2$ the curve~$C$ has degree~$d - 1$, hence~$\cO_\tY(H)\vert_E \cong \cO_E((d-1)F)$,
and if~$d = 1$ the divisor~$E$ is contracted by~$\sigma$ to a point, hence~$\cO_\tY(H)\vert_E \cong \cO_E$.
In any case, we have
\begin{equation*}
\Ext^\bullet(\cO_\tY(H), \cO_E(d\rf)) =
\Ext^\bullet(\cO_E((d-1)\rf), \cO_E(d\rf)) =
\rH^\bullet(E, \cO_E(\rf)) =
\kk^2.
\end{equation*}
The natural evaluation morphism $\cO_\tY(H) \oplus \cO_\tY(H) \to \cO_E(d\rf)$
is surjective because the line bundle~$\cO_E(F)$ is globally generated.
Since~$E$ is a Cartier divisor, the projective dimension of the sheaf~$\cO_E(d\rf)$ on~$\tY$ is~$1$, 
hence the kernel of this epimorphism is a vector bundle of rank~2, and its first Chern class is obvious.

To prove the stability of~$\pi_*(\cU_\tY^\vee)$ consider any saturated subsheaf~$\cF \subset \pi_*(\cU_\tY^\vee)$ of rank~$1$.
The restriction of~$\cF$ to~$X \setminus \{x_0\}$ is reflexive;
considering it as a sheaf on~$\tY \setminus \tL$ and taking its reflexive extension to~$\tY$,
we obtain a reflexive subsheaf~$\tcF \subset \cU_\tY^\vee$ of rank~$1$, 
which is invertible because~$\tY$ is smooth.
Composing the embeddings~$\tcF \hookrightarrow \cU_\tY^\vee \hookrightarrow \cO_\tY(H)^{\oplus 2}$,
we see that the line bundle~$\tcF^\vee(H)$ has global sections.
This line bundle is nontrivial, because~$\rH^\bullet(\tY, \cU_\tY^\vee(-H)) = 0$ by construction,
hence by Proposition~\ref{prop:x-y}\ref{item:eff-cone-ty} we have
\begin{equation*}
\rc_1(\tcF^\vee(H)) = aE + b(H-E),
\qquad 
a,b \ge 0,
\qquad 
(a,b) \ne (0,0).
\end{equation*}
It follows that~$\rc_1(\cF) = H - aE - b(H - E)$, hence~$\rc_1(\cF) \le H - E$ or~$\rc_1(\cF) \le H - (H - E) = E$.
Now using~\eqref{eq:he-intersections} we see that the slope of~$\tcF$ with respect to~$2H - E$ is bounded from above 
by the maximum of
\begin{equation*}
(H - E) \cdot (2H - E)^2 = d + 1
\qquad\text{and}\qquad 
E \cdot (2H - E)^2 = 2d.
\end{equation*}
On the other hand, the slope of~$\cU_\tY^\vee$ is equal to
\begin{equation*}
\tfrac12 (2H - E)^3 = 2d + 1.
\end{equation*}
It follows that the slope of~$\tcF$ is strictly less than the slope of~$\cU_\tY^\vee$,
hence the quadratic term in the reduced Hilbert polynomial of~$\pi_*\tcF$
is less than the quadratic term in the reduced Hilbert polynomial of~$\pi_*(\cU_\tY^\vee)$.
But the sheaves~$\cF$ and~$\pi_*\tcF$ are isomorphic away from the point~$x_0$, 
hence their reduced Hilbert polynomials agree up to a constant,
and hence the reduced Hilbert polynomial of~$\cF$ 
is less than the reduced Hilbert polynomial of~$\pi_*(\cU_\tY^\vee)$.
Therefore, $\pi_*(\cU_\tY^\vee)$ is stable.
\end{proof}

\begin{remark}
\label{rem:ruling}
Note that in the case~$d = 1$ the bundle~$\cU_\tY$ depends on the choice of a ruling of~$E$.
\end{remark}

\section{Derived categories}
\label{sec:component}

In~\S\ref{sec:bridge} we introduced
the bridge construction of 
the 1-nodal prime Fano threefold~$X$
from a del Pezzo threefold~$Y$ with some extra data as in Setup~\ref{setup}.
In this section we relate the component~$\cB_Y$ of the derived category of~$Y$, defined in~\eqref{eq:dby-i2},
and the components~$\cA_{\cX_b}$ of smoothings~$\cX_b$ of~$X$, defined in~\eqref{eq:dbx-i1}.

We start with a remark about the case~$d = 1$.
Recall that in this case~$Y$ is singular, hence the category~$\cB_Y$ defined in~\eqref{eq:dby-i2} is not proper.
On the other hand, $\tY$ is a resolution of singularities of~$Y$, 
so one can find a smooth and proper replacement for~$\cB_Y$ inside~$\Db(\tY)$.
We do this in the next lemma.

In contrast to Lemma~\ref{lemma:mutation-cu}, it is crucial here to assume that~$Y$ is 1-nodal, 
because in the cuspidal case the sheaf~$\cO_E(F)$ is not exceptional;
in fact, in this case~$\Ext^\bullet(\cO_E(F), \cO_E(F)) \cong \kk \oplus \kk[-1] \oplus \kk[-2]$ 
(where we consider~$\cO_E(F)$ as an object of~$\Db(\tY)$).

\begin{lemma}
\label{lem:tcb-y}
Let~$(Y,F)$ be as in Setup~\textup{\ref{setup}\ref{it:curve-d=1}} and~$Y$ is $1$-nodal, 
so that~$F$ is a ruling of the exceptional divisor~$E$ of the blowup~$\tY = \Bl_{y_0}(Y)$ at the node~$y_0 \in Y$.
Then there is a semiorthogonal decomposition
\begin{equation}
\label{eq:def-tcby}
\Db(\tY) = \langle \tcB_Y, \cO_\tY, \cO_\tY(H), \cO_E, \cO_E(\rf) \rangle,
\end{equation} 
where~$\tcB_Y$ is a smooth and proper triangulated category. 
\end{lemma}

\begin{proof}
We apply~\cite[Theorem~5.8]{KS22:abs} 
(see also~\cite[Theorem~1.1]{CGLMMPS} for a slightly different treatment or~\cite[\S4]{K08} for a general approach).
Using~\cite[(43)]{KS22:abs} we obtain a semiorthogonal decomposition
\begin{equation*}
\Db(\tY) = \langle \cO_E(E), \cO_E(E + F), \tcD_Y \rangle,
\end{equation*}
where the category~$\tcD_Y = \{\cF \in \Db(\tY) \mid \cF\vert_E \in \langle \cO_E(-F), \cO_E \rangle \}$
is smooth and proper and contains the image of the fully faithful functor~$\sigma^* \colon \Dp(Y) \to \Db(\tY)$.
Mutating the objects~$\cO_E(E)$ and~$\cO_E(E + F)$ to the right of~$\tcD_Y$
and using the fact that~$-K_\tY = 2H - E$ restricts to~$E$ as~$-E$, 
we obtain a semiorthogonal decomposition
\begin{equation*}
\Db(\tY) = \langle \tcD_Y, \cO_E, \cO_E(F) \rangle.
\end{equation*}
As~$\sigma^*$ is fully faithful, the pair~$(\cO_Y, \cO_Y(H))$ from~\eqref{eq:dby-i2} gives rise to the exceptional pair~$(\cO_\tY, \cO_\tY(H))$ in~$\tcD_Y$
and therefore to the semiorthogonal decomposition~\eqref{eq:def-tcby} defining the subcategory~$\tcB_Y$.
It is admissible in the derived category of a smooth and proper variety~$\tY$, 
hence it is a smooth and proper category.
\end{proof}

\begin{remark}
\label{rem:ccr}
One can make a relation between the categories~$\tcB_Y$ and~$\cB_Y$ more precise;
in fact, using the techniques developed in~\cite{KS22:abs} and~\cite{KS22:hfd}
one can check that the restrictions of the functors~$\sigma_* \colon \Db(\tY) \to \Db(Y)$
and~$\sigma^* \colon \Dp(Y) \to \Db(\tY)$ give functors
\begin{equation*}
\sigma_* \colon \tcB_Y \to \cB_Y
\qquad\text{and}\qquad 
\sigma^* \colon \cB_Y \cap \Dp(Y) \to \tcB_Y
\end{equation*}
such that~$\sigma_*$ is a crepant categorical contraction and~$(\tcB_Y, \sigma^*, \sigma_*)$
is a crepant categorical resolution of~$\cB_Y$.
\end{remark}

For convenience, in the case where~$2 \le d \le 5$  we will write
\begin{equation}
\label{eq:tcby-big-d}
\tcB_Y \coloneqq \sigma^*(\cB_Y) \subset \Db(\tY).
\end{equation}
Thus, in all cases~$\tcB_Y \subset \Db(\tY)$ is a smooth and proper admissible subcategory,
and when~$d \ge 2$ the functors~$\sigma_* \colon \tcB_Y \to \cB_Y$ and~$\sigma^* \colon \cB_Y \to \tcB_Y$ 
are equivalences.

In the following crucial proposition we show that the bundle~$\cU_\tY$ constructed in Lemma~\ref{lemma:mutation-cu}
is a pullback of a Mukai bundle~$\cU_X$ on~$X$ (as defined in Definition~\ref{def:mukai})
and construct a semiorthogonal decomposition of~$\Db(X)$ containing~$\tcB_Y$ as one of components.

Recall that an object~$\rP \in \Db(X)$ is called a~{\sf $\Pinfty{2}$-object} 
if~$\Ext^\bullet(\rP, \rP) = \kk[\uptheta]$ with~$\deg(\uptheta) = 2$,
see~\cite[Definition~2.6 and~Remark 2.7]{KS22:abs}.

\begin{proposition}
\label{prop:mutations-1}
Assume Setup~\textup{\ref{setup}} and if~$d = 1$ assume that~$Y$ is $1$-nodal.
Let~$X$ be the $1$-nodal Fano threefold
constructed in Proposition~\textup{\ref{prop:x-y}} 
and let~$\cU_\tY$ be the vector bundle constructed in Lemma~\textup{\ref{lemma:mutation-cu}}.
Then
\begin{equation}
\label{eq:def-cux}
\cU_\tY \cong \pi^*\cU_X,
\end{equation} 
where~$\cU_X \cong \pi_*\cU_\tY$ is a Mukai bundle on~$X$,
and there is a semiorthogonal decomposition
\begin{equation}
\label{eq:sod-cdo}
\Db(X) = \langle \cP_X, \bcA_X, \cO_X, \cU_X^\vee \rangle,
\end{equation}
where the category~$\cP_X$ is generated by the $\Pinfty{2}$-object~$\rP_X \coloneqq \pi_*(\cO_{\tY}(E-H))$ and
\begin{equation*}
\bcA_X \simeq \tcB_Y.
\end{equation*}
In particular, the category~$\bcA_X$ is smooth and proper.
\end{proposition}

\begin{proof}
If~$d \ge 2$ we apply the blowup formula for~$\sigma \colon \tY = \Bl_C(Y) \to Y$;
combining it with~\eqref{eq:dby-i2} and~\eqref{eq:tcby-big-d}
and the standard semiorthogonal decomposition~$\Db(C) = \langle \cO_C(d-1), \cO_C(d) \rangle$ for the curve~$C \cong \P^1$,
we obtain the semiorthogonal decomposition:
\begin{equation}
\label{eq:dbty-start}
\Db(\tY) = \langle \tcB_Y, \cO_\tY, \cO_\tY(H), \cO_E((d-1)\rf), \cO_E(d\rf) \rangle,
\end{equation}
where~$\tcB_Y$ is defined in~\eqref{eq:tcby-big-d}
and~$\rf$ stands for the class of a fiber of~$E \to C$.
If~$d = 1$ we use the semiorthogonal decomposition~\eqref{eq:def-tcby};
it has exactly the same form as~\eqref{eq:dbty-start}.

Next, we modify~\eqref{eq:dbty-start} by a sequence of mutations.

{\bf Step 1.}
First, mutate the last two objects to the left of~$\cO_\tY(H)$.
The computation similar to that of Lemma~\ref{lemma:mutation-cu} 
shows that~$\Ext^\bullet(\cO_\tY(H), \cO_E((d-1)\rf)) = \kk$, 
hence the mutation of~$\cO_E((d-1)\rf)$ is given by the exact sequence
\begin{equation*}
0 \to \cO_\tY(H-E) \to \cO_\tY(H) \to \cO_E((d-1)\rf) \to 0;
\end{equation*}
in particular, the result of the mutation is the exceptional bundle~$\cO_\tY(H-E)$.
The mutation of~$\cO_\tY(d\rf)$ is described in Lemma~\ref{lemma:mutation-cu} and the result is~$\cU_\tY^\vee$.
Thus, we obtain the semiorthogonal decomposition
\begin{equation}
\label{eq:dbty-second}
\Db(\tY) = \langle \tcB_Y, \cO_\tY, \cO_\tY(H-E), \cU_\tY^\vee, \cO_\tY(H) \rangle.
\end{equation}

{\bf Step 2.}
Next, mutate $\cO_\tY(H-E)$ to the left of~$\cO_\tY$.
Twisting~\eqref{eq:tl-koszul-untwisted} by~$\cO_\tY(H-E)$ and using~\eqref{eq:intersection-ty}
we obtain an exact sequence
\begin{equation}
\label{eq:koszul-tl-twisted-h-e}
0 \to \cO_{\tY}(E-H) \to \cO_{\tY}^{\oplus 2} \to \cO_{\tY}(H-E) \to \cO_\tL(-1) \to 0.
\end{equation}
By~\eqref{eq:dbty-second} the pair~$(\cO_\tY, \cO_\tY(H - E))$ is exceptional, 
hence~$\rH^\bullet(\tY, \cO_\tY(E - H)) = 0$,
and so it follows from~\eqref{eq:koszul-tl-twisted-h-e}
that~$\rH^\bullet(\tY, \cO_\tY(H-E)) = \kk^2$ and that the middle arrow in~\eqref{eq:koszul-tl-twisted-h-e} 
is the evaluation morphism.
Therefore, the result of the mutation is the cone of the middle arrow in~\eqref{eq:koszul-tl-twisted-h-e},
and it also follows that the same object can be represented
as the cone of the morphism~$\cO_\tL(-1)[-1] \to \cO_{\tY}(E-H)[1]$ associated with~\eqref{eq:koszul-tl-twisted-h-e}, 
which itself is considered as Yoneda extension.

On the other hand, Serre duality gives
\begin{equation*}
\Ext^\bullet(\cO_\tY(H - E), \cO_\tY(E - H)) =
\rH^\bullet(\tY, \cO_\tY(2E - 2H)) \cong
\rH^\bullet(\tY, \cO_\tY(-E)[3])^\vee = 0.
\end{equation*}
Combining this with the vanishing~$\rH^\bullet(\tY, \cO_\tY(E - H)) = 0$ proved above,
we deduce from~\eqref{eq:koszul-tl-twisted-h-e} that
\begin{equation*}
\Ext^\bullet(\cO_\tL(-1), \cO_{\tY}(E-H)) \cong 
\Ext^\bullet(\cO_{\tY}(E-H)[2], \cO_{\tY}(E-H)) \cong 
\kk[-2]
\end{equation*}
and that the morphism~$\cO_\tL(-1)[-1] \to \cO_{\tY}(E-H)[1]$ associated with~\eqref{eq:koszul-tl-twisted-h-e} 
is the evaluation morphism. 
Since by~\eqref{eq:cn-tl-ty} the object~$\cO_\tL(-1)$ is spherical, 
we finally conclude that the result of the mutation (up to twist)
is the spherical twist~$\bT_{\cO_\tL(-1)}(\cO_\tY(E - H))$
and we have the triangle
\begin{equation}
\label{eq:twist-triangle}
\cO_\tL(-1)[-2] \to \cO_\tY(E - H) \to \bT_{\cO_\tL(-1)}(\cO_\tY(E - H)).
\end{equation}
Thus, we obtain the semiorthogonal decomposition
\begin{equation*}
\Db(\tY) = \langle \tcB_Y, \bT_{\cO_\tL(-1)}(\cO_\tY(E - H)), \cO_\tY, \cU_\tY^\vee, \cO_\tY(H) \rangle.
\end{equation*}

{\bf Step 3.}
Next, we mutate~$\cO_\tY(H)$ to the far left.
Since~$K_\tY = E - 2H$, we obtain
\begin{equation*}
\Db(\tY) = \langle \cO_\tY(E - H), \tcB_Y, \bT_{\cO_\tL(-1)}(\cO_\tY(E - H)), \cO_\tY, \cU_\tY^\vee \rangle.
\end{equation*}

{\bf Step 4.}
Finally, we mutate~$\tcB_Y$ to the right of~$\bT_{\cO_\tL(-1)}(\cO_\tY(E - H))$:
\begin{equation}
\label{eq:db-ty-2}
\Db(\tY) = \langle 
\cO_\tY(E - H), 
\bT_{\cO_\tL(-1)}(\cO_\tY(E - H)), 
\bR_{\bT_{\cO_\tL(-1)}(\cO_\tY(E - H))}(\tcB_Y), 
\cO_\tY, 
\cU_\tY^\vee 
\rangle,
\end{equation}
where~$\bR_{\bT_{\cO_\tL(-1)}(\cO_\tY(E - H))}$ is the right mutation functor 
through the subcategory~$\bT_{\cO_\tL(-1)}(\cO_\tY(E - H))$.

Now we check that~\eqref{eq:db-ty-2} induces a decomposition of~$\Db(X)$.
Indeed, the functor~$\pi_* \colon \Db(\tY) \to \Db(X)$ is a Verdier localization 
with respect to the subcategory generated by the sheaf~$\cO_\tL(-1)$
(this, e.g., follows from~\cite[Theorem~5.8 and Corollary~5.11]{KS22:abs}), i.e.,
\begin{equation*}
\Db(X) \simeq \Db(\tY) / \cO_\tL(-1).
\end{equation*}
By~\cite[Propositions~5.5 and~6.1]{KS22:hfd} the functors~$\pi_*$ and~$\pi^*$ induce an equivalence 
between the left orthogonal of~$\cO_\tL(-1)$ in~$\Db(\tY)$ and~$\Dp(X)$.
By~\eqref{eq:twist-triangle} and~\eqref{eq:db-ty-2} the bundle~$\cU_\tY^\vee$ belongs to this left orthogonal, 
hence~$\cU_\tY^\vee$ is the pullback of a perfect complex~$\cU_X^\vee$ on~$X$.
Since pullback of perfect complexes commutes with dualization, we obtain~\eqref{eq:def-cux},
and then the projection formula implies~$\cU_X \cong \pi_*\cU_\tY$.

It is clear that~$\cU_X$ is a vector bundle, $\rank(\cU_X) = 2$, and~$\rc_1(\cU_X) = K_X$.
Moreover, $\cU_X$ is $(-K_X)$-stable by Lemma~\ref{lemma:mutation-cu},
and it is acyclic and exceptional 
because~$\pi^*$ is fully faithful and the pair~$(\cO_\tY,\cU_\tY^\vee)$ is exceptional by~\eqref{eq:db-ty-2}.
Thus, $\cU_X$ is a Mukai bundle on~$X$.

Similarly, $\bR_{\bT_{\cO_\tL(-1)}(\cO_\tY(E - H))}(\tcB_Y)$ is left orthogonal to~$\cO_\tL(-1)$, 
hence~$\pi_*$ and~$\pi^*$ induce an equivalence
\begin{equation*}
\bR_{\bT_{\cO_\tL(-1)}(\cO_\tY(E - H))}(\tcB_Y) \simeq \bcA_X 
\end{equation*}
with a subcategory~$\bcA_X \subset \Dp(X) \subset \Db(X)$, which is smooth and proper because~$\tcB_Y$ is.
Finally, applying~\cite[Proposition~4.1]{KS22:abs} we obtain the required decomposition~\eqref{eq:sod-cdo}, where 
\begin{equation*}
\cP_X \coloneqq \langle \cO_\tY(E - H), \bT_{\cO_\tL(-1)}(\cO_\tY(E - H)) \rangle / \cO_\tL(-1).
\end{equation*}
and using~\cite[Theorem~6.17]{KS22:abs}, we deduce that~$\cP_X$ is generated by the object
\begin{equation}
\label{eq:def-rpx}
\rP_X \coloneqq \pi_*(\cO_\tY(E - H)),
\end{equation}
which is a $\Pinfty{2}$-object.
\end{proof}

\begin{remark}
\label{rem:equivalences}
For future reference we note that the equivalence~$\tcB_Y \simeq \bcA_X$ is given by the functors
\begin{equation}
\label{eq:equivalences}
\pi_* \circ \bR_{\bT_{\cO_\tL(-1)}(\cO_\tY(E - H))} \colon \tcB_Y \to \bcA_X
\qquad\text{and}\qquad
\bL_{\bT_{\cO_\tL(-1)}(\cO_\tY(E - H))} \circ \pi^* \colon \bcA_X \to \tcB_Y,
\end{equation}
where~$\bL_{\bT_{\cO_\tL(-1)}(\cO_\tY(E - H))}$ is the left mutation functor,
and when~$d \ge 2$ the equivalence~$\cB_Y \simeq \bcA_X$ is given by their compositions
with~$\sigma^*$ and~$\sigma_*$, respectively.
\end{remark}

\begin{remark}
\label{rem:gr2-embedding}
When~$X$ is a smooth prime Fano threefold of genus~$2d + 2$
and~$\cU_X$ is a Mukai bundle, then the dual $\cU_X^\vee$
is globally generated and induces an embedding~$X \hookrightarrow \Gr(2,d+3)$.
For~$d \le 4$ this leads to a nice description of~$X$ which will be recalled in~\S\ref{sec:stacks}.
One can check that for 1-nodal threefolds constructed in Proposition~\ref{prop:x-y}
the dual Mukai bundle~$\cU_X^\vee$ constructed in Proposition~\ref{prop:mutations-1}
is globally generated if and only if~$d \ge 3$;
it would be interesting to understand how the description of the image of~$X$ in~$\Gr(2,d+3)$ differs from the smooth case.
\end{remark}

All the components of the semiorthogonal decomposition~\eqref{eq:sod-cdo} constructed in Proposition~\ref{prop:mutations-1}
except for~$\cP_X$ are smooth and proper; 
in terminology of~\cite{KS22:abs} this means that~$\cP_X$ \emph{absorbs singularities} of~$X$ (see~\cite[Definition~1.1]{KS22:abs}).
An important property of a subcategory generated by a $\Pinfty2$-object,
called the \emph{universal deformation absorption property} in~\cite[Definition~1.4]{KS22:abs},
is that it disappears in any smoothing of the variety, while its orthogonal complement deforms.
Using this property we prove the main result of this section,
a more precise version of Theorem~\ref{thm:intro-simple} from the Introduction.

Recall that {\sf a smoothing} of~$X$ 
is a flat projective morphism~$f \colon \cX \to B$ to a smooth pointed curve~$(B,o)$
such that the total space~$\cX$ is smooth, the central fiber~$\cX_o$ is isomorphic to~$X$,
and~$f$ is smooth over~$B \setminus \{o\}$.
Note also that by~\cite{Na97} any nodal Fano threefold~$X$ admits a smoothing
such that for each~$b \ne o$ the fiber~$\cX_b$ is a smooth prime Fano threefold with~$\g(\cX_b) = \g(X)$, 
see~Theorem \ref{thm:smoothing} for details.

\begin{theorem}
\label{thm:cax-family}
Assume Setup~\textup{\ref{setup}} and if~$d = 1$ assume that~$Y$ is $1$-nodal.
Let~$X$ be the $1$-nodal Fano threefold constructed in Proposition~\textup{\ref{prop:x-y}} 
and let~$\cU_X$ and~$\rP_X$ be the Mukai bundle and the~$\Pinfty2$-object on~$X$ constructed in Proposition~\textup{\ref{prop:mutations-1}}.

For any smoothing~$f \colon \cX \to B$ of~$X$ 
after appropriate \'etale base change there is a vector bundle~$\cU_\cX$ on~$\cX$ such that
\begin{itemize}
\item 
$\cU_\cX\vert_X \cong \cU_X$,
\item 
$\cU_\cX\vert_{\cX_b}$ is a Mukai bundle for all~$b \ne o$ in~$B$, and
\item 
the object~$\io_*\rP_X \in \Db(\cX)$ is exceptional, 
where~$\io \colon X \hookrightarrow \cX$ is the embedding of the central fiber.
\end{itemize}
Moreover, there is a $B$-linear semiorthogonal decomposition
\begin{equation}
\label{eq:sod:dbcx-cd}
\Db(\cX) = \langle \io_*\rP_X, \bcA_\cX, f^*\Db(B), f^*\Db(B) \otimes \cU_\cX^\vee \rangle,
\end{equation}
where~$\bcA_\cX$ is an admissible $B$-linear subcategory in~$\Db(\cX)$
which is smooth and proper over~$B$.
Finally,
\begin{equation*}
(\bcA_\cX)_o \simeq \tcB_Y
\qquad\text{and}\qquad
(\bcA_\cX)_b \simeq \cA_{\cX_b}
\quad\text{for~$b \ne o$},
\end{equation*}
where~$\tcB_Y$ is defined in~\eqref{eq:def-tcby} for~$d = 1$ and in~\eqref{eq:tcby-big-d} for~$d \ge 2$,
and~$\cA_{\cX_b}$ is defined in~\eqref{eq:dbx-i1}.
\end{theorem}

\begin{proof}
To construct a vector bundle~$\cU_\cX$ we consider the relative moduli space of stable vector bundles on the fibers of~$\cX/B$.
Since~$\cU_X$ is stable (by Proposition~\ref{prop:mutations-1} and Definition~\ref{def:mukai}), 
it corresponds to a point of this relative moduli space over~$o \in B$,
and by~\cite[Corollary~4.5.2]{HL10} the property~$\Ext^\bullet(\cU_X,\cU_X) = \kk$ implies 
that the moduli space is \'etale over~$B$ at the point~$[\cU_X]$.
Therefore, after a base change from~$B$ to a small \'etale neighborhood of~$o \in B$, 
we will have a section of this relative moduli space, i.e., a fiberwise stable vector bundle~$\cU_\cX$ of rank~$2$ on~$\cX$
which restricts to~$\cU_X$ on the central fiber.

Since the Picard sheaf of~$\cX/B$ is locally constant by Corollary~\ref{cor:pic-constant}
and~$\rc_1(\cU_X) = K_X$, it follows that~$\rc_1(\cU_\cX\vert_{\cX_b}) = K_{\cX_b}$ for all~$b \in B$.
Semicontinuity of cohomology implies that, after appropriate shrinking of~$B$, 
we may assume that for each~$b \in B$ the bundle~$\cU_\cX\vert_{\cX_b}$ is acyclic and exceptional;
therefore, it is a Mukai bundle.

The object~$\io_*\rP_X$ is exceptional by~\cite[Theorem~1.8]{KS22:abs}.
The sheaves~$\cO_{\cX}$ and~$\cU_\cX^\vee$ form a relative exceptional pair,
and~\eqref{eq:sod-cdo} implies that they are semiorthogonal to~$\io_*\rP_X$.
Therefore, together they induce the required semiorthogonal decomposition~\eqref{eq:sod:dbcx-cd}.

Since~$\io_*\rP_X$ is supported on the central fiber of~$f \colon \cX \to B$
the base change of the subcategory it generates along the embedding~$b \hookrightarrow B$ for~$b \ne o$ is zero.
Therefore, such a base change results in a semiorthogonal decomposition 
\begin{equation*}
\Db(\cX_b) = \langle (\bcA_\cX)_b, \cO_{\cX_b}, \cU_{\cX_b}^\vee \rangle.
\end{equation*}
Comparing it with~\eqref{eq:dbx-i1} we conclude that~$(\bcA_\cX)_b = \cA_{\cX_b}$.

On the other hand, using~\cite[Theorem~1.5]{KS22:abs} 
we obtain a semiorthogonal decomposition
\begin{equation*}
\Db(\cX_o) = \langle \rP_X, (\bcA_\cX)_o, \cO_X, \cU_X^\vee \rangle
\end{equation*}
and comparing it with~\eqref{eq:sod-cdo}, we conclude that~$(\bcA_\cX)_o = \bcA_X$.
Finally, Proposition~\ref{prop:mutations-1} provides an equivalence~$(\bcA_\cX)_o \simeq \cB_Y$.

Since~$\bcA_{\cX_b}$ is smooth and proper for all points~$b \in B$ (including~$b = o$), 
the category~$\bcA_\cX$ is smooth and proper over~$B$ by~\cite[Theorem~2.10]{K21}.
\end{proof}

Informally, the category~$\bcA_\cX$ provides an interpolation between the component~$\cA_{\cX_b} \subset \Db(\cX_b)$ 
and the component~$\cB_Y \subset \Db(Y)$ if~$d \ge 2$
or its categorical resolution~$\tcB_Y \subset \Db(\tY)$, if~$d = 1$.


\section{Moduli stacks}
\label{sec:stacks}

In this section we study the stacks~$\bMF_{\sY_d}$, $\bMF_{\sX_g}$, and~$\bMFM_{\sX_g}$ 
of del Pezzo threefolds, prime Fano threefolds, and Fano--Mukai pairs, respectively, 
introduced in Definitions~\ref{def:mf} and~\ref{def:mfm} in the Introduction.

To start with, we discuss some properties of Mukai bundles, see Definition~\ref{def:mukai}.
We denote by~$\cM(X;p)$ the coarse moduli space of $(-K_X)$-semistable sheaves on~$X$ 
with Hilbert polynomial~$p(t) \in \QQ[t]$ (computed with respect to~$-K_X$);
this is a projective scheme, see, e.g., \cite[Theorem~4.3.4]{HL10}.

\begin{lemma}
\label{lem:mukai-bundles}
Let~$X$ be a nodal prime Fano threefold of genus~$g \in \{4,6,8,10,12\}$.
For any Mukai bundle~$\cU$ on~$X$ one has
\begin{equation}
\label{eq:mukai-c2-p}
\rc_2(\cU) \cdot K_X = -\tfrac{g+2}2
\qquad\text{and}\qquad 
p_\cU(t) = \pmu(t) \coloneqq \tfrac{2g-2}{3} t^3 + \tfrac{16-g}{6} t.
\end{equation}
Moreover, the subscheme~$\cMuk(X) \subset \cM(X;\pmu)$ 
corresponding to Mukai bundles is an open subscheme.
Finally, $\cMuk(X)$ is a finite reduced scheme.
\end{lemma}

\begin{proof}
Assume~$\cU$ is a Mukai bundle on a nodal prime Fano threefold~$X$ 
of genus~$g \in \{4,6,8,10,12\}$.
Let~$S \subset X$ be a general anticanonical divisor;
then~$S$ is a smooth K3 surface
(when~$X$ is smooth, this is a result of Shokurov~\cite[Theorem~1.2]{Shokurov} and Reid~\cite[Theorem~0.5]{Reid},
and in the nodal case we can use~\cite[Theorem~1]{Mel} because~$g \ge 4$).
Consider the standard exact sequence
\begin{equation*}
0 \to \cU^\vee \otimes \cU \otimes \omega_X \to \cU^\vee \otimes \cU \to \cU\vert_S^\vee \otimes \cU\vert_S \to 0.
\end{equation*}
Using exceptionality of~$\cU$ and Serre duality on~$X$, we deduce that
\begin{equation}
\label{eq:us-simple}
\Ext^\bullet(\cU\vert_S, \cU\vert_S) = 
\rH^\bullet(S, \cU\vert_S^\vee \otimes \cU\vert_S) = 
\kk \oplus \kk[-2],
\end{equation}
hence~$\upchi(\cU\vert_S, \cU\vert_S) = 2$.
Since~$\rc_1(\cU) = K_X$ by Definition~\ref{def:mukai}, we have $\rc_1(\cU\vert_S)^2 = K_X^2 \cdot (-K_X) = 2g - 2$, 
hence the Riemann--Roch theorem on~$S$ applied to~$\cU\vert_S^\vee \otimes \cU\vert_S$ gives
\begin{equation*}
\rc_2(\cU\vert_S) = \tfrac{g + 2}{2} 
\qquad\text{and}\qquad 
p_{\cU\vert_S}(t) = (2g-2)(t^2 - t) + \tfrac{g+4}2.
\end{equation*}
Since~$p_{\cU\vert_S}(t) = p_\cU(t) - p_{\cU}(t-1)$, 
these formulas imply~\eqref{eq:mukai-c2-p} up to constant term of~$p_{\cU}(t)$.
On the other hand, the constant term must be zero because~$\cU$ is acyclic by Definition~\ref{def:mukai}.

What we have proved so far implies that~$\cMuk(X) \subset \cM(X;\pmu)$.
This is an open embedding because stability and local freeness of~$\cU$
and the conditions~$\rH^\bullet(X,\cU) = 0$ and~$\Ext^\bullet(\cU,\cU) = \kk$
defining~$\cMuk(X)$ are open properties.
Furthermore, first order deformations of~$\cU$ are classified by the space~$\Ext^1(\cU,\cU)$
and the obstruction space is~$\Ext^2(\cU,\cU)$ (see, e.g., ~\cite[\S2.A.6 and Corollary~4.5.2]{HL10});
since~$\cU$ is exceptional, both spaces vanish, 
hence both schemes~$\cMuk(X)$ and~$\cM(X;\pmu)$ are smooth and zero dimensional at~$[\cU]$.
This means that~$\cMuk(X)$ is a union of reduced isolated points of~$\cM(X;\pmu)$,
and since~$\cM(X;\pmu)$ is projective, it follows 
that~$\cMuk(X)$ is a finite reduced scheme.
\end{proof}

In the next proposition we identify the scheme~$\cMuk(X)$ in the cases where~$X$ is smooth.
Recall that a smooth prime Fano threefold of genus~4 is a complete intersection of a quadric and a cubic in~$\P^5$;
moreover, the quadric passing through~$X$ is unique (and thus canonically defined by~$X$) and has corank~$0$ or~$1$.
We denote this quadric by~$Q(X)$.

\begin{proposition}
\label{prop:mukai-bundles}
Let $X$ be a smooth prime Fano threefold of genus~$g \in \{4,6,8,10,12\}$.
Then
\begin{equation*}
\cMuk(X) \cong
\begin{cases}
\Spec(\kk), & \text{if~$g \in \{6,8,10,12\}$},\\
\Spec(\kk) \sqcup \Spec(\kk), & \text{if~$g = 4$ and~$Q(X)$ is smooth,}\\
\varnothing, & \text{if~$g = 4$ and~$Q(X)$ is singular.}
\end{cases}
\end{equation*}
\end{proposition}

\begin{proof}
If~$g \in \{6,8,10,12\}$ this is~\cite[Theorem~1.1]{BKM},
see also~\cite[Theorem~B.1.1, Proposition~B.1.5, and Lemma~B.1.9]{KPS18}.

Now let~$g = 4$.
First, assume the quadric~$Q(X)$ is smooth.
Let~$\cS$ be one of the two spinor bundles on~$Q(X)$ (see~\cite{Ottaviani}).
Then the restriction~$\cS\vert_X$ 
is a vector bundle of rank~$2$ with~$\rc_1(\cS\vert_X) = K_X$.
Using the vanishing of cohomology of~$\cS$ and~$\cS(-3)$ (see~\cite[Theorems~2.1 and~2.3]{Ottaviani})
and the Koszul resolution
\begin{equation*}
0 \to \cS(-3) \to \cS \to i_*(\cS\vert_X) \to 0,
\end{equation*}
where~$i \colon X \hookrightarrow Q(X)$ is the embedding, we deduce that~$\cS\vert_X$ is acyclic.
Similarly, using the vanishing of~$\Ext^\bullet(\cS,\cS(-3))$ and exceptionality of~$\cS$ on~$Q(X)$,
we deduce that~$\Ext^\bullet(\cS, i_*(\cS\vert_X)) \cong \kk$, 
which by adjunction imples that~$\cS\vert_X$ is exceptional.
Finally, since~$\cS\vert_X$ has rank~$2$ and~$\rc_1(\cS\vert_X) = K_X$ and~$X$ is smooth of Picard rank~$1$,
stability of~$\cS\vert_X$ follows from the vanishing of~$\rH^\bullet(X,\cS\vert_X)$ proved above.
We conclude that~$\cS\vert_X$ is a Mukai bundle.

Similarly, the restriction~$\cS'\vert_X$ of the other spinor bundle to~$X$ is another Mukai bundle, not isomorphic to~$\cS\vert_X$.
Indeed, one has~$\Ext^\bullet(\cS',\cS) = 0$ and~$\Ext^\bullet(\cS',\cS(-3)) = \kk[-3]$,
therefore~$\Ext^\bullet(\cS'\vert_X,\cS\vert_X) = \kk[-2]$.

On the other hand, the argument of~\cite[Proposition~B.1.5]{KPS18} shows that 
if~$\cU$ is a Mukai bundle on $X$
then either~$\cU \cong \cS\vert_X$ or~$\Ext^1(\cS\vert_X(1),\cU) \ne 0$.
If we assume the second then applying the functor~$\Ext^\bullet(-,\cU)$ to the natural exact sequence
(restricted from~$Q(X)$, see~\cite[Theorem~2.8]{Ottaviani})
\begin{equation*}
0 \to \cS'\vert_X \to \cO_X^{\oplus 4} \to \cS\vert_X(1) \to 0
\end{equation*}
and using the acyclicity of~$\cU$ we deduce that~$\Hom(\cS'\vert_X,\cU) \ne 0$.
Since~$\cS'\vert_X$ and~$\cU$ are stable bundles with equal rank and first Chern class, they are isomorphic.
This proves that either~$\cU \cong \cS\vert_X$ or~$\cU \cong \cS'\vert_X$, 
hence the moduli stack~$\cMuk(X)$ has exactly two points: $[\cS\vert_X]$ and~$[\cS'\vert_X]$.

Now assume that the quadric~$Q(X)$ is singular, i.e., $Q(X) \cong \Cone(\bar{Q})$ is the cone over a smooth quadric threefold~$\bar{Q}$.
Since~$X$ is smooth, the cubic hypersurface cutting out~$X \subset Q(X)$ does not pass through the vertex of the cone, 
hence the projection out of the vertex defines a regular triple covering~\mbox{$X \to \bar{Q}$}.
Let~$\cS\vert_X$ be the pullback to~$X$ of the unique spinor bundle of~$\bar{Q}$.
It is stable, acyclic, but not exceptional; 
in fact it is easy to see that
\begin{equation*}
\Ext^\bullet(\cS\vert_X,\cS\vert_X) \cong \kk \oplus \kk[-1] \oplus \kk[-2].
\end{equation*}
Moreover, using the exact sequence (pulled back from~$\bar{Q}$, see~\cite[Theorem~2.8]{Ottaviani})
\begin{equation*}
0 \to \cS\vert_X \to \cO_X^{\oplus 4} \to \cS\vert_X(1) \to 0
\end{equation*}
and the same argument as above, one can check that any Mukai bundle~$\cU$ would have been isomorphic to~$\cS\vert_X$,
but as~$\cS\vert_X$ is not a Mukai bundle, we conclude that~$\cMuk(X) = \varnothing$.
\end{proof}

\begin{remark}
In~\cite{BKM,BKM2} the existence and uniqueness of Mukai bundles is proved for any locally factorial Fano threefolds 
of genus~$g \ge 6$ with terminal Gorenstein singularities.
\end{remark}

Before we pass to the stacks~$\bMF_{\sY_d}$, $\bMF_{\sX_g}$, and~$\bMFM_{\sX_g}$,
we discuss a simple but instructive example of the moduli stacks~$\MQ_n \subset \bMQ_n$ of smooth and nodal $n$-dimensional quadrics 
(with the definition analogous to Definition~\ref{def:mf})
and the moduli stack~$\MQS_n$ of smooth quadrics endowed with a spinor bundle 
(with the definition analogous to Definition~\ref{def:mfm}).
We denote by~$\xi_\sQ \colon \MQS_n \to \MQ_n$ the morphism of stacks that forgets the spinor bundle.

\begin{lemma}
\label{lem:mq-mqs}
The stacks~$\MQ_n$, $\bMQ_n$, and~$\MQS_n$ are smooth irreducible Artin stacks of finite type over~$\kk$,
the substack~$\MQ_n \subset \bMQ_n$ is open and dense, 
and the complement~$\bMQ^{(1)}_{n} \coloneqq \bMQ_n \setminus \MQ_n$ is a Cartier divisor.
Moreover, if~$n$ is even the morphism~$\xi_\sQ \colon \MQS_n \to \MQ_n$ is a finite \'etale covering of degree~$2$,
branched over~$\bMQ^{(1)}_n$.
\end{lemma}

\begin{proof}
Let~$V_{n + 2}$ denote a vector space of dimension~$n + 2$.
Consider the open subsets
\begin{equation*}
\rU_n^0 \subset \rU_n^{\le 1} \subset \P(\Sym^2V_{n+2}^\vee)
\end{equation*}
parameterizing $n$-dimensional quadrics in~$\P(V_{n+2})$ of corank~$0$ (i.e., smooth) 
and corank~$\le 1$ (i.e., nodal), respectively.
Note that the complement~$\P(\Sym^2V_{n+2}^\vee) \setminus \rU_n^0$ is a Cartier divisor of degree~$n + 2$,
and its intersection~$\rU_n^{\le 1} \setminus \rU_n^0$ with~$\rU_n^{\le 1}$ is smooth.
If~$n$ is even, we denote by~$\hrU_n^{\le 1} \to \rU_n^{\le 1}$ and~$\hrU_n^0 \to \rU_n^0$
the flat double covering branched over~$\rU_n^{\le 1} \setminus \rU_n^0$
and its restriction, which is \'etale over~$\rU_n^0$.
Then it is easy to prove the following global quotient representations 
\begin{equation*}
\MQ_n \cong \rU_n^0 / \PGL(V_{n+2}),
\qquad
\bMQ_n \cong \rU_n^{\le 1} / \PGL(V_{n+2})
\qquad\text{and}\qquad 
\MQS_n \cong \hrU_n^0  / \PGL(V_{n+2}).
\end{equation*}
Thus, $\rU_n^{\le 1}$, $\rU_n^0$, and~$\hrU_n^0$ are smooth irreducible atlases 
for~$\bMQ_n$, $\MQ_n$, and~$\MQS_n$, respectively, and all statements of the lemma follow easily.
\end{proof}

Now we finally consider the moduli stacks of Fano threefolds.
When discussing the stacks~$\bMF_{\sX_g}$ of prime Fano threefolds 
we concentrate on the cases where~$g \in \{4,6,8,10,12\}$ that are crucial for this paper,
although the same arguments can be also applied for~$g \in \{2,3,5,7,9\}$.

Recall from~\cite[\S12.2]{IP}, \cite[\S2]{Muk92}, or~\cite[Theorem~1.1]{BKM2}
the following descriptions of smooth prime Fano threefolds~$X$ of genus~$g \in \{4,6,8,10,12\}$
as zero loci of regular global sections of vector bundles:
\begin{itemize}
\item 
if~$g = 12$ then~$X$ is the zero locus of a section 
of the vector bundle~$(\wedge^2\cU_3^\vee)^{\oplus 3}$ on~$\Gr(3,7)$,
\item 
if~$g = 10$ then~$X$ is the zero locus of a section 
of the vector bundle~$\cU_2^\perp(1) \oplus \cO(1)^{\oplus 2}$ on~$\Gr(2,7)$,
\item 
if~$g = 8$ then~$X$ is the zero locus of a section 
of the vector bundle~$\cO(1)^{\oplus 5}$ on~$\Gr(2,6)$,
\item 
if~$g = 6$ then~$X$ is the zero locus of a section 
of the vector bundle~$\cO(1)^{\oplus 3} \oplus \cO(2)$ on~$\CGr(2,5)$,
\item 
if~$g = 4$ then~$X$ is the zero locus of a section 
of the vector bundle~$\cO(2) \oplus \cO(3)$ on~$\P^5$,
\end{itemize}
where~$\cU_k$ stands for the tautological vector bundle of rank~$k$ on the Grassmannian~$\Gr(k,n)$, 
$\cU_k^\perp$ is the dual of the corresponding quotient bundle,
and~$\CGr(2,5)$ is the cone over~$\Gr(2,5)$.
The stacks~$\bMF_{\sX_g}$, $\bMF_{\sY_d}$, and~$\bMFM_{\sX_g}$
were defined in Definitions~\ref{def:mf} and~\ref{def:mfm};
the 1-nodal loci are defined in~\S\ref{subsec:nodal-morphisms}.

\begin{theorem}
\label{thm:mf-x-y}
The stacks~$\bMF_{\sX_g}$ and~$\bMFM_{\sX_g}$ for~$g \in \{4,6,8,10,12\}$
and~$\bMF_{\sY_d}$ for~$d \in \{1,2,3,4,5\}$ 
are smooth irreducible algebraic stacks of finite type over~$\kk$,
their substacks
\begin{equation*}
\MF_{\sX_g} \subset \bMF_{\sX_g},
\qquad
\MFM_{\sX_g} \subset \bMFM_{\sX_g},
\qquad\text{and}\qquad 
\MF_{\sY_d} \subset \bMF_{\sY_d}
\end{equation*}
are open and dense, and~$\bMF_{\sX_g} \setminus \MF_{\sX_g}$, $\bMFM_{\sX_g} \setminus \MFM_{\sX_g}$, 
and~$\bMF_{\sY_d} \setminus \MF_{\sY_d}$ are Cartier divisors,
smooth along the $1$-nodal loci~$\bMF_{\sX_g}^{(1)}$, $\bMFM_{\sX_g}^{(1)}$, and~$\bMF_{\sY_d}^{(1)}$, respectively.

Moreover, the forgetful morphism~$\xi \colon \bMFM_{\sX_g} \to \bMF_{\sX_g}$ 
is \'etale, separated, and representable,
and its fiber over~$[X] \in \bMF_{\sX_g}$ is the moduli space~$\cMuk(X)$ of Mukai bundles on~$X$.
More precisely,
\begin{itemize}
\item 
if~$g \ge 6$ then~$\xi$ is an open embedding
and induces an isomorphism~$\MFM_{\sX_g} \xrightiso{\ \xi\vert_{\MFM_{\sX_g}}\ } \MF_{\sX_g}$, and
\item 
if~$g = 4$ then~$\xi\vert_{\MFM_{\sX_4}}$ is induced by the double covering~$\xi_\sQ$, i.e., there is a Cartesian diagram
\begin{equation}
\label{eq:mfx4-mq4}
\vcenter{\xymatrix{
\MFM_{\sX_4} \ar[r] \ar[d]_{\xi\vert_{\MFM_{\sX_4}}} &
\MQS_4 \ar[d]^{\xi_\sQ}
\\
\MF_{\sX_4}^\circ \ar[r] &
\MQ_4,
}}
\end{equation}
where~$\MF_{\sX_4}^\circ \subset \MF_{\sX_4}$ is the open substack 
parameterizing smooth~$X$ with smooth quadric~$Q(X)$,
and the bottom horizontal arrow is defined by~$[X] \mapsto [Q(X)]$.
\end{itemize}
\end{theorem}

\begin{proof}
Let~$\rM$ be either of the stacks~$\MF_{\sX_g}$ or~$\MF_{\sY_d}$ 
and let~$\brM$ be the corresponding stack~$\bMF_{\sX_g}$ or~$\bMF_{\sY_d}$.

By~\cite[Lemmas~2.4 and~2.5]{JL} the stack~$\rM$ is a smooth algebraic stacks locally of finite type over~$\kk$
(it is a union of connected components of the stack~${\normalfont\textsc{Fano}}$ considered in~\cite{JL}).
Moreover, $\rM$ is irreducible of finite type because the corresponding Fano varieties 
have uniform descriptions listed just before the theorem (in the case where~$\rM = \MF_{\sX_g}$)
and at the beginning of~\S\ref{sec:bridge} (in the case where~$\rM = \MF_{\sY_d}$);
indeed, we see that in each case an open subset of an appropriate vector space (of global sections) surjects onto~$\rM$.

Note that smoothness of Fano varieties is only used in~\cite[\S2]{JL} 
to deduce unobstructedness of their deformations in the proof of~\cite[Lemma~2.5]{JL}.
However, nodal (or even terminal Gorenstein) Fano threefolds are also unobstructed by~\cite[Proposition~3]{Na97}, 
hence the same argument proves that the stack~$\brM$ 
is a smooth algebraic stacks of finite type over~$\kk$
(for finiteness just note that terminal Gorenstein Fano varieties 
form a bounded family by~\cite[Proposition~2.6.4]{Kollar}).
Furthermore, $\rM$ is dense in~$\brM$ because any nodal Fano variety is smoothable by~\cite[Proposition~4]{Na97};
therefore irreducibility of~$\rM$ implies irreducibility of~$\brM$,
and since~$\brM$ is smooth, also its integrality.

Next, we prove that the 
boundary~$\brM \setminus \rM$ is a Cartier divisor in~$\brM$.
First, the argument of~\cite[Theorem~11]{Na97} shows that 
any nodal Fano variety has a partial smoothing which is 1-nodal;
this means that the boundary~$\brM \setminus \rM$ 
is equal to the closure of the 1-nodal locus~$\brM^{(1)} \subset \brM \setminus \rM$.
On the other hand, as the property of being Cartier divisor is local on the base in the smooth topology,
Corollary~\ref{cor:nodal-cartier} proves that~$\brM^{(1)}$ 
is a Cartier divisor in~$\brM^{\le 1}$, the substack of threefolds with at most one node.
Therefore, $\brM \setminus \rM$ is a Weil divisor in~$\brM$,
and since ~$\brM$ is smooth, $\brM \setminus \rM$ is a Cartier divisor in~$\brM$.
Note also that for any point~$[X]$ in the $1$-nodal locus~$\brM^{(1)}$
there is a smoothing~$\cX/B$ of~$X$ with smooth~$\cX$ (see Theorem~\ref{thm:smoothing}),
hence by Lemma~\ref{lem:s1-cartier} the corresponding curve~$B \to \brM$ intersects the divisor~$\brM^{(1)}$ transversely at~$[X]$,
hence the divisor is smooth at~$[X]$.

Now consider the canonical morphism of stacks~$\xi \colon \bMFM_{\sX_g} \to \bMF_{\sX_g}$.
Let~$(X,\cU_X)$ be a $\kk$-point of~$\bMFM_{\sX_g}$.
By standard deformation theory (see~\cite[Corollary~4.5.2]{HL10}) the relative tangent space to deformations of the pair~$(X,\cU_X)$ 
over~$\bMF_{\sX_g}$ is~$\Ext^1(\cU_X, \cU_X)$ and the obstruction space is~$\Ext^2(\cU_X, \cU_X)$.
Since~$\cU_X$ is exceptional, both spaces vanish, hence the morphism~$\xi$ is \'etale.

Since~$\bMF_{\sX_g}$ is smooth, we deduce that~$\bMFM_{\sX_g}$ is also smooth.
Since the preimage of a Cartier divisor under an \'etale map is a Cartier divisor,
we conclude that~$\bMFM_{\sX_g} \setminus \MFM_{\sX_g}$ is a Cartier divisor in~$\bMFM_{\sX_g}$,
and since the complement of a Cartier divisor is dense, 
we conclude that~$\MFM_{\sX_g}$ is dense in~$\bMFM_{\sX_g}$. 
It also follows that the $1$-nodal locus~$\bMFM_{\sX_g}^{(1)}$ is smooth.

Now consider the fiber of~$\xi$ over a $\kk$-point~$[X]$.
By Definition~\ref{def:mfm} this is the fibered category over~$(\Sch/\kk)$ whose fiber over a scheme~$S$
is the groupoid --- or rather the set --- of global sections~$\cU \in \VB_{X \times S/S}(S)$ 
such that~$\cU_s$ is a Mukai bundle on~$X$ for every geometric point~$s \in S$.
By definition of the \'etale sheaf~$\VB$ the datum of~$\cU$ 
is the same as a collection of vector bundles on~$X \times S_i$ for an \'etale covering~$\{S_i\} \to S$
whose pullbacks are isomorphic on~$X \times (S_i \times_S S_j)$.
Obviously, this fibered category is equivalent to the \'etale sheafification 
of the usual functor of semistable sheaves as in~\cite[\S4.1]{HL10}.
Finally, using the condition that~$\cU_s$ are Mukai bundles and Lemma~\ref{lem:mukai-bundles},
we conclude that~$\xi^{-1}([X]) \cong \cMuk(X)$.

The same argument combined with a relative version of Lemma~\ref{lem:mukai-bundles} 
shows that~$\xi$ factors through an open embedding into a relative moduli space of semistable sheaves, 
which is projective over~$\bMF_{\sX_g}$.
Therefore, $\xi$ is separated and representable.

Now assume~$g \ge 6$. 
If~$X$ is smooth then~$\xi^{-1}([X]) \cong \Spec(\kk)$ by Proposition~\ref{prop:mukai-bundles}.
Since~$\xi$ is \'etale and separated, and~$\MF_{\sX_g}$ is dense in~$\bMF_{\sX_g}$, 
it follows that every nonempty fiber of~$\xi$ is a point.
Since the property of being open immersion is local on the base in the smooth 
(and even in the fpqc) topology (see~\cite[Lemma~02L3]{stacks-project}),
it follows that~$\xi$ is an open embedding;
in particular, it follows that~$\bMFM_{\sX_g}$ is irreducible.

Finally, assume~$g = 4$.
Given an $S$-point~$f \colon \cX \to S$ of~$\MF_{\sX_4}^\circ$
we consider the relative anticanonical embedding~\mbox{$\cX \hookrightarrow \P_S(\cV)$}, 
where~$\cV \coloneqq (f_*\omega_X^{-1})^\vee$.
It is easy to see that there is a unique flat family of quadrics~$\cQ(\cX) \subset \P_S(\cV)$ containing~$\cX$,
and this family defines an $S$-point of~$\MQ_4$. 
This construction is functorial in~$S$, hence defines a morphism of stacks~$\MF_{\sX_4}^\circ \to \MQ_4$.
The argument of Proposition~\ref{prop:mukai-bundles} then proves that~$\MFM_{\sX_4} \cong \MQS_4 \times_{\MQ_4} \MF_{\sX_4}^\circ$, 
i.e., that~\eqref{eq:mfx4-mq4} is a Cartesian diagram.
It also follows that~$\MFM_{\sX_4}$ is an open substack in a projective space bundle over~$\MQS_4$,
and since~$\MQS_4$ is irreducible (by Lemma~\ref{lem:mq-mqs}), 
$\MFM_{\sX_4}$, and hence also~$\bMFM_{\sX_4}$, is irreducible.
\end{proof}

\begin{remark}
If~$g \le 10$ and~$d \le 4$ the stacks~$\MF_{\sX_g}$, $\MFM_{\sX_g}$, and~$\MF_{\sY_d}$ are Deligne--Mumford stacks;
this follows from finiteness of the automorphisms groups of the corresponding Fano threefolds (\cite[Theorem~1.1.2]{KPS18})
by the argument of~\cite[Proposition~A.2]{KP18}.
\end{remark}


\section{A boundary divisor}
\label{sec:stacks-proofs}

In this section we show that Fano threefolds obtained in~\S\ref{sec:bridge} by the bridge construction
form a connected component of the open substack~$\bMFM_{\sX_g}^{(1)}$ 
of the boundary divisor~$\bMFM_{\sX_g}^{\ge 1} = \bMFM_{\sX_g} \setminus \MFM_{\sX_g}$.
We also discuss generalizations of Theorem~\ref{thm:cax-family} 
for families of nodal Fano threefolds over arbitrary bases.

We need to define yet another stack.
Note that for any family of 1-nodal threefolds~$\cY \to S$, 
if~$y_0$ is the nodal $S$-point of~$\cY$ then the blowup~$\Bl_{y_0}(\cY)$ is smooth over~$S$ 
and its exceptional divisor~$E \to S$ is a smooth family of quadric surfaces, 
see Lemma~\ref{lem:uniformly-1n}.

\begin{definition}
\label{def:tmf}
For each~$d \in \{1,2,3,4,5\}$ we define the stack~$\tMF_{\sY_d}$ as the fibered category over~$(\Sch/\kk)$ 
whose fiber over a scheme~$S$ is the groupoid of:
\begin{itemize}
\item 
(for~$d \ge 2$) pairs~$(\cY \to S,\cC)$, 
where~$\cY \to S$ is an $S$-point of~$\MF_{\sY_d}$,
i.e., a family of smooth del Pezzo threefolds of degree~$d$,
and~$\cC \subset \cY$ is a family of smooth rational curves of degree~$d - 1$;
\item 
(for~$d = 1$) pairs~$(\cY \to S,\cF)$, 
where~$\cY \to S$ is an $S$-point of~$\bMF_{\sY_1}^{(1)}$,
i.e., a family of $1$-nodal del Pezzo threefolds of degree~$1$,
and~$\cF \in \Pic_{E/S}(S)$ is a global section of the relative Picard sheaf for the exceptional divisor~$E$
of the blowup~$\Bl_{y_0}(\cY)$, where~$y_0$ is the nodal $S$-point of~$\cY$,
such that for each geometric point~$s \in S$ the divisor class~$\cF_s \in \Pic(E_s)$ is a ruling of~$E_s \cong \P^1 \times \P^1$.
\end{itemize}
Morphisms in~$\tMF_{\sY_d}$ are defined as morphisms in~$\bMF_{\sY_d}$ 
compatible with~$\cC$ or~$\cF$ in the obvious way.
\end{definition}

\begin{lemma}
\label{lem:mfc-smooth-connected}
For all~$d \in \{1,2,3,4,5\}$ the stack~$\tMF_{\sY_d}$ is smooth and connected.
\end{lemma}

\begin{proof}
For~$d \ge 3$ the stack~$\tMF_{\sY_d}$ is an open substack 
in the relative Hilbert scheme of rational curves of degree~$d - 1$ 
over the stack~$\MF_{\sY_d}$ of del Pezzo threefolds of degree~$d$.
Its connectedness and smoothness follow from Theorem~\ref{thm:mf-x-y} and Lemma~\ref{lem:hilbert-smooth}.

Let~$d = 2$. 
Since Lemma~\ref{lem:hilbert-smooth} fails for ramification lines, see Remark~\ref{rem:f1-sing}, we argue differently.
Every smooth del Pezzo threefold~$Y$ of degree~$2$ is a quartic hypersurface
in the weighted projective space~$\P(1,1,1,1,2)$ (see, e.g., \cite[Theorem~1.2(ii)]{KP22}),
which can be thought of as a cone in~$\P^{10}$ over the second Veronese embedding~$\P^3 \hookrightarrow \P^9$.
The Hilbert scheme~$\rF_1^\circ(\P(1,1,1,1,2))$ of curves~\mbox{$C \subset \P(1,1,1,1,2)$} of degree~$1$ 
with respect to the generator of the class group~$\Cl(\P(1,1,1,1,2))$ 
(i.e., of degree~$2$ in the ambient~$\P^{10}$)
not passing through the singular point of the cone
has the following structure:
the projection from the vertex of the cone maps any such curve~$C$ isomorphically onto a line in~$\P^3$
(which under the second Veronese embedding becomes a conic in~$\P^9$),
the preimage of this line in~$\P(1,1,1,1,2)$ is~$\P(1,1,2)$, i.e., a quadratic cone in~$\P^3 \subset \P^{10}$,
and the curve~$C$ is a hyperplane section of this quadratic cone. 
Thus, the Hilbert scheme~$\rF_1^\circ(\P(1,1,1,1,2))$ is an open subscheme in a projective space bundle over~$\Gr(2,4)$, 
hence it is smooth and connected.
Furthermore, the Hilbert scheme of pairs~$C \subset Y$ is an open subscheme 
in a projective space bundle over~$\rF_1^\circ(\P(1,1,1,1,2))$, hence it is also smooth and connected.
Finally, the moduli stack~$\tMF_{\sY_2}$ is the quotient stack of this smooth and connected scheme
by the automorphism group of the weighted projective space,
hence it is smooth and connected.

Finally, let~$d = 1$.
Every $1$-nodal del Pezzo threefold~$Y$ of degree~$1$ 
is a sextic hypersurface in the weighted projective space~$\P(1,1,1,2,3)$ (see, e.g., \cite[Theorem~1.2(i)]{KP22})
and the position of the node~$y_0 \in Y$ is constrained to the complement of the line~$\P(2,3) \subset \P(1,1,1,2,3)$,
(see, e.g., \cite[Proposition~2.2(ii)]{KP22}).
Therefore, the Hilbert scheme of pairs~$y_0 \in Y$ is an open subscheme
in a projective space bundle over~$\P(1,1,1,2,3) \setminus \P(2,3)$, hence it is smooth and connected.
The moduli stack~$\bMF_{\sY_1}^{(1)}$ of $1$-nodal del Pezzo threefolds of degree~$1$ 
is the quotient stack of this smooth and connected scheme 
by the automorphism group of the weighted projective space, hence it is smooth and connected.

Furthermore, since the spinor bundles on a smooth quadric surface are the line bundles associated with the rulings,
the definition of the stack~$\tMF_{\sY_1}$ implies that there is a Cartesian diagram
\begin{equation}
\label{eq:mfy1-covering}
\vcenter{\xymatrix{
\tMF_{\sY_1} \ar[r]^{\hat\upepsilon} \ar[d] &
\MQS_2 \ar[d]^{\xi_{\sQ}} 
\\
\bMF_{\sY_1}^{(1)} \ar[r]^{\upepsilon} &
\MQ_2, 
}}
\end{equation}
where~$\upepsilon$ takes a family~$\cY \to S$ of $1$-nodal del Pezzo threefolds
to the family of quadric surfaces~$E \to S$ where~$E$ is the exceptional divisor of~$\Bl_{y_0}(\cY)$,
and~$\xi_\sQ$ is the \'etale double covering defined in Lemma~\ref{lem:mq-mqs}.
It follows that the left vertical arrow in~\eqref{eq:mfy1-covering} is finite \'etale of degree~$2$,
hence the smoothness of~$\bMF_{\sY_1}^{(1)}$ proved above implies the smoothness of~$\tMF_{\sY_1}$.
It also follows that the morphism~$\hat\upepsilon \colon \tMF_{\sY_1} \to \MQS_2$ 
is a quotient of an locally trivial fibration with smooth conected fibers by an algebraic group, 
and since~$\MQS_2$ is connected by Lemma~\ref{lem:mq-mqs}, the same is true for~$\tMF_{\sY_1}$.
\end{proof}

Next, we discuss a universal version of the bridge construction from~\S\ref{sec:bridge} and its inverse.
To explain it we need universal versions of the line bundles~$\cO_Y(H)$ and~$\cO_E(F)$ and the Mukai bundle~$\cU_X$ used in~\S\ref{sec:bridge}.
A convenient notion to use here is that of a vector bundle on a family of varieties
twisted by a Brauer class pulled back from the base of the family, see~\cite[\S2.2]{K22} for a discussion of this notion.

Note that if~$\cX \to S$ is a flat morphism of schemes, $\upbeta \in \Br(S)$ is a Brauer class,
and~$\cE \in \Coh(\cX,\upbeta)$ is a $\upbeta$-twisted vector bundle on~$\cX$, 
then there is an \'etale covering~$\{S_i\} \to S$ on which~$\upbeta$ becomes trivial,
hence the pullbacks of~$\cE$ to~$\cX \times_S S_i$ are untwisted vector bundles~$\cE_i$ 
such that for all~$i,j$ the pullbacks of~$\cE_i$ and~$\cE_j$ to~$\cX \times_S S_i \times_S S_j$ are isomorphic.
Therefore, the collection~$\{\cE_i\}$ defines a global section of the \'etale sheaf~$\VB_{\cX/S}$.
If this is the case we will say that the twisted bundle~$\cE$ 
{\sf represents} the corresponding global section of~$\VB_{\cX/S}$, which we will abusively also denote~$\cE$.

In the next lemma we show that some global sections of~$\VB$ can be represented by twisted sheaves.

\begin{lemma}
\label{lem:twisted-universal}
Let~$d \in \{1,2,3,4,5\}$ and let~$S$ be a $\kk$-scheme.
\begin{alenumerate}
\item 
\label{it:beta-h}
For any $S$-point~$\cY \to S$ of~$\bMF_{\sY_d}$
there is a $2$-torsion Brauer class~\mbox{$\upbeta_\cH \in \Br(S)$} 
and a $\upbeta_\cH$-twisted line bundle~$\cO_\cY(\cH) \in \Coh(\cY, \upbeta_\cH)$
such that for any geometric point~$s \in S$ the restriction~$\cO_\cY(\cH)\vert_{\cY_s}$ 
is the ample generator of~$\Pic(\cY_s)$.

\item 
\label{it:beta-f}
For any $S$-point~$(\cY \to S,\cC)$ or~$(\cY \to S, \cF)$ of~$\tMF_{\sY_d}$
there is a $2$-torsion Brauer class~\mbox{$\upbeta_\cF \in \Br(S)$} 
and a $\upbeta_\cF$-twisted line bundle~$\cO_E(\cF) \in \Coh(E, \upbeta_\cF)$
on the exceptional divisor~$E$ of~$\tcY = \Bl_\cC(\cY)$ or~$\tcY = \Bl_{y_0}(\cY)$
such that for any geometric point~$s \in S$ the restriction~$\cO_E(\cF)\vert_{E_s}$ 
is the line bundle corresponding to the class of a fiber of the map~$E_s \to C_s$ 
or to the chosen ruling of~$E_s$.

\item
\label{it:beta-u}
For any $S$-point~$(\cX \to S,\cU_\cX)$ of~$\bMFM_{\sX_{2d+2}}$
there is a $2$-torsion Brauer class~\mbox{$\upbeta_\cU \in \Br(S)$} 
and a $\upbeta_\cU$-twisted vector bundle~$\cU_\cX \in \Coh(\cX, \upbeta_\cU)$ 
that represents the global section~$\cU_\cX \in \VB_{\cX/S}(S)$.
\end{alenumerate}
\end{lemma}

\begin{proof}
We use~\cite[Lemma~2.11]{K22} which is stated for smooth morphisms and line bundles,
but the second part of this lemma that we are applying here does not use the smoothness assumption,
and besides, the same argument works for any global section of~$\VB$ whose fibers are \emph{simple} sheaves.

\ref{it:beta-h}
If~$\cY \to S$ is an $S$-point of~$\bMF_{\sY_d}$, 
the \'etale Picard sheaf~$\Pic_{\cY/S}$ is constant of rank~$1$ by Corollary~\ref{cor:pic-constant},
hence there is a global section~$\cH \in \Pic_{\cY/S}(S)$ 
that restricts to the ample generator of the Picard group of each geometric fiber.
By~\cite[Lemma~2.11]{K22} it gives rise to a Brauer class~$\upbeta_\cH$ and a twisted line bundle as required.
The index of~$\cY/S$ is~$2$, hence~$\upbeta_\cH$ is $2$-torsion, see~\cite[Corollary~2.17]{K22}.

\ref{it:beta-f}
If~$(\cY \to S,\cC)$ is an $S$-point of~$\tMF_{\sY_d}$ for $d \ge 2$, then~$\cC \to S$ is a $\P^1$-bundle.
The argument of part~\ref{it:beta-h} allows us to construct a $2$-torsion Brauer class on~$S$
and a twisted line bundle on~$\cC$ that restricts to the ample generator of the Picard group of each geometric fiber;
its pullback along~$E \to \cC$
is then the required twisted line bundle on~$E$.

Similarly, if~$(\cY \to S,\cF)$ is an $S$-point of~$\tMF_{\sY_1}$ 
we just apply~\cite[Lemma~2.11]{K22} to~$\cF \in \Pic_{E/S}(S)$.

\ref{it:beta-u}
Let~$(\cX \to S,\cU_\cX)$ be an $S$-point of~$\bMFM_{\sX_{2d+2}}$.
Then arguing as in the second part of the proof of~\cite[Lemma~2.11]{K22} 
and using exceptionality of~$\cU_\cX$ on the fibers of~$\cX/S$
we construct the Brauer class~$\upbeta_\cU$ and a $\upbeta_\cU$-twisted vector bundle on~$\cX$ representing~$\cU_\cX$.
The class~$\upbeta_\cU$ is $2$-torsion, because~$\wedge^2\cU_\cX$ 
is, on the one hand, $\upbeta_\cU^2$-twisted (by~\cite[Lemma~2.13]{K22}),
and on the other hand, it is untwisted, because it is isomorphic to the canonical line bundle on each fiber.
\end{proof}

\begin{remark}
\label{rem:brauer-vanishing}
Using the argument of~\cite[Corollary~2.17]{K22} it is easy to check that
\begin{equation*}
\upbeta_\cH = 1\quad\text{when~$d \in \{1,3,5\}$},
\qquad\text{and}\qquad
\upbeta_\cU = 1\quad\text{when~$d \in \{2,4\}$, i.e., $g \in \{6,10\}$}.
\end{equation*}
Indeed, the pushforward of~$\cO_\cY(\cH)$ to~$S$ is a $\upbeta_\cH$-twisted vector bundle of rank~$d+2$,
hence its top wedge power is a $\upbeta_\cH^{d+2}$-twisted line bundle on~$S$, hence~$\upbeta_\cH^{d+2} = 1$.
Since~$d+2$ is odd for~$d \in \{1,3,5\}$ and~$\upbeta_\cH$ is $2$-torsion, we conclude that~$\upbeta_\cH = 1$.
Similarly, the pushforward of~$\cU_\cX^\vee$ to~$S$ is a $\upbeta_\cU$-twisted vector bundle of rank~$d+3$ 
(this follows from the Hilbert polynomial computation in Lemma~\ref{lem:mukai-bundles}),
and the same argument applies.
It is also easy to see that~$\upbeta_\cF = \upbeta_\cH$ (hence~$\upbeta_\cF \cdot \upbeta_\cH = 1$) for~$d \in \{2,4\}$.
\end{remark}

Now we explain the relative version of the bridge construction.

\begin{lemma}
\label{lem:mu}
For each~$d \in \{1,2,3,4,5\}$ there is a morphism of stacks
\begin{equation}
\label{eq:upmu}
\upmu \colon \tMF_{\sY_d} \to \bMFM_{\sX_{2d + 2}}^{(1)}
\end{equation}
which takes a geometric point~$(Y,C)$ or~$(Y,F)$ of~$\tMF_{\sY_d}$ to the Fano--Mukai pair~$(X,\cU_X)$,
where~$X$ is constructed in Proposition~\textup{\ref{prop:x-y}} 
and~$\cU_X$ is constructed in Proposition~\textup{\ref{prop:mutations-1}}.
\end{lemma}

\begin{proof}
First, assume~$d \ge 2$.
Let~$(g \colon \cY \to S, \cC)$ be an $S$-point of~$\tMF_{\sY_d}$.
Then~$\tcY \coloneqq \Bl_\cC(\cY) \xrightarrow{\ \sigma\ } \cY$ is a family of smooth threefolds over~$S$.
Let~$\tg \coloneqq g \circ \sigma \colon \tcY \to S$,
and let~$\cO_\cY(\cH)$ and~$\cO_E(\cF)$
be the twisted line bundles constructed in Lemma~\ref{lem:twisted-universal},
where~$E$ is the exceptional divisor of the blowup~$\sigma$.
By Proposition~\ref{prop:x-y}\ref{item:map-to-X} the (untwisted, because~$\upbeta_\cH$ is $2$-torsion) 
line bundle~$\cO_\tcY(2\cH - E)$ is globally generated over~$S$.
Let~$\cV \coloneqq (\tg_*\cO_\tcY(2\cH - E))^\vee$.
Taking the Stein factorization 
\begin{equation*}
\tcY \xrightarrow{\ \pi\ } \cX \xrightarrow{\quad} \P_S(\cV)
\end{equation*}
of the morphism to~$\P_S(\cV)$ given by~$\cO_\tcY(2\cH - E)$ (a relative version of~\eqref{eq:stein}),
we obtain a family~$f \colon \cX \to S$ of $1$-nodal prime Fano threefolds of genus~$2d + 2$.
Furthermore, the relative version of~\eqref{eq:def-cu}
\begin{equation*}
0 \to \cU_\tcY^\vee \to \cO_\tcY(\cH) \otimes \tg^*\tg_*\cO_E(\cF) \to \cO_E(\cH + \cF) \to 0,
\end{equation*}
defines a $(\upbeta_\cF \cdot \upbeta_\cH)$-twisted vector bundle~$\cU_\tcY$ on~$\tcY$,
Proposition~\ref{prop:mutations-1} implies that~$\cU_\cX \coloneqq \pi_*\cU_\tcY$ 
is a~$(\upbeta_\cF \cdot \upbeta_\cH)$-twisted Mukai bundle on~$\cX$ such that~$\cU_\tcY \cong \pi^*\cU_\cX$,
so that~$(f \colon \cX \to S, \cU_\cX)$ is an $S$-point of~$\bMFM_{\sX_g}^{(1)}$.

Similarly, if~$d = 1$, let~$(g \colon \cY \to S, \cF)$ be an $S$-point of~$\tMF_{\sY_1}$
and let~$y_0$ be the nodal $S$-point of~$\cY/S$.
Then~$\tcY \coloneqq \Bl_{y_0}(\cY)$ is a family of smooth threefolds over~$S$ (see Lemma~\ref{lem:uniformly-1n}),
and as before we construct a family~\mbox{$f \colon \cX \to S$} of $1$-nodal prime Fano threefolds of genus~$4$
with a $(\upbeta_\cF \cdot \upbeta_\cH)$-twisted Mukai bundle~$\cU_\cX$ on~$\cX$ 
so that~$(f \colon \cX \to S, \cU_\cX)$ is an $S$-point of~$\bMFM_{\sX_4}^{(1)}$.

In both cases the construction is functorial in~$S$, hence defines a morphism of stacks.
\end{proof}

In the case~$d = 1$ we will need the following obvious observation.

\begin{lemma}
\label{lem:umpu-z2-equivariant}
There is a morphism of stacks~$\bar\upmu \colon \bMF_{\sY_1}^{(1)} \to \bMF_{\sX_4}^{(1)}$ and a Cartesian diagram 
\begin{equation}
\label{eq:mubar}
\vcenter{\xymatrix{
\tMF_{\sY_1} \ar[r]^-\upmu \ar[d] &
\bMFM_{\sX_4}^{(1)} \ar[d]^\xi
\\
\bMF_{\sY_1}^{(1)} \ar[r]^-{\bar\upmu} &
\bMF_{\sX_4}^{(1)},
}}
\end{equation}
where the left vertical arrow is the natural \'etale double cover 
and~$\xi$ is defined in Theorem~\textup{\ref{thm:mf-x-y}}.
\end{lemma}

\begin{proof}
Indeed, in the case~$d = 1$ the construction of the family~$f \colon \cX \to S$ of prime Fano threefolds in Lemma~\ref{lem:mu}  
does not depend on the choice of~$\cF \in \Pic_{E/S}(S)$ (only the Mukai bundle does), 
therefore it defines a morphism~$\bar\upmu$ such that the diagram~\eqref{eq:mubar} commutes.

The diagram~\eqref{eq:mubar} induces a morphism
\begin{equation*}
\tMF_{\sY_1} \to \bMFM_{\sX_4}^{(1)} \times_{\bMF_{\sX_4}^{(1)}} \bMF_{\sY_1}^{(1)},
\end{equation*}
which is \'etale, because the vertical arrows in~\eqref{eq:mubar} are \'etale.
Now, the fibers of the left side over~$\bMF_{\sY_1}^{(1)}$ are 2-point sets, 
while the fibers of the right side over~$\bMF_{\sY_1}^{(1)}$ are at most 2-points sets 
(by separatedness of~$\xi$ and Proposition~\ref{prop:mukai-bundles}), 
hence the morphism is bijective, and therefore the diagram~\eqref{eq:mubar} is Cartesian.
\end{proof}

Recall from Lemma~\ref{lem:mfc-smooth-connected} that the stack~$\tMF_{\sY_d}$ is connected.
For each~$d \in \{1,2,3,4,5\}$ let
\begin{equation*}
\bMFM_{\sX_{2d+2},\sY_d}^{(1)} \subset \bMFM_{\sX_{2d+2}}^{(1)} \subset \bMFM_{\sX_{2d+2}}
\end{equation*}
be the connected component of the 1-nodal locus~$\bMFM_{\sX_{2d+2}}^{(1)}$ of the stack~$\bMFM_{\sX_{2d+2}}$
containing~$\upmu(\tMF_{\sY_d})$.

\begin{theorem}
\label{thm:tmfy-bmfx-d2}
For all~$d \in \{1,2,3,4,5\}$ the morphism~$\upmu$ defined in Lemma~\textup{\ref{lem:mu}} 
induces an isomorphism 
\begin{equation*}
\tMF_{\sY_d} \xrightiso{\ \upmu\ } \bMFM_{\sX_{2d+2},\sY_d}^{(1)}.
\end{equation*}
\end{theorem}

\begin{proof}
We will construct a morphism~$\bMFM_{\sX_{2d+2},\sY_d}^{(1)} \to  \tMF_{\sY_d}$ inverse to~$\upmu$.

Let~$S \to \bMFM_{\sX_{2d+2},\sY_d}^{(1)}$ be a chart (i.e., a smooth morphism from a connected scheme~$S$)
meeting the image of~$\upmu$.
Since~$\bMFM_{\sX_{2d+2},\sY_d}^{(1)}$ is smooth by Theorem~\ref{thm:mf-x-y}, the scheme~$S$ is also smooth.
Let~$(\cX \to S,\cU_\cX)$ be the corresponding $1$-nodal Fano--Mukai pair. 
Let~$x_0$ be the nodal $S$-point of~$\cX$ and let
\begin{equation*}
\tcX \coloneqq \Bl_{x_0}(\cX)
\qquad\text{and}\qquad 
D \subset \tcX
\end{equation*}
be the blowup and its exceptional divisor.
For general~$s \in S$ we have~$\uprho(\tcX_s) = 3$ by Remark~\ref{rem:sarkisov}.
Therefore, by Corollary~\ref{cor:pic-constant} we have~$\uprho(\tcX_s) = 3$ for every~$s \in S$, 
hence~$\cX_s$ is nonfactorial for all~$s$.
Moreover, by~\cite[Theorem~1.1]{KP23} for each point~$s \in S$ the anticanonical class of~$\tcX_{s}$ is nef and big;
it is also not ample for general~$s$, hence not ample for all~$s$ by openness of ampleness.

Let~$\bar{\cX} \to S$ be the relative anticanonical model of~$\tcX \to S$.
For general~$s \in S$ it follows from~\cite[Theorem~1.4 and Remark~1.5]{KP23} 
that~$\bar{\cX}_s \subset Y \times \P^1$ is a complete intersection of ample divisors,
where~$Y$ is a del Pezzo threefold of degree~$d$,
and~$\Pic(\bar{\cX}_s)$ is generated by the pullbacks of the ample generators of~$\Pic(Y_d)$ and~$\Pic(\P^1)$.
By Corollary~\ref{cor:pic-constant} the \'etale Picard sheaf~$\Pic_{\bar{\cX}/S}$ is constant of rank~$2$ 
and has relatively nef generators~$H_1$ and~$H_2$ that by semicontinuity satisfy the inequalities
\begin{equation*}
\dim \rH^0(\bar{\cX}_s, \cO_{\bar{\cX}_s}(H_1)) \ge d + 2
\qquad\text{and}\qquad 
\dim \rH^0(\bar{\cX}_s, \cO_{\bar{\cX}_s}(H_2)) \ge 2
\end{equation*}
for all~$s \in S$.
Now applying~\cite[Theorem~1.4 and Remark~1.5]{KP23} we conclude that for each point~$s \in S$
the anticanonical model of~$\tcX_s$ is a divisor
\begin{equation*}
\bar{\cX}_s \subset \cY_s \times \P^1
\end{equation*}
of bidegree~$(1,1)$,
where~$\cY_s$ is a del Pezzo threefold of degree~$d$ (which is smooth if~$d \ge 2$ and has a single node or cusp if~$d = 1$)
and~$H_1$, $H_2$ are the pullbacks of the ample generators of the Picard groups of the factors.
The relative over~$S$ contraction defined by the class~$H_1$ is a morphism
\begin{equation*}
\bar\cX \to \cY,
\end{equation*}
where~$\cY \to S$ is a flat family of del Pezzo threefolds of degree~$d$.
If~$d \ge 2$ this family is smooth over~$S$
and if~$d = 1$ the singular locus of~$\cY$ over~$S$ is an $S$-point~$y_0$ of~$\cY$
such that~$\Bl_{y_0}(\cY)$ is smooth over~$S$ and its exceptional divisor~$E$ is a family of irreducible quadric surfaces over~$S$.

Assume~$d \ge 2$; then~$\cY \to S$ is an $S$-point of~$\MF_{\sY_d}$.
Moreover, in this case the exceptional divisor of the composition~$\tcX \to \bar\cX \to \cY$ has two irreducible components:
one of them is the exceptional divisor~$D$ and the other component~$E$ 
is contracted onto a family of smooth rational curves~$\cC \subset \cY$ of degree~$d - 1$.
Then~$(\cY \to S,\cC)$ is an $S$-point of the stack~$\tMF_{\sY_d}$, and it is clear that 
\begin{equation}
\label{eq:mu-d-big}
\upmu(\cY \to S, \cC) = (\cX \to S, \cU_\cX).
\end{equation}
The above construction is obviously functorial in~$S$, hence it gives a morphism of stacks which is right inverse to~$\upmu$.
Moreover, a simple verification shows that this map is also left inverse to~$\upmu$,
which completes the proof of the theorem in the case~$d \ge 2$.

Now assume~$d = 1$. 
Then the above argument gives a morphism
\begin{equation}
\bMF_{\sX_4,\sY_1}^{(1)} \to \bMF_{\sY_1}^{(1+)}
\end{equation}
from a component of the stack of 1-nodal Fano threefolds of genus~4
to the stack of 1-nodal or 1-cuspidal del Pezzo threefolds of degree~1 (whose definition is analogous to Definition~\ref{def:mf})
which is right inverse to the natural extension of the morphism~$\bar\upmu$ to the stack~$\bMF_{\sY_1}^{(1+)}$. 
Thus, these two stacks are isomorphic. 

Furthermore, the morphism~$\upmu$ is a lift of~$\bar\upmu$ to the natural double coverings
\begin{equation*}
\tMF_{\sY_1} \to \MF_{\sY_1}^{(1)},
\qquad\text{and}\qquad 
\bMFM_{\sX_4, \sY_1}^{(1)} \xrightarrow{\ \xi\ } \bMF_{\sX_4, \sY_1}^{(1)}
\end{equation*}
where the first is induced by~$\xi_\sQ$, see diagram~\eqref{eq:mfy1-covering},
and the second is the restriction of the morphism defined in Theorem~\ref{thm:mf-x-y}.
Since~$\xi$ is \'etale and~$\xi_\sQ$ is ramified over the substack~$\bMQ_2^{(1)} \subset \bMQ_2$ of nodal quadrics 
by construction (see Lemma~\ref{lem:mq-mqs}), it follows that the image of~$\bMFM_{\sX_4, \sY_1}^{(1)}$ under the composition
\begin{equation*}
\bMFM_{\sX_4, \sY_1}^{(1)} \xrightarrow{\ \xi\ } 
\bMF_{\sX_4, \sY_1}^{(1)} \xrightarrow{\ \bar\upmu^{-1}\ } 
\bMF_{\sY_1}^{(1+)} \xrightarrow{\ \upepsilon\ }
\bMQ_2
\end{equation*}
does not intersect the branch divisor~$\bMQ_2^{(1)} \subset \bMQ_2$ of~$\xi_\sQ$, 
hence the image of~$\bar\upmu^{-1} \circ \xi$ does not intersect the cuspidal locus in~$\bMF_{\sY_1}^{(1+)}$.
This proves that the restriction of~$\bar\upmu$ induces an isomorphism
\begin{equation*}
\bar\upmu \colon \bMF_{\sY_1}^{(1)} \xrightiso{} \xi(\bMFM_{\sX_4, \sY_1}^{(1)})
\end{equation*}
Finally, Lemma~\ref{lem:umpu-z2-equivariant} implies that~$\upmu$ 
induces an isomorphism~$\tMF_{\sY_1} \xrightiso{} \bMFM_{\sX_{4},\sY_1}^{(1)}$ over~$\bar\upmu$.
\end{proof}

In conclusion  we explain how one can construct semiorthogonal decompositions for families of Fano threefolds
more general than those considered in Theorem~\ref{thm:cax-family}.

On the one hand, for all~$d \in \{1,2,3,4,5\}$ and all $S$-points
\begin{equation*}
(f \colon \cX \to S,\cU_\cX) \in \bMFM_{\sX_{2d+2}}(S),
\qquad
(g \colon \cY \to S) \in \bMF_{\sY_d}(S),
\qquad\text{and}\qquad
(g \colon \cY \to S,\cF) \in \tMF_{\sY_1}(S)
\end{equation*}
the construction of Lemma~\ref{lem:twisted-universal} provides $2$-torsion Brauer classes on~$S$
and the corresponding twisted sheaves~$\cU_\cX$, $\cO_\cY(\cH)$, and~$\cO_E(\cF)$
on~$\cX$, $\cY$, and the exceptional divisor~$E$ of~$\tcY = \Bl_\cC(\cY)$ or~$\tcY = \Bl_{y_0}(\cY)$, respectively,
that form relative exceptional collections over~$S$ and induce semiorthogonal decompositions
\begin{align}
\label{eq:dbcx-i1}
\Db(\cX) &= \langle \cA_\cX, f^*\Db(S), f^*\Db(S, \upbeta_\cU) \otimes \cU_\cX^\vee \rangle,\\
\label{eq:dbcy-i2}
\Db(\cY) &= \langle \cB_\cY, g^*\Db(S), g^*\Db(S, \upbeta_\cH) \otimes \cO_\cY(\cH) \rangle,\\
\label{eq:dbtcy}
\Db(\tcY) &= \langle 
\tcB_\cY, 
\tg^*\Db(S), 
\tg^*\Db(S, \upbeta_\cH) \otimes \cO_\cY(\cH), 
\tg^*\Db(S) \otimes \cO_E, 
\tg^*\Db(S, \upbeta_\cF) \otimes \cO_E(\cF) 
\rangle.
\end{align}
Here~$\tg \colon \tcY \to S$ is the composition of the blowup morphism with~$g$
and for~$d \ge 2$ the category~$\tcB_\cY$ defined by~\eqref{eq:dbtcy} 
is equivalent to the category~$\cB_\cY$ defined by~\eqref{eq:dbcy-i2},
while for~$d = 1$ it provides a crepant categorical resolution of~$\cB_\cY$ as in Remark~\ref{rem:ccr}.

On the other hand, consider the open substack
\begin{equation}
\label{eq:bmfm-x-y}
\bMFM_{\sX_{2d + 2},\sY_d} \coloneqq 
\MFM_{\sX_{2d + 2}} \mathop{\sqcup} \bMFM_{\sX_{2d + 2},\sY_d}^{(1)} =
\MFM_{\sX_{2d + 2}} \mathop{\sqcup} \upmu(\tMF_{\sY_d}) \subset 
\bMFM_{\sX_{2d + 2}}
\end{equation}
that parameterizes Fano--Mukai pairs~$(X,\cU_X)$ such that either~$X$ is smooth 
or~$(X,\cU_X)$ is obtained by the bridge construction of~\S\ref{sec:bridge}.
Note that~$\bMFM_{\sX_{2d + 2},\sY_d}^{(1)} \subset \bMFM_{\sX_{2d + 2},\sY_d}$ is a smooth Cartier divisor.
Consider an $S$-point of~$\bMFM_{\sX_{2d + 2},\sY_d}$ \emph{transverse to the boundary},
i.e., an $S$-point~$(f \colon \cX \to S, \cU_\cX)$ such that the $1$-nodal locus
\begin{equation*}
S^{(1)} = S \times_{\bMFM_{\sX_{2d + 2},\sY_d}} \bMFM_{\sX_{2d + 2},\sY_d}^{(1)}
\end{equation*}
is a smooth Cartier divisor in~$S$.
By Theorem~\ref{thm:tmfy-bmfx-d2} there is an $S^{(1)}$-point 
$(\cY \to S^{(1)}, \cC)$ or~$(\cY \to S^{(1)}, \cF)$ of~$\tMF_{\sY_d}$
such that
\begin{equation*}
\cX^{(1)} \coloneqq \cX \times_S S^{(1)} \cong \tcY_\can,
\end{equation*}
where~$\tcY = \Bl_\cC(\cY)$ or~$\tcY = \Bl_{y_0}(\cY)$, and
\begin{equation*}
\pi \colon \tcY \to \tcY_\can = \cX^{(1)}
\end{equation*}
is a small birational contraction.
Consider the twisted sheaf
\begin{equation*}
\rP \coloneqq \pi_*\cO_\tcY(E - \cH) \in \Db(\cX^{(1)}, \upbeta_\cH),
\end{equation*}
this is a relative twisted version of the $\Pinfty2$-object constructed in~\eqref{eq:def-rpx}.
Let~$\io \colon \cX^{(1)} \hookrightarrow \cX$ be the embedding induced by the embedding~$S^{(1)} \hookrightarrow S$.
One can prove that the $S$-linear Fourier--Mukai functor
\begin{equation*}
\Phi_\rP \colon \Db(S^{(1)}, \upbeta_\cH) \to \Db(\cX),
\qquad
\Phi_\rP(-) = \io_*((f\vert_{\cX^{(1)}})^*(-) \otimes \rP)
\end{equation*}
is fully faithful and there is an $S$-linear semiorthogonal decomposition
\begin{equation}
\label{eq:dbcx-bca}
\Db(\cX) = 
\Big\langle 
\Phi_\rP(\Db(S^{(1)}, \upbeta_\cH)), 
\bcA_\cX, 
f^*\Db(S), 
f^*\Db(S, \upbeta_\cU) \otimes \cU_\cX^\vee 
\Big\rangle
\end{equation}
refining~\eqref{eq:dbcx-i1},
where~$\bcA_\cX$, defined as the orthogonal complement in~$\cA_\cX$ of the image of~$\Phi_\rP$,
is smooth and proper over~$S$.
Finally, setting~$S^{(0)} \coloneqq S \setminus S^{(1)}$ and~$\cX^{(0)} \coloneqq \cX \times_S S^{(0)}$,
one can check that
\begin{equation}
\label{eq:bca-general}
(\bcA_\cX)_{S^{(0)}} \simeq \cA_{\cX^{(0)}}
\qquad\text{and}\qquad
(\bcA_\cX)_{S^{(1)}} \simeq 
\begin{cases}
\cB_{\cY}, & \text{if~$d \ge 2$},\\
\tcB_{\cY}, & \text{if~$d = 1$},
\end{cases}
\end{equation}
where the categories in the right-hand sides are defined in~\eqref{eq:dbcx-i1}, \eqref{eq:dbcy-i2},
and~\eqref{eq:dbtcy}, respectively.

One way to prove~\eqref{eq:dbcx-bca} and~\eqref{eq:bca-general}
is by extending our results from~\cite{KS22:abs} about~$\Pinfty2$-objects to the relative setting.
Another way is to bootstrap from Theorem~\ref{thm:cax-family} by base change to appropriate curves~\mbox{$B \subset S$}
and the technique developed in~\cite{K06}.
We leave the details to the interested reader.


\appendix

\section{Nodal varieties}
\label{sec:nodal}

In this appendix we discuss geometry of nodal varieties and of families of nodal varieties.
We assume that the base field~$\kk$ is algebraically closed of characteristic zero.
Recall that a scheme~$X$ of dimension~$n$ {\sf has hypersurface singularities} 
if it is locally isomorphic to a hypersurface~$\{\varphi = 0\}$ in~$\AA^{n+1}$;
this is equivalent to the inequality~$\dim T_x(X) \le n + 1$ for each geometric point~$x \in X$.
We will use the following

\begin{definition}
\label{def:nodal}
A $\kk$-scheme~$X$ of dimension~$n$ is {\sf nodal} if it has isolated hypersurface singularities
and the Hessian matrix~$(\partial^2\!\varphi/\partial x_i\partial x_j)_{i,j=0}^{n+1}$ 
of its local equation is nondegenerate at every singular point of~$X$.
\end{definition}

By definition any nodal scheme~$X$ is Gorenstein.
Moreover, it follows (see Lemma~\ref{lem:uniformly-1n} for an argument in the relative case) 
that the blowup~$\pi \colon \tX \to X$ of~$X$ at its singular points is smooth, 
each connected component~$E \subset X$ of the exceptional divisor is a smooth quadric of dimension~$n - 1$
with conormal bundle~$\cO_E(-E)$ isomorphic to the hyperplane bundle of the quadric.
Moreover, the discrepancy of~$E$ equals~$n - 2$;
in particular, $X$ has terminal singularities when~$n \ge 3$.

\subsection{Morphisms with nodal fibers}
\label{subsec:nodal-morphisms}

We will say that~$f \colon \cX \to S$ is a morphism {\sf with nodal fibers} 
if~$f$ is flat and all geometric fibers of~$f$ are smooth or nodal.
We always assume that the base~$S$ is Noetherian.

\begin{lemma}
\label{lem:nodal-hypersurface}
If every geometric fiber of a flat morphism~$f \colon \cX \to S$ has hypersurface singularities
then~$\cX/S$ has hypersurface singularities, 
i.e., $\cX$ is locally isomorphic to a hypersurface in an affine space over~$S$.
In particular, a morphism with nodal fibers has hypersurface singularities.
\end{lemma}

\begin{proof}
Let~$x \in \cX$ be a geometric point and set~$s = f(x) \in S$.
Since the fiber~$\cX_s$ has hypersurface singularities, 
there is a neighborhood~$(U_s,x)$ of~$x$ in~$\cX_s$
and an open embedding~$U_s \hookrightarrow \{\varphi = 0\} \subset \AA^{n+1}$, where~$n = \dim(\cX/S)$.
The morphism~$U_s \hookrightarrow \AA^{n+1}$ is given by~$n + 1$ functions;
they extend to regular functions on a neighborhood~$(U,x)$ of~$x$ in~$\cX$ 
such that~$U \cap \cX_s = U_s$ and define an $S$-morphism
\begin{equation*}
g \colon U \to \AA_S^{n+1}
\end{equation*}
that restricts to the above morphism of~$U_s$ over~$s$.
The image of~$g$ is a hypersurface over~$s$, hence it is a hypersurface over a neighborhood of~$s$ in~$S$.
The property of being an open immersion is local on the base (see~\cite[Lemma~02L3]{stacks-project}),
hence, shrinking~$U$ if necessary, we may assume that locally~$g$ is an open embedding into a hypersurface in~$\AA_S^{n+1}$.
\end{proof}

If~$f \colon \cX \to S$ is a flat morphism and~$\cX/S$ has hypersurface singularities, 
the Jacobian ideals of local equations of~$\cX \subset \AA^{n+1}_S$ 
provide the relative singular locus~$\Sing(\cX/S)$ with a scheme structure.
More precisely, if~$\{\varphi = 0\} \subset \AA^{n+1}_S$ is a local presentation of~$\cX$ then
\begin{equation*}
\Sing(\cX/S) \coloneqq \{ \varphi_0 = \varphi_1 = \dots = \varphi_n = 0 \} \subset \cX,
\end{equation*}
where~$\varphi_i = \partial\varphi/\partial x_i$ 
is the derivative of~$\varphi$ with respect to $i$-th coordinate on~$\AA_S^{n+1}$.
We always endow~$\Sing(\cX/S)$ with this scheme structure.

\begin{lemma}
\label{lem:sing-finite}
If~$f \colon \cX \to S$ is a flat proper morphism with nodal fibers
then~$\Sing(\cX/S)$ is finite over~$S$.
\end{lemma}

\begin{proof}
The singular locus of~$\cX$ over~$S$ is quasi-finite over~$S$ by Definition~\ref{def:nodal}, 
and on the other hand it is closed in~$\cX$, hence proper over~$S$;
therefore, it is finite over~$S$.
\end{proof}

Now assume~$f \colon \cX \to S$ is a morphism with nodal fibers and~$\Sing(\cX/S)$ is finite over~$S$
(by Lemma~\ref{lem:sing-finite} this holds if~$f$ is proper).
Then~$f_*\cO_{\Sing(\cX/S)}$ is a coherent sheaf on~$S$.
We denote by
\begin{equation*}
S^{\ge m} \subset S
\qquad\text{and}\qquad 
S^{\le m-1} \subset S
\end{equation*}
the closed subscheme where the rank of~$f_*\cO_{\Sing(\cX/S)}$ is at least~$m$
(it is defined by an appropriate Fitting ideal) 
and its open complement, respectively.
Furthermore, we denote 
\begin{equation*}
S^{(m)} \coloneqq S^{\le m} \cap S^{\ge m}.
\end{equation*}
The subscheme~$S^{(m)}$ is closed in~$S^{\le m}$, open in~$S^{\ge m}$, and locally closed in~$S$;
it parametrizes points of~$S$ over which the fibers are exactly~$m$-nodal.
We also write~$\cX^{\le m}/S^{\le m}$, $\cX^{\ge m}/S^{\ge m}$, and~$\cX^{(m)}/S^{(m)}$ 
for the corresponding families obtained from~$\cX/S$ by base change.

\begin{lemma}
\label{lem:s1-cartier}
If~$f \colon \cX \to S$ is a morphism with nodal fibers,
the morphism~$\Sing(\cX/S) \to S$ is unramified.
Moreover, if~$\Sing(\cX/S)$ is finite over~$S$ 
then the closed subscheme~$S^{(1)} \subset S^{\le 1}$ is locally principal 
and the morphism~$\Sing(\cX^{(1)}/S^{(1)}) \to S^{(1)}$ is an isomorphism. 
Finally, if~$S$ is smooth then~$S^{(1)}$ is smooth if and only if~$\cX$ is smooth along~$f^{-1}(S^{(1)})$.
\end{lemma}

\begin{proof}
By Lemma~\ref{lem:nodal-hypersurface} we may assume that~$\cX \subset \AA^{n+1}_S$ is a hypersurface with equation~$\varphi$.
Consider the critical locus of~$\varphi$, that is the subscheme
\begin{equation*}
\tS \coloneqq \{ \varphi_0 = \varphi_1 = \dots = \varphi_n = 0 \} \subset \AA^{n+1}_S.
\end{equation*}
Then~$\Sing(\cX/S) = \tS \cap \cX$, 
and since the Hessian matrix of~$\varphi$ is nondegenerate at every point of~$\Sing(\cX/S)$,
the projection~$f\vert_\tS \colon \tS \to S$ is \'etale along~$\Sing(\cX/S)$, hence~$\Sing(\cX/S) \to S$ is unramified.

For the second claim we may assume~$S = S^{\le 1}$, i.e., $S^{\ge 2} = \varnothing$.
The assumption~$S^{\ge 2} = \varnothing$ implies that the sheaf~$f_*\cO_{\Sing(\cX/S)}$ 
as $\cO_S$-module is locally generated by its unit, hence~$f\vert_{\Sing(\cX/S)} \colon \Sing(\cX/S) \to S$ is a closed embedding
and~$S^{(1)} = f(\Sing(\cX/S))$ by definition.
Finally, since locally around~$S^{(1)}$ the morphism~$f\vert_\tS \colon \tS \to S$ is an isomorphism,
and~$\Sing(\cX/S)$ is cut out in~$\tS$ by the equation~$\varphi = 0$, 
it follows that~$\Sing(\cX/S)$ is locally principal in~$\tS$,
hence~$S^{(1)}$ is locally principal in~$S$.

For the last claim note that~$\cX$ is smooth away from~$\Sing(\cX/S)$.
So, let~$x \in \Sing(\cX/S)$ be a point such that~$s = f(x) \in S^{(1)}$. 
Then~$x \in \tS$, hence~$\varphi_i(x) = 0$ for~$0 \le i \le n$, 
and the equality~$\Sing(\cX/S) = \tS \cap \cX$ implies that the tangent space to~$\cX$ at~$x$
is equal to the sum of the tangent space to~$\Sing(\cX/S)$ and the relative tangent space to~$\AA^{n+1}_S$ at~$x$.
Therefore, $\cX$ is smooth at~$x$ if and only if~$\Sing(\cX/S)$ is smooth at~$x$,
and since~$\Sing(\cX^{(1)}/S^{(1)}) \to S^{(1)}$ is an isomorphism, this holds if and only if~$S^{(1)}$ is smooth at~$s$.
\end{proof}

\begin{corollary}
\label{cor:nodal-cartier}
Let~$f \colon \cX \to S$ be a flat proper morphism with nodal fibers over an integral scheme~$S$.
If the general fiber of~$f$ is smooth
then~$S^{(1)} \subset S^{\le 1}$ is a Cartier divisor.
\end{corollary}

\begin{proof}
By Lemma~\ref{lem:s1-cartier} the subscheme~$S^{(1)} \subset S^{\le 1}$ is locally principal.
It is not equal to~$S^{\le 1}$ because the general fiber of~$f$ is smooth, 
hence the equation locally defining~$S^{(1)}$ in~$S^{\le 1}$ is nonzero,
and since~$S$ is integral, this equation is not a zero divisor.
Therefore, $S^{(1)} \subset S^{\le 1}$ is a Cartier divisor.
\end{proof}

We will say that~$f \colon \cX \to S$ is {\sf uniformly $m$-nodal} if~$\Sing(\cX/S)$ is finite over~$S$ and~$S^{(m)} = S$.
In the special case where~$m = 1$
the subscheme~$\Sing(\cX/S) \subset \cX$ is an $S$-point of~$\cX$;
it will be called {\sf the nodal $S$-point} or simply {\sf the node} of~$\cX/S$.

\begin{lemma}
\label{lem:uniformly-1n}
If~$f \colon \cX \to S$ is a uniformly $m$-nodal flat morphism,
then
\begin{renumerate}
\item 
\label{it:sing-etale}
the scheme~$\Sing(\cX/S)$ is finite \'etale of degree~$m$ over~$S$, 
\item 
\label{it:bl-smooth}
the blowup~$\Bl_{\Sing(\cX/S)}(\cX)$ is smooth over~$S$, 
\item 
\label{it:ed-smooth}
the exceptional divisor~$E \subset \Bl_{\Sing(\cX/S)}(\cX)$ 
is a smooth quadric fibration over~$\Sing(\cX/S)$, and
\item 
\label{it:ed-conormal}
the conormal bundle~$\cO_E(-E)$ is the relative hyperplane bundle on this quadric fibration.
\end{renumerate}
\end{lemma}

\begin{proof}
\ref{it:sing-etale}
The equality~$S^{(m)} = S$ means that~$f_*\cO_{\Sing(\cS/S)}$ is locally free of rank~$m$,
hence~$\Sing(\cX/S)$ is flat and finite
over~$S$.
Since it is also unramified over~$S$ by Lemma~\ref{lem:s1-cartier}, it is \'etale.

\ref{it:bl-smooth}--\ref{it:ed-conormal}
All these claims are \'etale local around~$\Sing(\cX/S)$, 
so we may assume that~$m = 1$ and~$\Sing(\cX/S)$ is just the nodal $S$-point~$x_0$ of~$\cX$.
Furthermore, by Lemma~\ref{lem:nodal-hypersurface} 
we may assume that~$\cX \subset \AA^{n+1}_S$ is a hypersurface with equation~$\varphi$.
Note that~$\varphi$ vanishes together with its first derivatives at~$x_0$,
while the Hessian matrix of~$\varphi$ at~$x_0$ is nondegenerate.
Clearly,
\begin{equation*}
\Bl_{x_0}(\cX) \subset \Bl_{x_0}(\AA^{n+1}_S) 
\end{equation*}
is the Cartier divisor with equation obtained by dividing the pullback of~$\varphi$ 
by the square of the equation of the exceptional divisor~$\tE$ of~$\Bl_{x_0}(\AA^{n+1}_S)$,
and the exceptional divisor~$E$ of~$\Bl_{x_0}(\cX)$ is the quadric in~$\tE \cong \P^n_S$ 
associated with the Hessian matrix at~$x_0$.
In particular, $E$ is smooth over~$S$, hence~$\Bl_{x_0}(\cX)$ is smooth along~$E$.
On the other hand, $\Bl_{x_0}(\cX) \setminus E = \cX \setminus\{x_0\}$ is also smooth over~$S$,
hence~$\Bl_{x_0}(\cX)$ is smooth over~$S$.
Finally, the conormal bundle of~$E$ is isomorphic to the restriction of the conormal bundle of~$\tE$;
the latter is obviously the hyperplane bundle, hence so is the former.
\end{proof}

\subsection{Families of nodal Fano varieties}

In this section we study properties of families of nodal Fano varieties,
i.e., proper morphisms~$f \colon \cX \to S$ with nodal fibers 
such that the relative anticanonical class~$\omega_{\cX/S}^{-1}$ is $f$-ample.
However, for the statement of Proposition~\ref{prop:pic-loc-const} we can weaken the assumptions.

We denote by~$\Pic_{\cX/S}$ the \'etale sheafification of the Picard functor, see~\cite[\S9.2]{Kleiman}.

\begin{proposition}
\label{prop:pic-loc-const}
Let~$f \colon \cX \to S$ be a flat projective morphism of relative dimension~$n \ge 3$ such that 
\begin{equation}
\label{eq:h1h2}
\rH^1({\cX_s}, \cO_{\cX_s}) = \rH^2({\cX_s}, \cO_{\cX_s}) = 0
\end{equation} 
for any geometric point~$s \in S$.
Assume also that either
\begin{alenumerate}
\item 
\label{it:nodal-fibers}
$f$ has nodal fibers, or
\item 
\label{it:single-fiber}
the scheme~$S$ is integral of positive dimension, 
the fiber of~$f$ over a geometric point~$s \in S$ has isolated hypersurface singularities,
and over~$S \setminus \{s\}$ the morphism~$f$ is smooth.
\end{alenumerate}
Then the \'etale sheaf\/~$\Pic_{\cX/S}$ is locally constant on~$S$.
In particular, if~$S$ is connected then the Picard rank of geometric fibers of~$f$ is constant.
\end{proposition}

\begin{proof}
Recall from~\cite[Theorem~9.4.8, Theorem~9.5.11, Remark~9.5.12, and Proposition~9.5.19]{Kleiman} that 
the assumption~\eqref{eq:h1h2} implies that the Picard sheaf~$\Pic_{\cX/S}$ 
is represented by the Picard scheme~$\bPic_{\cX/S}$ which is \'etale over~$S$.
Therefore, any line bundle on a geometric fiber~$\cX_s$ of~$f$ extends uniquely 
to a line bundle over an \'etale neighborhood of~$s$ in~$S$.

Let us prove that the morphism~$\bPic_{\cX/S} \to S$ satisfies the valuative criterion for properness
(note, however, that this morphism is not of finite type, hence not proper).
So, assume~$S$ is the spectrum of a discrete valuation ring, $s \in S$ is the closed point,
and~$\cL_0$ is a line bundle on the general fiber of~$\cX/S$.
We need to show that there is a line bundle~$\cL$ on~$\cX$ which restricts to the general fiber as~$\cL_0$.

First, consdier case~\ref{it:single-fiber}.
Then we extend~$\cL_0$ to a coherent sheaf on~$\cX$ and then, 
by taking the double dual, to a reflexive sheaf~$\cL$ of rank~$1$.
Since~$\cX/S$ is smooth over the general point of~$S$, we have an inclusion~$\Sing(\cX) \subset \Sing(\cX_s)$. 
Moreover, $\cX$ has isolated hypersurface singularities by Lemma~\ref{lem:nodal-hypersurface}.
Finally, since~$\dim(\cX_s) \ge 3$, we have~$\dim(\cX) \ge 4$.
Now it follows from~\cite[XI, Corollaire~3.14]{Gr} that~$\cX$ is locally factorial,
hence~$\cL$ is a line bundle.

Now, consider case~\ref{it:nodal-fibers}.
Since~$\Sing(\cX/S)$ is finite and unramified over~$S$ by Lemmas~\ref{lem:sing-finite} and~\ref{lem:s1-cartier}, 
we have
\begin{equation*}
\Sing(\cX/S) = Z \sqcup Z',
\end{equation*}
where~$Z$ is finite \'etale over~$S$ and~$Z'$ is a finite reduced scheme contained in the central fiber of~$\cX/S$.
Consider the blowup
\begin{equation*}
\pi \colon \cX' \coloneqq \Bl_{Z}(\cX) \to \cX.
\end{equation*}
By Lemma~\ref{lem:uniformly-1n} the map~$\cX' \to S$ has nodal fibers, $\Sing(\cX'/S) = \Sing(\cX'_s) = Z'$,
and the exceptional divisor~$E$ of~$\pi$ is a smooth quadric fibration over~$Z$.
In particular, $\cX' \to S$ satisfies the assumption of part~\ref{it:single-fiber}.
Therefore, the line bundle~$\cL'_0 \coloneqq \pi^*\cL_0$ extends to a line bundle~$\cL'$ on~$\cX'$.
Let~$E_0$ and $Z_0$ be the general fibers of~$E$ and~$Z$ over~$S$.
The line bundle~$\cL'\vert_{E_0} \cong \cL'_0\vert_{E_0} \cong \pi^*\cL_0\vert_{E_0}$ is trivial over~$Z_0$,
and since~$E \to Z$ is a smooth quadric bundle, $\cL'\vert_E$ is trivial over~$Z$, 
and therefore there is a line bundle~$\cL$ on~$\cX$ such that~$\cL' \cong \pi^*\cL$.
Clearly, the restriction of~$\cL$ to the general fiber of~$\cX/S$ is isomorphic to~$\cL_0$.

Finally, since~$\bPic_{\cX/S}$ is \'etale and satisfies the valuative criterion for properness over~$S$, 
it follows that \'etale sheaf~$\Pic_{\cX/S}$ represented by it is locally constant.
\end{proof}

Since any nodal Fano variety satisfies~\eqref{eq:h1h2} (see~\cite[Proposition~2.1.2(i)]{IP}),
the conclusion of Proposition~\ref{prop:pic-loc-const} holds for families of nodal Fano varieties of dimension at least~$3$.

\begin{corollary}
\label{cor:pic-constant}
If~$f \colon \cX \to S$ is a family of nodal Fano varieties of dimension~$n \ge 3$
then the Picard rank, Fano index, and anticanonical degree of the fibers of~$\cX/S$ are locally constant.

Moreover, if the Picard rank of fibers is~$1$,
the \'etale sheaf~$\Pic_{\cX/S} \cong \underline{\ZZ}_S$ is constant.
\end{corollary}

\begin{proof}
The sheaf~$\Pic_{\cX/S}$ is locally constant by Proposition~\ref{prop:pic-loc-const},
hence by~\cite[\S2.1]{K22} it corresponds to a monodromy action of the \'etale fundamental group~$\uppi_1(S)$ 
on the Picard group~$\Pic(\cX_s)$ of a  geometric fiber.
It follows that the Picard rank is locally constant.
Since the canonical class~$K_{\cX_s} \in \Pic(\cX_s)$ is monodromy invariant (see~\cite[Lemma~2.5]{K22}), 
the Fano index and anticanonical degree are also locally constant.

Assume now the Picard rank of~$\cX_s$ is~$1$.
Since~$\cX_s$ is Fano, the group~$\Pic(\cX_s)$ is torsion free by~\cite[Proposition~2.1.2(ii,iii)]{IP}
hence~$\Pic(\cX_s) \cong \ZZ$.
Since the canonical class~$K_{\cX_s} \in \Pic(\cX_s)$ is monodromy invariant and nonzero,
the monodromy action is trivial, and the sheaf~$\Pic_{\cX/S}$ is constant.
\end{proof}

The following result proved by Namikawa is crucial for our paper.

\begin{theorem}[{\cite{Na97}}]
\label{thm:smoothing}
Let~$X$ be a nodal Fano threefold.
There exists a smoothing of~$X$, 
i.e., a flat projective morphism~$\cX \to B$ over a smooth connected pointed curve~$(B,o)$ 
such that
\begin{itemize}
\item 
the total space~$\cX$ is smooth,
\item 
$\cX_o \cong X$,
\item 
the morphism~$f^{-1}(B \setminus \{o\}) \to B \setminus \{o\}$ is a smooth family of Fano threefolds 
of the same Picard number, Fano index, and anticanonical degree as~$X$.
\end{itemize}
\end{theorem}
\begin{proof}
In~\cite[Proposition~3]{Na97} it is proved that deformations of a nodal Fano threefold~$X$ are unobstructed,
and in~\cite[Proposition~4]{Na97} it is proved that any local deformation of~$X$ around its singular locus 
lifts to a global deformation.
Since nodal singularities are hypersurface singularities, they are locally smoothable, 
hence a global smoothing~$\cX \to B$ also exists.

Since ampleness is an open property, shrinking~$B$ if necessary we can assume that~$\omega_{\cX/B}^{-1}$ is ample over~$B$,
hence every fiber of~$\cX \to B$ is a Fano variety.
The equalities of the Picard rank, Fano index, and anticanonical degree of all fibers of~$\cX/B$ 
follow from Corollary~\ref{cor:pic-constant}.
\end{proof}

\subsection{Maximal nonfactoriality of Fano threefolds}
\label{sec:mnf}

In this subsection we work over $\kk = \CC$ and study the topology of maximally nonfactorial threefolds.

Recall from~\cite[Definition~3.4]{Kalck-Pavic-Shinder} that a nodal threefold $X$ is {\sf maximally nonfactorial}
(respectively, {\sf $\QQ$-maximally nonfactorial})
if the natural morphism from the class group of~$X$ to the direct sum of the local class groups of~$X$ at the singular points 
is surjective (respectively, has finite cokernel).
See~\cite[Definition~6.10]{KS22:abs} for an equivalent definition in terms of the blowup of~$X$ at the nodes.
For us, however, the following reformulation is more convenient:
assume~$X$ has a small resolution of singularities~$\hX \to X$, let~$C_1,\dots,C_m \subset \hX$ be its exceptional curves,
and consider the natural linear map
\begin{equation}
\label{eq:restriction-Pic}
\Pic(\hX) \to \Z^m,
\qquad
D \mapsto (D \cdot C_i)_{i=1}^m.
\end{equation}
Then by~\cite[Lemma~6.14]{KS22:abs}
maximal nonfactoriality (respectively, $\QQ$-maximal nonfactoriality) of~$X$ is equivalent
to surjectivity (respectively, surjectivity after tensoring with~$\QQ$) of~\eqref{eq:restriction-Pic}.
Restating this criterion in terms of homology groups, we obtain

\begin{lemma}
\label{lem:Qmnf}
Let~$\varpi\colon \hX \to X$ be a small resolution of a nodal projective complex threefold~$X$ such that
the condition~\eqref{eq:h1h2} holds for~$X$.
Let $C_1, \dots, C_m \subset \hX$ be the exceptional curves of~$\varpi$.
Then 
\begin{renumerate}
\item 
$X$ is $\QQ$-maximally nonfactorial if and only if the classes~$[C_1], \dots, [C_m] \in \rH_2(\hX,\ZZ)$ are linearly independent,
\item 
$X$ is maximally nonfactorial if and only if the classes~$[C_1], \dots, [C_m] \in \rH_2(\hX,\ZZ)$ are linearly independent
and the subgroup in~$\rH_2(\hX, \ZZ)$ generated by these classes
is saturated.
\end{renumerate}
\end{lemma}
\begin{proof}
Let~$\gamma \colon \ZZ^m \to \rH_2(\hX, \ZZ)$ be the map given by the classes~$[C_1],\dots,[C_m]$ of the exceptional curves.
It follows from~\eqref{eq:h1h2} that~$\Pic(\hX) \simeq \rH^2(\hX, \Z)$,
hence the map~\eqref{eq:restriction-Pic} can be rewritten as the composition
\begin{equation*}
\rH^2(\hX, \ZZ) \twoheadrightarrow
\Hom(\rH_2(\hX, \ZZ), \ZZ) \xrightarrow{\ \gamma^*\ }
\Hom(\ZZ^m, \ZZ).
\end{equation*}
The first arrow is surjective by the universal coefficient theorem,
hence~\eqref{eq:restriction-Pic} is surjective if and only if~$\gamma^*$ is,
which holds exactly when~$\gamma$ is an isomorphism onto a saturated subgroup of~$\rH_2(X,\ZZ)$.
Similarly, tensor product of~\eqref{eq:restriction-Pic} with~$\QQ$ is surjective if and only if~$\gamma^*_\QQ$ is,
which holds exactly when~$\gamma$ is injective.
\end{proof}

By~\cite[Proposition~6.13]{KS22:abs} for any $\QQ$-maximally nonfactorial nodal threefold~$X$ there exists a small resolution.
We will need the following standard topological computation (see, e.g., \cite[\S1]{Clemens}).

\begin{lemma}
\label{lem:small-coh}
Let~$X$ be a maximally nonfactorial $m$-nodal 
threefold
and let~$\varpi\colon \hX \to X$ be a small resolution.
Then the map~$\varpi^* \colon \rH^3(X, \Z) \to \rH^3(\hX, \Z)$ is an isomorphism.
If, moreover, $X$ is proper then
\begin{equation}
\label{eq:b2-X}
\rb_4(X) = \rb_2(X) + m
\end{equation}
and the natural mixed Hodge structure of~$\rH^3(X,\ZZ)$ is pure.
\end{lemma}

\begin{proof}
Let~$C_1,\dots,C_m \subset \hX$ be the exceptional curves.
The cohomology exact sequence of a pair gives 
\begin{equation*}
0  \to
\rH^2(X, \ZZ) \xrightarrow{\ \varpi^*\ }
\rH^2(\hX, \ZZ) \xrightarrow{\ \gamma^*\ }
\bigoplus_{i=1}^m \rH^2(C_i, \ZZ) \to
\rH^3(X, \ZZ) \xrightarrow{\ \varpi^*\ }
\rH^3(\hX, \ZZ) \to
 0,
\end{equation*}
where~$\gamma \colon \sqcup C_i \hookrightarrow \hX$ is the embedding.
Since~\eqref{eq:restriction-Pic} is surjective by the maximal nonfactoriality assumption,
the map~$\gamma^*$ is also surjective, hence~$\varpi^* \colon \rH^3(X, \ZZ) \to \rH^3(\hX, \ZZ)$ is an isomorphism.
It also follows that~$\rb_2(\hX) = \rb_2(X) + m$ and~$\rb_4(\hX) = \rb_4(X)$.

If~$X$ is proper, then~$\hX$ is smooth and proper, the Hodge structure of~$\rH^3(\hX,\ZZ)$ is pure,
hence the Hodge structure of~$\rH^3(X,\ZZ)$ is pure as well.
Finally, Poincar\'e duality on~$\hX$ implies~\eqref{eq:b2-X}.
\end{proof}

The following result is well known but we could not find a reference.
For an abelian group~$A$ we denote by~$A_\tor \subset A$ the subgroup of torsion elements in~$A$.

\begin{lemma}
\label{lem:brauer}
If~$X$ is a smooth complex Fano threefold then~$\rH^3(X,\ZZ)_\tor = 0$ and~$\rH_2(X,\ZZ)_\tor = 0$.
\end{lemma}

\begin{proof}
By the universal coefficient theorem it is enough to prove the second vanishing.
We use the classification of Fano threefolds and two simple observations.

First, using an appropriate version of the Lefschetz hyperplane theorem we deduce~$\rH_2(X,\ZZ)_\tor = 0$ 
if~$X$ is a smooth complete intersection of Cartier divisors in a weighted projective space (see~\cite[Proposition~6]{Dimca-wps}),
or a complete intersection of ample divisors 
in a smooth rational variety (see, e.g., \cite[Lemma~3.2(c)]{DK19}),
or a double covering of a smooth rational variety with smooth ample branch divisor,
(see, e.g., \cite[Lemma~3.3(c)]{DK19}).

Moreover, since~$\rH_2(X,\ZZ)_\tor$ is a birational invariant, 
we conclude that~$\rH_2(X,\ZZ)_\tor = 0$ for all~$X$ birational to threefolds as above;
in particular, this holds if~$X$ is rational.

Finally, looking at the classification of smooth Fano threefolds (see~\cite{MM} or~\cite{Fanography}),
we see that the above observations cover all the possibilities.
\end{proof}

\begin{remark}
Lemma~\ref{lem:brauer} also proves that the Brauer group~$\Br(X)$ is zero for any Fano threefold.
\end{remark}

The proof of the following result, answering in the case of Fano threefolds the question posed after~\cite[Proposition~6.13]{KS22:abs},
uses an idea suggested to us by Claire Voisin.

\begin{proposition}
\label{prop:mnf}
If~$X$ is a nodal $\QQ$-maximally nonfactorial complex Fano threefold then~$X$ is maximally nonfactorial.
\end{proposition}

\begin{proof}
Let~$\varpi\colon \hX \to X$ be a small resolution (recall that it exists by~\cite[Proposition~6.13]{KS22:abs}).
Since~$X$ is $\QQ$-maximally nonfactorial, the classes~$[C_1], \dots, [C_m] \in \rH_2(\hX, \Z)$ of the exceptional curves
are linearly independent by Lemma~\ref{lem:Qmnf}.
On the other hand, by Theorem~\ref{thm:smoothing} there exists a smoothing~$\cX$ of~$X$ over a smooth pointed curve~$(B,o)$.
Using~\cite[Lemma~8.1]{Friedman-CY} for~$b \ne o$ in~$B$ we obtain
\begin{equation*}
\rH_2(\cX_b, \Z) = 
\rH_2(\hX, \Z) / 
\langle
[C_1], \dots, [C_m]
\rangle.
\end{equation*}
Since $\cX_b$ is a smooth Fano threefold, Lemma~\ref{lem:brauer} proves that~$\rH_2(\cX_b, \Z)$ is torsion free.
Thus, the subgroup in~$\rH_2(\hX, \Z)$ generated by the classes~$[C_1], \cdots, [C_m]$ is saturated, 
hence~$X$ is maximally nonfactorial by Lemma~\ref{lem:Qmnf}.
\end{proof}

\subsection{The family of intermediate Jacobians}

In this section we discuss the Hodge theory of nodal complex Fano threefolds;
in particular, we give a Hodge-theoretic proof of Corollary~\ref{cor:intro-jacobians}.
We will need the following standard topological computation for topological one-parameter degenerations
see, e.g., \cite[\S C.2.2]{Peters-Steenbrink}.

\begin{lemma}
\label{lem:monodromy}
Let~$f\colon \cX \to (\Delta,o)$ be a smoothing of a proper $1$-nodal maximally nonfactorial threefold~$X = \cX_o$ over a pointed complex disc.
For any point~$b \ne o$ in~$\Delta$ the monodromy action on~$\rH^3(\cX_b, \Z)$ is trivial,
and we have an isomorphism of abelian groups
\begin{equation}
\label{eq:iso-rt}
\rH^3(X, \Z) \cong
\rH^3(\cX, \Z)  \cong
\rH^3(\cX_b, \Z).
\end{equation}
\end{lemma}
\begin{proof}
Since~$\cX$ is smooth, there is a continuous retraction map~$r \colon \cX \to X$.
Obviously, it induces an isomorphism of cohomology groups~$r^* \colon \rH^3(X, \Z) \xrightiso{\ \ } \rH^3(\cX, \Z)$.

Choose a point~$b \ne o$ in~$\Delta$, 
consider the map~$r_b \coloneqq r\vert_{\cX_b} \colon \cX_b \to X$,
and denote by~$\rS^3_b \coloneqq r_b^{-1}(x_0) \subset \cX_b$ the vanishing 3-sphere
and by~$i_b \colon \rS^3_b \hookrightarrow \cX_b$ its embedding.
The cohomology sequence of a pair (see, e.g., \cite[(C-10)]{Peters-Steenbrink}) 
gives an exact sequence
\begin{equation*}
0 \to
\rH^3(X, \ZZ) \xrightarrow{\ r_b^*\ }
\rH^3(\cX_b, \ZZ) \xrightarrow{\ i_b^*\ }
\rH^3(\rS^3_b, \ZZ) \xrightarrow{\quad}
\rH^4(X, \ZZ) \xrightarrow{\ r_b^*\ }
\rH^4(\cX_b, \ZZ) \to
0.
\end{equation*}
It also follows that~$\rb_2(X) = \rb_2(\cX_b)$.
Combining this with~\eqref{eq:b2-X} and Poincar\'e duality on~$\cX_b$, we obtain
\begin{equation*}
\rb_4(X) = 
\rb_2(X) + 1 =
\rb_2(\cX_b) + 1 = 
\rb_4(\cX_b) + 1.
\end{equation*}
We conclude that the kernel of the morphism~$r_b^* \colon \rH^4(X, \ZZ) \to \rH^4(\cX_b, \ZZ)$ has rank~$1$, 
hence the morphism~$\rH^3(\rS^3_b, \ZZ) \to \rH^4(X, \ZZ)$ must be injective, 
and hence the morphism~$i_b^* \colon \rH^3(\cX_b, \ZZ) \to \rH^3(\rS^3_b, \ZZ)$ must be zero.
This means that the vanishing sphere~$\rS^3_b$ is homologically trivial,
hence the natural monodromy action on~$\rH^3(\cX_b, \ZZ)$ is trivial as well,
and the map~$r_b^* \colon \rH^3(X, \Z) \to \rH^3(\cX_b, \Z)$ is an isomorphism.
\end{proof}

Now we can state the main result of this subsection.
We denote by~$\Jac(X)$ the intermediate Jacobian of a smooth projective rationally connected threefold~$X$;
this is a principally polarized abelian variety.

\begin{proposition}
\label{prop:jacobian-extension}
Let~$f \colon \cX \to B$ be a smoothing of a $1$-nodal maximally nonfactorial Fano threefold~$X$.
There is a smooth and proper family~$\cJ \to B$ 
of principally polarized abelian varieties such that
\begin{equation*}
\cJ_b \cong 
\begin{cases}
\Jac(\cX_b), & \text{if~$b \ne o$,}\\
\Jac(\hX), & \text{if~$b = o$,}
\end{cases}
\end{equation*}
where~$\hX$ is a small resolution of~$X$.
Moreover, these isomorphisms are compatible with the principal polarizations.
\end{proposition}

We note that a general approach for relative Jacobians of degenerations has been developed in~\cite{Zucker}, 
but in our simple situation of trivial monodromy we do not need that formalism.

\begin{proof}
By Lemma~\ref{lem:monodromy} the analytic sheaf~$\cH^3_{\cX/B} \coloneqq \bR^3 f_* \ZZ_\cX$ on~$B$ is locally constant
and we have a base change isomorphism
\begin{equation}
\label{eq:cho-hx}
(\cH^3_{\cX/B})_o \xrightiso{} \rH^3(X, \ZZ).
\end{equation}
Moreover, away from the origin~$o \in B$ the local system~$\cH^3_{\cX/B}$ is endowed with a polarized variation of pure Hodge structures,
and since the monodromy around the origin~$o$ acts trivially on~$\rH^3(\cX_b, \Z)$ (again by Lemma~\ref{lem:monodromy}) 
the Hodge filtration of~$(\cH^3_{\cX/B})_b$ has a limit when~$b$ approaches~$o$, see~\cite[(6.15)]{Schmid},
and the limiting filtration endows~$(\cH^3_{\cX/B})_o$ with a rational mixed Hodge structure.

Consider a semistable reduction~$\sigma \colon \tcX \to \cX$ of the family~$\cX/B$
and the composition of maps
\begin{equation*}
\rH^3(X, \QQ) = \rH^3(\cX_o, \QQ) \xrightarrow{\ \sigma_o^*\ }
\rH^3(\tcX_o, \QQ) \xrightarrow{\ \spe\ } 
(\cH^3_{\cX/B})_o \otimes \QQ \xrightiso{\quad}
\rH^3(X, \QQ),
\end{equation*}
where~$\spe$ is the specialization map, see~\cite[\S11.3.1]{Peters-Steenbrink}, 
and the last arrow is induced by the map~\eqref{eq:cho-hx}.
Unwinding the definitions, we see that their composition is an isomorphism of ~$\QQ$-vector spaces,
hence~$\spe \circ \sigma_o^*$ is also an isomorphism.
On the other hand, both~$\sigma_o^*$ and~$\spe$ are compatible with the Hodge filtrations;
this is obvious for~$\sigma_o^*$, and for~$\spe$ this is~\cite[Theorem~11.29]{Peters-Steenbrink}.
Therefore, the limiting mixed Hodge structure is isomorphic to the Hodge structure of~$\rH^3(X, \QQ)$;
in particular, it is pure and has a natural integral structure.

Now we see that the local system~$\cH^3_{\cX/B}$ carries a polarized variation of pure integral Hodge structures,
so we can define~$\cJ \to B$ as the corresponding family of principally polarized abelian varieties.
The isomorphism~$\cJ_b \cong \Jac(\cX_b)$ follows from the definition of~$\cH^3_{\cX/B}$,
and the isomorphism~$\cJ_o \cong \Jac(\hX)$ follows from the isomorphism of the limiting Hodge structure
and the Hodge structure of~$\rH^3(X, \Z)$ proved above combined with the isomorphism of Lemma~\ref{lem:small-coh}.
\end{proof}

Now we can use this to prove Corollary~\ref{cor:intro-jacobians} from the Introduction.

\begin{proof}[Proof of Corollary~\textup{\ref{cor:intro-jacobians}}]
Consider the smoothing~$\cX \to B$ constructed in Theorem~\ref{thm:intro-simple}.
Its central fiber~$X$ is maximally nonfactorial (either by Remark~\ref{rem:nonfact}, or by Proposition~\ref{prop:mnf}),
hence Proposition~\ref{prop:jacobian-extension} provides a smooth and proper family~$\cJ \to B$ of abelian varieties
such that~$\cJ_b \cong \Jac(\cX_b)$ for~$b \ne o$.
Moreover, using the small resolution~$\pi \colon \tY \to X$ from Proposition~\ref{prop:x-y}
we obtain an isomorphism~$\cJ_o \cong \Jac(\tY)$.
Finally, if~$d \ge 2$ the morphism~$\sigma \colon \tY \to Y$ is the blowup of a smooth rational curve, hence~$\Jac(\tY) \cong \Jac(Y)$.
\end{proof}

\newcommand{\etalchar}[1]{$^{#1}$}

\end{document}